\documentclass[11pt]{amsart}

\usepackage{macros}

\usepackage[
  backend=biber,
  style=alphabetic,
  maxbibnames=99,
  maxcitenames=3,
  giveninits=true
]{biblatex}

\title[]{Higher-dimensional chiral algebras in the Jouanolou model and free-field realization}%
\author{Zhengping Gui, Minghao Wang, Brian R. Williams}%
\address{ZG: Center for Mathematics and Interdisciplinary Sciences, Fudan University, Shanghai
200433, China;
Shanghai Institute for Mathematics and Interdisciplinary Sciences (SIMIS), Shanghai
200433, China}
\address{MW, BW: Department of Mathematics, Boston University, United States}%
\email{ zgui@fudan.edu.cn, minghaow@bu.edu, bwill22@bu.edu}%

\thanks{}%
\begin{document}

\begin{abstract}
 We appeal to the theory of Jouanolou torsors to model the coherent cohomology of configuration spaces of points in
affine space~$\mathbb{A}^d$.
Using this model, we develop the operadic notion of chiral operations, thus generalizing the notion of \textit{chiral
algebras} of \cite{BD} to higher dimensions.
To produce examples, we use a higher-dimensional conceptualization of the residue which is inspired by Feynman
graph integrals.
One of our main results is the realization, using higher chiral operations, of the higher-dimensional Kac--Moody and
Virasoro algebras.
\end{abstract}
\maketitle
\setcounter{tocdepth}{1}
\tableofcontents


\section*{Introduction}

The concept of a \textit{chiral algebra} was introduced by Beilinson and Drinfeld as a geometrization of the
so-called \textit{operator product expansion} which describes the algebraic structure of local operators in (real) two-dimensional chiral
conformal field theory (CFT) \cite{BD}.
The primary goal of this paper is to introduce a model for the algebraic structure behind the (derived) operator product
expansion in higher dimensional \textit{holomorphic} quantum field theory.

For a chiral algebra on an algebraic variety $X$, the chiral operations are parametrized by configuration spaces of
points on $X$.
For example, on $\mathbb{A}^1$ all such configuration spaces are affine, so algebraically it suffices to pass to the
algebra of holomorphic functions on configuration spaces.
On the other hand, for the configuration spaces of points in $\mathbb{A}^d, d > 1$ are no longer affine so to formulate a higher-dimensional theory of
chiral operations it is essential that we instead pass to the derived global sections.

In complex geometry, there is a strict commutative dg algebra which models derived global sections given by the Dolbeault
complex.
What we require is an analog of the Dolbeault complex that can be used in the algebraic setting.
Following the ideas of \cite{FHK}, we use the so-called Jouanolou construction to accomplish this.
We first find an affine variety, affinely fibered over the configuration space $\mathrm{Conf}_n(\mathbb{A}^d)$.
The Jouanolou trick is to then consider the relative de Rham complex of this fibration as the explicit derived model
for the configuration space.
When the number of points is $n=2$ we have $\mathrm{Conf}_2(\mathbb{A}^d) = (\mathbb{A}^d - \{0\})\times \mathbb{A}^d$ and this model is used in \cite{FHK} to give a definition of the dg Lie algebra of
derived currents valued in a Lie algebra (so, it plays the role of the Laurent series when $d=1$).

Chiral algebras are defined within the context of $\mathcal{D}$-modules.
For $\mathcal{D}$-module $\mathcal{A}$ defined on an algebraic curve $X$, the space of $n$-ary chiral operations, following \cite{BD}, is
\begin{equation}\label{}
  \mathrm{Hom}_{\mathcal{D}_{X^{n}}}\left(j_{*}j^*\mathcal{A}^{\boxtimes{n}},\Delta_* \mathcal{A}\right)
\end{equation}
where $j \colon \mathrm{Conf}_n(X) \to X \times \cdots \times X$ is the open immersion and $\Delta \colon X \to X \times
\cdots \times X$ is the diagonal embedding.
Combining chiral operations of all arities results in an operad, and a chiral algebra is simply a Lie algebra object in
this operad.
In generalizing this to $X = \mathbb{A}^d$ we use the Jouanolou trick to provide an explicit model for the $\mathcal{D}
$-module derived pushforward $\mathbf{R} j_* j^* \mathcal{A}^{\boxtimes n}$.

  Let $\mathcal{A}$ be a $\mathcal{D}$-module on $\mathbb{A}^d$.
  Then, for any finite set $I$, the Jouanolou model $\mathbf{J}_{\mathbb{A}^d}^I[\mathcal{A}]$ (see definition
  \ref{dfn:jmodel}) is a model for the $\mathcal{D}$-module derived pushforward $\mathbf{R} j_* j^* \mathcal{A}
  ^{\boxtimes I}$.
  Furthermore, as the set $I$~runs over finite sets, the collection
  \begin{equation}\label{}
    \left\{\op{Hom}_{\mathcal{D}_{(\mathbb{A}^{d})^I}} \left(\mathbf{J}^I_{\mathbb{A}^d} [\mathcal{A}] \; ,
    \; \Delta_*^I \mathcal{A} \right) \right\}_{I}
  \end{equation}
  forms a dg operad.
  Denote this dg operad by $\mathcal{P}_d$.
  Notice that by construction this operad is $\mathrm{GL}_d$-equivariant.
  This is the main advantage of the Jouanolou model as compared to the \v{C}ech--Thom--Sullivan cosimplicial model for configuration spaces used in \cite{FGY}.

  \begin{defn*}
  A $d$-dimensional \textit{\textbf{homotopy chiral algebra}}, in the Jouanolou
model, is a map of dg operads
  \begin{equation}\label{}
  \mathcal{L} \mathrm{ie}_\infty \to \mathcal{P}_d [\mathcal{A}]
  \end{equation}
  where $\mathcal{L}\mathrm{ie}_\infty$ is the dg operad of $L_\infty$ algebras.
\end{defn*}

Homotopy chiral algebras in the Jouanolou model define chiral algebras in the sense of \cite{FG}, but we make no
assertion that the resulting $\infty$-category of Jouanolou chiral algebras is equivalent to the $\infty$-categorical
description of \cite{FG}.

The next main goal of this article is to provide explicit descriptions and examples of chiral algebras in higher dimensions.
Firstly, we describe, in explicit detail, the chiral operations underlying the \textit{unit} chiral algebra on $\mathbb{A}^d$.
As is familiar from the theory of vertex algebras, the Cauchy residue encodes the chiral operation of the unit chiral algebra
when $d=1$.
We generalize the residue to higher-dimensional configuration spaces in the Jouanolou model.
The key to our definition of the residue is the use of Feynman graph integrals in higher-dimensional holomorphic quantum field theory.
The analysis which underlies the convergence of such graph integrals uses the compactification of Schwinger spaces
following \cite{wang2024feynman}.
A peculiar fact is that in dimensions $d > 1$ the unit chiral algebra (in the Jouanolou model) has nontrivial chiral
operations of arity $k$ for all $k = 1,2,3,\ldots$.

Using the higher-dimensional residue, we formulate a higher-dimensional version of Wick's theorem for
chiral algebras in section \ref{s:example}.
This leads to explicit descriptions of so-called \textit{free} chiral algebras (those which come from free holomorphic
quantum field theories); they are characterized, in part, by the property that their chiral operations are quantizations
of the chiralization of the shifted symplectic space $T^*[d-1] V$ where $V$ is a (possibly super) vector space.
The associated $d$-dimensional chiral algebra will be denoted $\cF_V$. 
It is a higher-dimensional analog of the free ghost system chiral algebras used in string theory.
Indeed, when $d=1$ this is the familiar $\beta-\gamma$ or $b-c$ system depending on the parity of $V$.

As an application of our formulation of free chiral algebras, we provide a higher-dimensional free-field realization of the higher-dimensional Kac--Moody algebra
of \cite{FHK} and the higher-dimensional Virasoro algebra of \cite{GWvir}.
We briefly recall these algebras.
For a Lie algebra~$\lie{g}$, higher-dimensional Kac--Moody algebras are dg Lie algebra central extensions of $\lie{g}$-valued
derived global sections of the structure sheaf on punctured affine space $\AA^d - \{0\}$.
Such central extensions, which are additionally translation and $GL_d$-invariant, are in bijective correspondence with
$\theta \in \op{Sym}^{d+1}(\lie{g}^*)^{G}$--that is, invariant polynomials on $\lie{g}$ of degree $d+1$.
The higher Virasoro algebras, on the other hand, are central extensions of the dg Lie algebra of derived sections of the tangent sheaf
on punctured affine space (the higher-dimensional Witt algebra).
Such central extensions correspond to degree $2d+2$ universal characteristic classes $H^{2d+2}
(BGL_d)$.
Explicit $L_\infty$ models for the higher Virasoro algebras appear in \cite{GWvir}.

We study realizations of these algebras on chiral algebras of free field type.
Such realizations have been found in terms of factorization algebras (in the style of \cite{CG1}) in \cite{GWkm,BWac}. 
Our approach here is completely within the theory of higher-dimensional chiral algebras, which we briefly set up and refer the reader to section \ref{s:example} for more details.
Denote the de Rham complex of a right $\mathcal{D}$-module $\cA$ on $X$ by $h_{X}(\mathcal{A})$.
If $\cA$ is a $d$-dimensional homotopy chiral algebra then the chiral operations endow $h_{\mathring{\AA}^d}(\cA)$ with the structure of an $L_\infty$-algebra.
For $\cA$ a chiral algebra of free field type, this $L_\infty$ algebra is a strict Lie algebra given explicitly by contractions in the spirit of Wick's theorem.
We show that to obtain a full ``quantum" realization on such free field chiral algebras, one must precisely use the central extensions of the previous paragraph.
We state here a general result which combines the free field realizations of Kac--Moody and Virasoro type.
In other words, we start with the Jouanolou model of the semi-direct product dg Lie algebra 
\begin{equation}\label{eq:semi}
  \R \Gamma(\AA^d - \{0\}, \mathcal{T}) \ltimes \R \Gamma(\AA^d - \{0\}, \lie{g} \otimes \mathcal{O}) 
\end{equation}
Central extensions of this algebra can be defined for any class in 
\[
    H^{2d+2}(BGL_d \times BG) .
\]

\begin{thm*}[See theorem \ref{FullCentral}]
\label{thm:}
Suppose that $V$ is a $\lie{g}$-representation and let $\cF_V$ be the free ghost chiral algebra on $\AA^2$ as defined in
Section \ref{s:example}.
Then, there is an $L_\infty$ homomorphism 
  \[
  \rho_V \colon (\lie{witt}_2^\bu \ltimes \lie{g}_2^\bu)^{\flat}_{\mathsf{c}} \to h_{\mathring{\AA}^2}(\cF_V) 
\]
where the domain is a model for the $L_\infty$ central extension of \eqref{eq:semi} corresponding to the class
\[
  \mathsf{c} = \left[\op{Todd}(\mathcal{T}) \op{Ch}(V)\right]_6 \in H^6(BGL_2 \times BG) .
\]
\end{thm*}
This result is specific to the $d=2$-dimensional case, but the generalization to $d>2$ will appear in upcoming work.

In the last part of the paper we focus on the case of complex dimension $d=2$.
We flip much of the logic of the bulk of this paper in reverse, by using the $L_\infty$ relations to explicitly
characterize Feynman graph integrals.
As we have stated, holomorphic quantum field theory provides our inspiration for the definition of the higher-dimensional residue in
the Jouanolou model.
Explicitly, our residues are given as sums over weights of Feynman graph integrals.
We show in section \ref{s:feynman} that these residues satisfy an appropriate $L_\infty$-relation, this is the essential
structure underlying the unit chiral algebra in higher dimensions.
In section \ref{s:A2}, we utilize the $L_\infty$-relations to formulate a \textit{recursive} formula for weights
of Feynman graph integrals in complex dimension two.
The formula holds for graphs of a special type, which we refer to as Laman type I'.\footnote{The term Laman graph is derived
from \cite{Laman1970OnGA}.}
We can state the recursive formula, heuristically, as follows.

\begin{thm*}[see theorem \ref{MainTheoremRecursiveRS} for the precise formulation]
  Suppose the complex dimension is $d=2$, and suppose that a chiral algebra has a two-ary operation expressed in terms
  of the residue
  as we have defined it.
Then, the chiral operation associated to an arbitrary graph $\Gamma$ of type Laman I' is completely fixed by the $L_\infty$
relations.
Namely, there is a recursive formula for the chiral operation associated to $\Gamma$ constructed using the residue.
\end{thm*}

As an application of this recursive formula, we explore a class of graph integrals when $d=2$.
We compute, in full generality and explicitly, the weight of Laman I' Feynman graphs at loop orders one, two, and three (the latter using
\texttt{Mathematica}).
At lowest order in derivatives, we find an exact agreement with the work of physicists in \cite{Gaiotto:2024gii}.
For example, in the case of the two-loop $\Theta$-diagram, in complex dimension $d=2$, we find the chiral operation
\begin{equation}\label{}
  \mu_4 (P_{30} P_{32} P_{10} P_{12}P_{20} d \mathbf{z}_{0,1,2,3}) = \frac{1}{24} \left(\lambda_1 \wedge (\lambda_3 + 2 \lambda_2)) \cdot
    (\lambda_3 \wedge (\lambda_1 + 2 \lambda_2) \right)\cdot d \mathbf{z}_0^{\blacklozenge} ,
\end{equation}
to lowest order in derivatives.
We introduce the notation in the main part of the text.
(for example,  $d \mathbf{z}^{\blacklozenge}$ denotes the holomorphic volume form shifted in cohomological degree, and
$P$ denotes the propagator).
Our computation is logically independent and more general than \cite{budzik2023feynman} in the sense that we also include arbitrary derivative insertions
at vertices. 

\subsection*{Conventions}
\begin{itemize}

  \item Let $\mathbf{k}$ be a field with $\mathrm{char}(\mathbf{k})=0$.  
        For an affine variety $X=\mathrm{Spec}(\mathbf{k}[X])$ over $\mathbf{k}$ and a quasi-coherent sheaf $M$ on $X$, we will not distinguish $M$ and $\Gamma(X,M)$ the space of global sections. We will mainly work over the complex numbers $\mathbf{k} = \CC$, but many definitions in this paper work for a general field $\mathbf{k}.$

\item We will denote by $\mathbf{fSet}$ the category of non-empty finite sets with surjective morphisms.
\item For ${I}=\{1,\dots,k\}\in {\mathbf{fSet}}$ and a variety $X$, we write $X^{{I}}$ for $X_1\times \cdots \times X_k$ where $X_i$ is a copy of $X$. Suppose that $M$ is a quasi-coherent sheaf over $X$, we will denote $M^{\boxtimes{I}}$ the exterior product $M_1\boxtimes \cdots \boxtimes M_k$ over $X^{{I}}$ where $M_i$ over $X_i=X$ is the sheaf $M$.
  \item A graded vector space $V$ is a direct sum of vector spaces $V=\mathop{\bigoplus}\limits_{i\in\mathbb{Z}}V_i$ and we say that $V_{i}$ is in degree $i$.
        We will use the $k$-th shift notation $V[k],k\in \mathbb{Z}$ as well as the suspension notation $s^{-k}V,k\in\mathbb{Z}$
$$
(V[k])_i=V_{i+k}.$$
We also use $|a|$ for the degree of a homogeneous element $a\in V.$
\item Given a manifold $M$, we use $\mathcal{A}^{\bu}(M)$ to denote the space
        of complex-valued differential forms on $M$.
        If $M$ is a
    complex manifold, to emphasize their bi-graded nature, we will use
    $\mathcal{A}^{\bu,\bu} (M)$ instead.
\item We use the following Koszul sign rules for integrals: If $M, N$ are manifolds whose dimensions are $m, n$ respectively. Let $\alpha \in \mathcal{A}^{i}(M),     \beta \in \mathcal{A}^{j}(N)$, then
$$
\int_M \int_N \alpha \wedge\beta = (- 1)^{n i} \int_M \alpha \int_N \beta .
$$
\end{itemize}
\subsection*{Acknowledgements} The authors would like to thank Si Li, Kai Wang and Keyou Zeng for helpful communications and discussions. Some of the results in this paper were announced at the 4th International Conference on Operad Theory and Related Topics (Tianjin, November 2024). Z.G. would like to thank the organizers and participants of this conference.
\section{The Jouanolou model and a definition of higher chiral algebra}\label{s:Jouanolou}

\subsection{The Jouanolou model for configuration space}
We follow \cite[Section~4.1.3, p.~279]{BD}. Suppose
\(\pi:Y\rightarrow X\) is a Jouanolou map; that is, \(Y\) is a torsor over
\(X\) for the action of a vector bundle, and \(Y\) is an affine scheme. Then,
for any quasi-coherent sheaf \(\mathcal F\) on \(X\), the global relative de
Rham complex
\[
  \Gamma\left(Y,\Omega^\bu_{Y/X}\otimes \pi^*\mathcal F\right)
\]
is a model for \(\R\Gamma(X,\mathcal F)\). Indeed, we have quasi-isomorphisms
\[
\Gamma\left(Y,\Omega^\bu_{Y/X}\otimes \pi^*\mathcal F\right)
\simeq
\R\Gamma\left(Y,\Omega^\bu_{Y/X}\otimes \pi^*\mathcal F\right)
\simeq
\R\Gamma\left(X,\R\pi_*(\Omega^\bu_{Y/X}\otimes\pi^*\mathcal F)\right).
\]
Because \(\pi\) is a Zariski locally trivial fibration with fiber
\(\AA^N\), the Poincar\'e lemma for differential forms on \(\AA^N\) implies
that the inclusions
\[
\mathcal F
\hookrightarrow
\pi_*(\Omega^\bu_{Y/X}\otimes\pi^*\mathcal F)
\hookrightarrow
\R\pi_*(\Omega^\bu_{Y/X}\otimes\pi^*\mathcal F)
\]
are quasi-isomorphisms.

For the finite set \(I=\{1,\ldots,k\}\), we define a Jouanolou torsor over
\(\op{Conf}_{I}(\AA^d)\) by
\[
\mathbb J^I_{\AA^d}
\define
\left\{
\sum_{\sss=1}^d
x^{\sss}_{ij}\left(z^{\sss}_i-z^{\sss}_j\right)=1
\;\middle|\;
1\leq i<j\leq k
\right\}
\subset
(\AA^d)^k\times (\AA^d)^{\binom{k}{2}},
\]
with the Jouanolou map
\[
  p_I:\mathbb J^I_{\AA^d}\longrightarrow \op{Conf}_{I}(\AA^d)
\]
given by the evident projection. The projection \(p_I\) is affine with fiber
isomorphic to \(\AA^{(d-1)\binom{k}{2}}\). It carries an action of the vector
bundle
\(\mathbb V^I_{\AA^d}\rightarrow \op{Conf}_{I}(\AA^d)\) defined by
\[
\mathbb V^I_{\AA^d}
\define
\left\{
\sum_{\sss=1}^d
v^{\sss}_{ij}\left(z^{\sss}_i-z^{\sss}_j\right)=0
\;\middle|\;
1\leq i<j\leq k
\right\}
\subset
(\AA^d)^k\times (\AA^d)^{\binom{k}{2}}.
\]

\begin{lem}
We have
\[
  \Omega^\bullet_{\mathbb J^I_{\AA^d}/\op{Conf}_{I}(\AA^d)}
  =
  \Omega^\bullet_{\mathbb J^I_{\AA^d}/(\AA^d)^I}.
\]
\end{lem}

\begin{proof}
This follows from a standard fact in algebraic geometry. Since
\(j_I:\op{Conf}_{I}(\AA^d)\hookrightarrow (\AA^d)^I\) is an open immersion,
\[
  \Omega^1_{\op{Conf}_{I}(\AA^d)/(\AA^d)^I}=0.
\]
The sequence of maps
\[
\mathbb J^I_{\AA^d}
\xrightarrow{p_I}
\op{Conf}_{I}(\AA^d)
\xrightarrow{j_I}
(\AA^d)^I
\]
induces an exact sequence
\[
  p_I^*\Omega^1_{\op{Conf}_{I}(\AA^d)/(\AA^d)^I}
  \longrightarrow
  \Omega^1_{\mathbb J^I_{\AA^d}/(\AA^d)^I}
  \longrightarrow
  \Omega^1_{\mathbb J^I_{\AA^d}/\op{Conf}_{I}(\AA^d)}
  \longrightarrow 0.
\]
The first term vanishes, so the lemma follows.
\end{proof}

We define the \textit{Jouanolou model} of the structure sheaf over
\(\op{Conf}_{I}(\AA^d)\) to be
\[
  \mathbf J^I_{\AA^d}
  \define
  \Gamma\left(
  \mathbb J^I_{\AA^d},
  \Omega^\bu_{\mathbb J^I_{\AA^d}/\op{Conf}_{I}(\AA^d)}
  \right).
\]
By the lemma, this model has the following explicit description:
\[
\mathbf J^I_{\AA^d}
=
\frac{
\kk\!\bigl[z^{\sss}_i,x^{\sss}_{jl},\mathbf d x^{\sss}_{jl}\bigr]
_{
\substack{\sss=1,\ldots,d\\ 1\leq i\leq k,\ 1\leq j<l\leq k}}
}{
\left\langle
\sum\limits_{\sss=1}^d x^{\sss}_{jl}(z^{\sss}_j-z^{\sss}_l)-1,
\quad
\sum\limits_{\sss=1}^d \mathbf d x^{\sss}_{jl}(z^{\sss}_j-z^{\sss}_l)
\;\middle|\;
1\leq j<l\leq k
\right\rangle
}.
\]
Here \(\mathbf d\) is the relative differential; it satisfies
\[
  \mathbf d(x^{\sss}_{jl})=\mathbf d x^{\sss}_{jl},
  \qquad
  \mathbf d z^{\sss}_i=0.
\]
We also use the variables \(x^{\sss}_{jl}\) and \(\mathbf d x^{\sss}_{jl}\) for
\(j>l\), with the conventions
\[
  x^{\sss}_{jl}=-x^{\sss}_{lj},
  \qquad
  \mathbf d x^{\sss}_{jl}=-\mathbf d x^{\sss}_{lj}.
\]

Given \(I\in\mathbf{fSet}\), let \(\mathcal D_{(\AA^d)^I}\) denote the sheaf of
algebraic differential operators on \((\AA^d)^I\). Its algebra of global
sections is generated by \(z^{\sss}_i\) and \(\partial_{z^{\scriptstyle\sss}_i}\), subject
to the usual relations
\[
  [\partial_{z^{\scriptstyle\sss}_i},z^{\ttt}_j]
  =
  \delta_{ij}\delta_{\sss\ttt}.
\]
We will mainly deal with quasi-coherent sheaves, including \(\mathcal D\)-modules,
on \((\AA^d)^k\). To simplify notation, we will not distinguish a quasi-coherent
sheaf from its space of global sections.

\begin{prop}\label{prop:dmodulestr}
The Jouanolou model \(\mathbf J^I_{\AA^d}\) has a left
\(\mathcal D_{(\AA^d)^I}\)-module structure. The functions \(z_i^{\sss}\) act by
multiplication, and the vector fields are determined by
\[
  [\partial_{z_i^{\scriptstyle\ttt}},\mathbf d]=0,
  \qquad
  \partial_{z_i^{\scriptstyle\ttt}}x^{\sss}_{jl}
  =
  \begin{cases}
  -x^{\ttt}_{jl}x^{\sss}_{jl}, & i=j,\\
  \phantom{-}x^{\ttt}_{jl}x^{\sss}_{jl}, & i=l,\\
  0, & i\notin\{j,l\},
  \end{cases}
\]
for \(1\leq \sss,\ttt\leq d\) and \(1\leq j<l\leq k\).
\end{prop}

\begin{proof}
Let
\[
  R_{jl}\define
  \sum_{\sss=1}^d x^{\sss}_{jl}(z^{\sss}_j-z^{\sss}_l)-1.
\]
It is enough to check that the vector fields preserve the ideal generated by
\(R_{jl}\) and \(\mathbf dR_{jl}\). If \(i=j\), then
\[
\partial_{z_j^{\scriptstyle\ttt}}R_{jl}
=
-x^{\ttt}_{jl}\sum_{\sss=1}^d x^{\sss}_{jl}(z^{\sss}_j-z^{\sss}_l)
+x^{\ttt}_{jl}
=
-x^{\ttt}_{jl}+x^{\ttt}_{jl}=0.
\]
If \(i=l\), the same computation gives
\[
\partial_{z_l^{\scriptstyle\ttt}}R_{jl}
=
 x^{\ttt}_{jl}\sum_{\sss=1}^d x^{\sss}_{jl}(z^{\sss}_j-z^{\sss}_l)
-x^{\ttt}_{jl}
=
0.
\]
If \(i\notin\{j,l\}\), the derivative is zero. Finally,
\([\partial_{z_i^{\scriptstyle\ttt}},\mathbf d]=0\) implies
\(\partial_{z_i^{\scriptstyle\ttt}}\mathbf dR_{jl}=\mathbf d(\partial_{z_i^{\scriptstyle\ttt}}R_{jl})=0\).
Thus the action preserves the defining ideal.
\end{proof}

\begin{rem}
The intuition behind this definition is that one can use the coordinate transformation
\[
  x^{\sss}_{kl}
  =
  \frac{u^{\sss}_{kl}}
  {\sum\limits_{\ttt=1}^d u^{\ttt}_{kl}(z^{\ttt}_k-z^{\ttt}_l)}.
\]
\end{rem}

\begin{defn}\label{dfn:jmodel}
Given a finite set \(I\) and an \(\mathcal O_{(\AA^d)^I}\)-module \(M_I\), the
\textit{Jouanolou model of \(M_I\)} is the dg
\(\mathcal O_{(\AA^d)^I}\)-module
\[
\begin{aligned}
\mathbf J^I_{\AA^d}(M_I)
&\define
\Gamma\left(
\mathbb J^I_{\AA^d},
\Omega^\bullet_{\mathbb J^I_{\AA^d}/\op{Conf}_{I}(\AA^d)}
\otimes p_I^*j_I^*M_I
\right)                                                     \\
&=
M_I\otimes_{\mathcal O_{(\AA^d)^I}}\mathbf J^I_{\AA^d}.
\end{aligned}
\]
If \(M_I\) is a right, respectively left, \(\mathcal D_{(\AA^d)^I}\)-module, then
\(\mathbf J^I_{\AA^d}(M_I)\) is a right, respectively left,
\(\mathcal D_{(\AA^d)^I}\)-module.
\end{defn}

\begin{rem}
One can show that \(\mathbf J^I_{\AA^d}(M_I)\) represents
\(\R j_{I*}j_I^*M_I\) in the derived category of \(\mathcal D\)-modules. In fact,
there is a universal Jouanolou model \cite[4.1.3(ii)]{BD}, defined as the jet
algebra
\[
  \mathcal J^I_{\AA^d}
  \define
  \mathscr J\mathbf J^I_{\AA^d}
  =
  \Omega^\bullet_{\mathscr J\mathbf J^{I,0}_{\AA^d}/\mathcal O_{(\AA^d)^I}}.
\]
In terms of this universal model, our Jouanolou model is obtained by imposing the
relations from Proposition~\ref{prop:dmodulestr}:
\[
\mathbf J^I_{\AA^d}
=
\frac{\mathcal J^I_{\AA^d}}
{\left\langle
\partial_{z_i^{\scriptstyle\ttt}}x^{\sss}_{ij}=-x^{\ttt}_{ij}x^{\sss}_{ij},
\quad
\partial_{z_j^{\scriptstyle\ttt}}x^{\sss}_{ij}=x^{\ttt}_{ij}x^{\sss}_{ij}
\;\middle|\;
 i<j,\ 1\leq\sss,\ttt\leq d
\right\rangle}.
\]
The universal Jouanolou model is also closely related to another model for the
configuration space,
\[
  \mathbf P^I_{\AA^d}
  =
  \frac{\mathcal J^I_{\AA^d}}
  {\left\langle
  (z_i^{\sss}-z_j^{\sss})(\partial_{z_i^{\scriptstyle\sss}}x^{\sss}_{ij})=x^{\sss}_{ij}
  \;\middle|\;
  i<j,\ 1\leq\sss\leq d
  \right\rangle},
\]
called the \textit{polysimplicial model}, which is considered in \cite{FGY}.
The quotient maps
\[
\mathbf J^I_{\AA^d}
\xleftarrow{\sim}
\mathcal J^I_{\AA^d}
\xrightarrow{\sim}
\mathbf P^I_{\AA^d}
\]
are quasi-isomorphisms as \(\mathcal D\)-modules. Furthermore, one can show that
these three models are flat \(\mathcal O_{(\AA^d)^I}\)-algebras. Hence, for a
\(\mathcal D\)-module \(M_I\),
\[
\mathbf J^I_{\AA^d}(M_I)
\xleftarrow{\sim}
\mathcal J^I_{\AA^d}(M_I)
\xrightarrow{\sim}
\mathbf P^I_{\AA^d}(M_I).
\]
The claim that \(\mathbf J^I_{\AA^d}(M_I)\simeq \R j_{I*}j_I^*M_I\) follows from
\(\mathbf P^I_{\AA^d}(M_I)\simeq \R j_{I*}j_I^*M_I\); see \cite{FGY}.
\end{rem}

Our main motivation for the Jouanolou model is to formulate the concept of
\textit{chiral operations} in higher dimensions. In complex dimension one, chiral
operations form the operadic structure controlling \textit{chiral algebras}.
Indeed, in \cite{BD}, for a \(\mathcal D_X\)-module \(M\), the space of chiral
operations is
\[
\operatorname{Hom}_{\mathcal D_{X^I}}
\left(
  j_*j^*M^{\boxtimes I},
  \Delta^{I/\{\star\}}_*M
\right),
\qquad
\Delta^{I/\{\star\}}:X^{\{\star\}}\hookrightarrow X^I.
\]
For \(X=\AA^1\), chiral operations form an operad for any
\(\mathcal D_{\AA^1}\)-module \(M=\mathcal A\); we denote this operad by
\(\mathcal P^{ch}_{\AA^1}[\mathcal A]\). Following Beilinson and Drinfeld, a
\textit{chiral algebra} structure on \(\mathcal A\) is a map of operads
\[
  \mathcal L\!\op{ie}\longrightarrow \mathcal P^{ch}_{\AA^1}[\mathcal A].
\]
Thus the Jouanolou model \(\mathbf J^I_{\AA^d}(M_I)\) can be thought of as the
domain of a higher-dimensional chiral operation. We now spell out the definition
of a higher-dimensional chiral algebra.

\subsection{The definition of higher chiral algebras}

Let \(I\in \mathbf{fSet}\), and let \(M\) be a right
\(\mathcal D_{\AA^d}\)-module. The space of global sections of the diagonal
\(\mathcal D\)-module pushforward of \(M\) is
\[
\Gamma\left((\AA^d)^I,\Delta^{I/\{\star\}}_*M\right)
=
M\otimes_{\kk[\lambda_\star]}\kk[\lambda_I],
\]
where
\(\Delta^{I/\{\star\}}:(\AA^d)^{\{\star\}}=\AA^d\hookrightarrow(\AA^d)^I\)
is the diagonal embedding and
\[
\kk[\lambda_\star]
\define
\kk[\lambda_\star^{\sss}\mid \sss=1,\ldots,d],
\qquad
\kk[\lambda_I]
\define
\kk[\lambda_i^{\sss}\mid i\in I,\ \sss=1,\ldots,d].
\]
The tensor notation \(-\otimes_{\kk[\lambda_\star]}\kk[\lambda_I]\) indicates
that
\[
  \lambda_\star^{\sss}=\sum_{i\in I}\lambda_i^{\sss},
  \qquad
  \sss=1,\ldots,d,
\]
and that \(\lambda_\star^{\sss}\) acts on \(M\) on the right as
\(\partial_{z^{\sss}}\).

The right \(\mathcal D_{(\AA^d)^I}\)-module structure on
\(M\otimes_{\kk[\lambda_\star]}\kk[\lambda_I]\) is given by
\[
  (n\otimes F)\cdot z_k^{\ttt}
  =
  (n\cdot z^{\ttt})\otimes F
  +
  n\otimes \partial_{\lambda_k^{\scriptstyle\ttt}}F,
  \qquad
  (n\otimes F)\cdot \partial_{z_k^{\scriptstyle\ttt}}
  =
  n\otimes F\lambda_k^{\ttt},
\]
for \(F\in\kk[\lambda_I]\).

For a non-empty finite set \(I\), we define the space of \(d\)-dimensional
chiral \(I\)-operations in the Jouanolou model to be the dg vector space
\[
\mathcal P^{ch,\mathbf J}_{\AA^d}[M](I\rightarrow\{\star\})
\define
\operatorname{Hom}_{\mathcal D_{(\AA^d)^I}}
\left(
\mathbf J^I_{\AA^d}(M^{\boxtimes I}),
\Delta^{I/\{\star\}}_*M
\right).
\]
These vector spaces are naturally cochain complexes with actions of the symmetric
groups.

\begin{defn}\label{SymmetricGroupAction}
We define the action of the symmetric group \(S_{|I|}\) on
\(\mathcal P^{ch,\mathbf J}_{\AA^d}[M](I\rightarrow\{\star\})\) by
\[
  (\sigma\mu)(-)\define\sigma\left(\mu(\sigma^{-1}-)\right),
\]
where the two symmetric-group actions are as follows.
\begin{itemize}
\item The action on \(\mathbf J^I_{\AA^d}(M^{\boxtimes I})\) is given by
\[
\begin{aligned}
&\sigma\left(
\alpha(z_i^{\sss},x_{ij}^{\sss},\mathbf d x_{ij}^{\sss})
\cdot m_1\boxtimes\cdots\boxtimes m_k
\right)                                                     \\
&\quad \define
(-1)^{\chi(m_1,\ldots,m_k;\sigma)}
\alpha(z_{\sigma(i)}^{\sss},x_{\sigma(i)\sigma(j)}^{\sss},
\mathbf d x_{\sigma(i)\sigma(j)}^{\sss})
\cdot
m_{\sigma^{-1}(1)}\boxtimes\cdots\boxtimes m_{\sigma^{-1}(k)}.
\end{aligned}
\]
Here \(m_i\in M\), and \(\chi(m_1,\ldots,m_k;\sigma)\) is the Koszul sign.

\item The action on \(\Delta^{\{1,\ldots,k\}/\{\star\}}_*M\) is given by
\[
  \sigma\left(m\otimes f(\lambda_1^{\sss},\ldots,\lambda_k^{\sss})\right)
  \define
  m\otimes f(\lambda_{\sigma(1)}^{\sss},\ldots,
  \lambda_{\sigma(k)}^{\sss}).
\]
\end{itemize}
\end{defn}

\begin{thm}\label{DgOperadStructure}
For a right \(\mathcal D_{\AA^d}\)-module \(M\), the collection
\[
  \{\mathcal P^{ch,\mathbf J}_{\AA^d}[M](k)\}_{k\geq 0},
\]
where
\[
\mathcal P^{ch,\mathbf J}_{\AA^d}[M](k)
=
\begin{cases}
0, & k=0,\\[2mm]
\operatorname{Hom}_{\mathcal D_{(\AA^d)^k}}
\left(
\mathbf J^k_{\AA^d}(M^{\boxtimes k}),
\Delta^{\{1,\ldots,k\}/\{\star\}}_*M
\right), & k\geq 1,
\end{cases}
\]
has the structure of a dg operad.
\end{thm}

\begin{proof}
The construction of the operad structure is completely parallel to \cite{FGY}. We
only spell out how to compose two chiral operations.

Suppose that \(I=I'\bigsqcup I''\), and suppose we are given chiral operations
\[
\mu_{I'}\in
\mathcal P^{ch,\mathbf J}_{\AA^d}[M](I'\twoheadrightarrow\{\bullet\}),
\qquad
\mu_{\{\bullet\}\sqcup I''}\in
\mathcal P^{ch,\mathbf J}_{\AA^d}[M](\{\bullet\}\sqcup I''\twoheadrightarrow\{\star\}).
\]
We define their partial composition
\[
  \mu_{\{\bullet\}\sqcup I''}\circ\mu_{I'}
  \in
  \mathcal P^{ch,\mathbf J}_{\AA^d}[M](I\twoheadrightarrow\{\star\}).
\]

First, by abuse of notation, we extend \(\mu_{I'}\) to a map
\[
\mu_{I'}\in
\operatorname{Hom}_{\mathcal D_{(\AA^d)^I}}
\left(
\mathbf J^I_{\AA^d}(M^{\boxtimes I}),
\mathbf J^{\{\bullet\}\sqcup I''}_{\AA^d}(M\boxtimes M^{\boxtimes I''})
\otimes_{\kk[\lambda_\bullet]}\kk[\lambda_{I'}]
\right).
\]
To describe this extension, introduce the partial Jouanolou torsor
\[
\mathbb J^{(I'),I''}_{\AA^d}
\define
\left\{
\sum_{\sss=1}^d x_{ij}^{\sss}(z_i^{\sss}-z_j^{\sss})=1
\;\middle|\;
\begin{array}{l}
 i,j\in I'',\ i<j,\ \text{or}\\
 i\in I',\ j\in I''
\end{array}
\right\}
\subset
(\AA^d)^I\times (\AA^d)^{\binom{|I''|}{2}+|I'|\cdot |I''|}.
\]
Let
\[
\mathbf J^{(I'),I''}_{\AA^d}
\define
\Gamma\left(
\mathbb J^{(I'),I''}_{\AA^d},
\Omega^\bullet_{\mathbb J^{(I'),I''}_{\AA^d}/\op{Conf}_{I}(\AA^d)}
\right).
\]
Then
\[
\mathbf J^I_{\AA^d}
\cong
\left(
\mathbf J^{I'}_{\AA^d}
\otimes_{\mathcal O_{(\AA^d)^{I'}}}
\mathcal O_{(\AA^d)^I}
\right)
\otimes_{\mathcal O_{(\AA^d)^I}}
\mathbf J^{(I'),I''}_{\AA^d}.
\]

The chiral operation \(\mu_{I'}\) gives a map
\[
\begin{aligned}
\mathbf J^I_{\AA^d}(M^{\boxtimes I})
&=
\mathbf J^{(I'),I''}_{\AA^d}
\left(
\mathbf J^{I'}_{\AA^d}(M^{\boxtimes I'})
\boxtimes M^{\boxtimes I''}
\right)                                                     \\
&\longrightarrow
\mathbf J^{(I'),I''}_{\AA^d}
\left(
\Delta^{I'}_*M\boxtimes M^{\boxtimes I''}
\right).
\end{aligned}
\]
Consider the \(\mathcal O_{(\AA^d)^I}\)-module
\[
  (M\boxtimes M^{\boxtimes I''})
  \otimes_{\mathcal O_{(\AA^d)^I}}
  \mathbf J^{(I'),I''}_{\AA^d}.
\]
The identity map extends to a \(\mathcal D\)-module map
\[
\begin{aligned}
 v^{(I')}:&\
\mathbf J^{(I'),I''}_{\AA^d}
(\Delta^{I'}_*M\boxtimes M^{\boxtimes I''})=
\left((M\otimes_{\kk[\lambda_\bullet]}\kk[\lambda_{I'}])
\boxtimes M^{\boxtimes I''}\right)
\otimes_{\mathcal O_{(\AA^d)^I}}
\mathbf J^{(I'),I''}_{\AA^d}                                \\
&\longrightarrow
\left(
(M\boxtimes M^{\boxtimes I''})
\otimes_{\mathcal O_{(\AA^d)^I}}
\mathbf J^{(I'),I''}_{\AA^d}
\right)
\otimes_{\kk[\lambda_\bullet]}\kk[\lambda_{I'}].
\end{aligned}
\]

Next define
\[
  c^{(I')}:\mathbf J^{(I'),I''}_{\AA^d}
  \longrightarrow
  \mathbf J^{\{\bullet\}\sqcup I''}_{\AA^d}
\]
by
\[
\begin{array}{lll}
  x_{ij}^{\sss}\mapsto x_{\bullet j}^{\sss},
  & \mathbf d x_{ij}^{\sss}\mapsto \mathbf d x_{\bullet j}^{\sss},
  & z_i^{\sss}\mapsto z_\bullet^{\sss},\\[1mm]
  x_{jj'}^{\sss}\mapsto x_{jj'}^{\sss},
  & \mathbf d x_{jj'}^{\sss}\mapsto \mathbf d x_{jj'}^{\sss},
  & z_j^{\sss}\mapsto z_j^{\sss},
\end{array}
\qquad
\sss=1,\ldots,d,
\quad
 i\in I',\ j,j'\in I''.
\]
This is an explicit model for restriction along the diagonal that collapses the
labels in \(I'\) to the new label \(\bullet\). It induces, again by abuse of
notation, a map
\[
\begin{aligned}
 c^{(I')}:&\
\left(
(M\boxtimes M^{\boxtimes I''})
\otimes_{\mathcal O_{(\AA^d)^I}}
\mathbf J^{(I'),I''}_{\AA^d}
\right)
\otimes_{\kk[\lambda_\bullet]}\kk[\lambda_{I'}]              \\
&\longrightarrow
\left(
(M\boxtimes M^{\boxtimes I''})
\otimes_{\mathcal O_{(\AA^d)^{\{\bullet\}\sqcup I''}}}
\mathbf J^{\{\bullet\}\sqcup I''}_{\AA^d}
\right)
\otimes_{\kk[\lambda_\bullet]}\kk[\lambda_{I'}].
\end{aligned}
\]

Using \(c^{(I')}\), we obtain the \(\mathcal D\)-module map
\[
 c^{(I')}\circ v^{(I')}:
\mathbf J^{(I'),I''}_{\AA^d}
(\Delta^{I'}_*M\boxtimes M^{\boxtimes I''})
\longrightarrow
\mathbf J^{\{\bullet\}\sqcup I''}_{\AA^d}(M\boxtimes M^{\boxtimes I''})
\otimes_{\kk[\lambda_\bullet]}\kk[\lambda_{I'}].
\]
Thus we have a map
\[
\begin{aligned}
\mu_{I'\subset I}:\quad
\mathbf J^I_{\AA^d}(M^{\boxtimes I})
&\xrightarrow{\mu_{I'}\boxtimes\mathrm{Id}}
\mathbf J^{(I'),I''}_{\AA^d}
(\Delta^{I'}_*M\boxtimes M^{\boxtimes I''})                 \\
&\xrightarrow{c^{(I')}\circ v^{(I')}}
\mathbf J^{\{\bullet\}\sqcup I''}_{\AA^d}
(M\boxtimes M^{\boxtimes I''})
\otimes_{\kk[\lambda_\bullet]}\kk[\lambda_{I'}].
\end{aligned}
\]
We then compose with \(\mu_{\{\bullet\}\sqcup I''}\):
\[
\begin{aligned}
\mathbf J^I_{\AA^d}(M^{\boxtimes I})
&\xrightarrow{\mu_{I'\subset I}}
\mathbf J^{\{\bullet\}\sqcup I''}_{\AA^d}
(M\boxtimes M^{\boxtimes I''})
\otimes_{\kk[\lambda_\bullet]}\kk[\lambda_{I'}]              \\
&\xrightarrow{\mu_{\{\bullet\}\sqcup I''}}
M\otimes_{\kk[\lambda_\star]}\kk[\lambda_{\{\bullet\}\cup I''}]
\otimes_{\kk[\lambda_\bullet]}\kk[\lambda_{I'}]              \\
&=
M\otimes_{\kk[\lambda_\star]}\kk[\lambda_I].
\end{aligned}
\]
By a slight abuse of notation, we define
\[
  \mu_{\{\bullet\}\sqcup I''}\circ\mu_{I'}
  \define
  \mu_{\{\bullet\}\sqcup I''}\circ\mu_{I'\subset I}.
\]
These partial compositions, together with the symmetric-group actions above, define the
dg operad structure.
\end{proof}

We conclude this section with our main definition.

\begin{defn}
A \textit{homotopy chiral algebra} in the Jouanolou model is a dg
\(\mathcal D_{\AA^d}\)-module \(\mathcal A\) together with a map of dg operads
\[
  \mathcal L\!\op{ie}_\infty
  \longrightarrow
  \mathcal P^{ch,\mathbf J}_{\AA^d}[\mathcal A],
\]
where \(\mathcal L\!\op{ie}_\infty\) is the dg operad controlling
\(L_\infty\)-algebras.
\end{defn}

Explicitly, to construct such a chiral algebra, one must produce a collection of
\(\mathcal D\)-module maps \(\{\mu_I\}\) satisfying the \(L_\infty\)-relations.

\begin{rem}
A homotopy chiral algebra in the Jouanolou model defines a \(d\)-dimensional
chiral algebra, up to equivalence, in the sense of \cite{FG}.
\end{rem}

\begin{rem}
In dimension \(d=1\), the notion of an \(L_\infty\) chiral algebra has recently
appeared in \cite{malikov2024homotopy}. When \(d=1\), our definition reduces to
that of \textit{loc. cit.}
\end{rem}

\begin{rem}\label{rem:CompareJouanolouPoly}
Using the Thom--Sullivan--Whitney method for computing homotopy limits of
cosimplicial spaces, a different model for configuration spaces has recently been
presented in \cite{FGY}. The notion of a chiral algebra in \cite{FGY} is
equivalent, in a homotopical sense, to the notion we present here. On the other
hand, explicit examples are presented quite differently in the two models. For
instance, it is clear from our constructions that the examples of chiral algebras
which appear in Section~\ref{s:example} depend continuously on the input data.
Additionally, the Jouanolou model has a \(\mathrm{GL}_d\)-symmetry, which endows
the Jouanolou operad with a \(\mathrm{GL}_d\)-equivariant structure. Furthermore,
our later construction of the unit chiral algebra is \(\mathrm{GL}_d\)-equivariant,
which is not true for the model given in \cite{FGY}.

Another advantage of our approach is that chiral operations in the Jouanolou
model are manifestly related to holomorphic quantum field theory; indeed, our
construction of the \textit{unit} chiral algebra uses Feynman graph integrals, as
we will soon explain. Going in the opposite direction, we will see in
Section~\ref{s:A2} how our chiral operations can be used to recover explicit
formulas for expectation values in higher-dimensional quantum field theory.
\end{rem}

Before we move on, we point out a general construction which mimics a construction
for ordinary, \(d=1\), chiral algebras. Let \(h(-)\) denote the de Rham cohomology
functor. Thus, for a dg right \(\mathcal D\)-module \(\mathcal A\), we obtain
\[
  h(\mathcal A)
  \define
  \mathcal A\otimes_{\mathcal D_{\AA^d}}\mathcal O_{\AA^d}.
\]
If \(\mathcal A\) is a chiral algebra in the Jouanolou model, then the chiral
operations endow \(h(\mathcal A)\) with the structure of a sheaf of
\(L_\infty\)-algebras.

Let \(I=\{1,\ldots,m\}\). There is a natural map
\[
  \mathcal A^{\boxtimes m}
  \longrightarrow
  \mathbf J^I_{\AA^d}(\mathcal A^{\boxtimes I}).
\]
Postcomposing this map with the chiral operation \(\mu_I\), we obtain a map of dg
\(\mathcal D\)-modules
\[
  \mathcal A^{\boxtimes m}
  \longrightarrow
  \Delta_*^{I/\{\star\}}\mathcal A.
\]
Applying the de Rham functor \(h(-)\), we obtain
\[
  h(\mathcal A)^{\boxtimes m}
  \longrightarrow
  \Delta_{\cdot}^I h(\mathcal A),
\]
where \(\Delta_{\cdot}^I\) denotes the sheaf-theoretic pushforward along the diagonal.
Applying adjunction, we obtain the \(m\)-linear map
\[
  [-]_m:h(\mathcal A)\times\cdots\times h(\mathcal A)
  \longrightarrow
  h(\mathcal A).
\]
The \(L_\infty\)-relations satisfied by the chiral operations \(\{\mu_I\}\)
immediately yield the following.

\begin{prop}\label{prop:derham}
The operations \(\{[-]_m\}\) endow \(h(\mathcal A)\) with the structure of a sheaf
of \(L_\infty\)-algebras.
\end{prop}

\section{Feynman graph integral construction}\label{s:feynman}

In this section, we will use Feynman graph integrals in the setting of holomorphic quantum field theory to construct the unit
$L_{\infty}$ chiral operations in any dimension.
In Section~\ref{s:feynman1d}, we recall the standard chiral operation in dimension $d=1$.
Of course, here, the chiral operations of the unit chiral algebra are strict in the $L_\infty$ sense.
In higher dimensions, we largely use the techniques developed in \cite{wang2024feynman}.

We will work over $\kk=\mathbb{C}$ throughout this section. The affine variety $\mathbb{A}^d$ is understood as $\mathbb{C}^d$ with the analytic topology. For example, $\mathcal{O}_{\mathbb{A}^d}$ is the sheaf of holomorphic functions on $\mathbb{C}^d$.

\begin{thm}
  The $\mathcal{D}$-module $\omega_{\mathbb{A}^d}[d-1]$ carries the structure of a $GL_d$-equivariant homotopy chiral
  algebra in the Jouanolou model.
  For each $k$, its $k$-ary operation is expressed in terms of the higher-dimensional residue, see
  Definition~\ref{dfn:res}.
  It is a model for the \textit{unit} homotopy chiral algebra.
\end{thm}

\subsection{The unit chiral algebra on $\mathbb{A}^1$} \label{s:feynman1d}
We begin by reviewing the one-dimensional unit chiral algebra $\omega_{\mathbb{A}^1}$ and how Feynman graph integrals
can be used to repackage the chiral operations.
Our most novel viewpoint is the utility of Schwinger space integrals (see Appendix~\ref{Schwinger spaces}) to express
these operations.
This approach will apply more generally and allow us to construct the unit chiral operations in higher dimensions.

When $d=1$, the model for the source of two-ary chiral operations for $\omega_{\AA^1}$ is
$$
\mathbf{J}_{\mathbb{A}^1}^{\{1,2\}}(\omega_{\mathbb{A}^1}^{\boxtimes\{1,2\}})=\left\{f(z_{1},z_{2})dz_{1}\boxtimes dz_{2} \; | \; f\in \kk\left[z_{1},z_{2},\frac{1}{z_{1}-z_{2}}\right]\right\}
$$
while the target for the two-ary chiral operations is
\begin{equation}
  \omega_{\mathbb{A}^1} \otimes_{\kk[\lambda_\star]} \kk[\lambda_1,\lambda_2]
\end{equation}
where $\lambda_\star$ acts on $\omega_{\AA^1}$ by $\del_z$ and, on the left, acts by $\lambda_\star = \lambda_1 +
\lambda_2$.
\begin{defn}
  We define the chiral operation $\mu_2 = \mu_{\{1,2\}}$ for the unit $\omega_{\mathbb{A}^1}$ by the following formula:
    $$
    \mu_2(f(z_{1},z_{2})dz_{1}\boxtimes dz_{2})=\frac{1}{2\pi i}e^{-(\lambda_{1}+\lambda_{2})w}\left. \left(\oint_{z_{1}=z_2}f(z_{1},z_{2})e^{\lambda_{1}z_{1}+\lambda_{2}z_{2}}dz_{1}\boxtimes dz_{2}\right)\right|_{z_{2}=w},
    $$
    where $\oint_{z_{1}=z_{2}}$ is the contour integral of variable $z_{1}$ over a simple loop around $z_{2}$.
  \end{defn}

  We unpack the definition.
  Suppose that
$$\alpha=\frac{g(z_{1},z_{2})}{(z_{1}-z_{2})^{n}}dz_{1}\boxtimes dz_{2}\in \mathbf{J}_{\mathbb{A}^1}^{\{1,2\}}
(\omega_{\mathbb{A}^1}^{\boxtimes\{1,2\}}),
$$
where $g(z_{1},z_{2})\in \kk[z_{1},z_{2}]$.
Expanding, we find that
\begin{align*}
 \mu_2(\alpha)
 &=\frac{1}{2\pi i}e^{-(\lambda_{1}+\lambda_{2})w}\left. \left(\oint_{z_{1}=z_2}\frac{g(z_{1},z_{2})}{(z_{1}-z_{2})^{n}}e^{\lambda_{1}z_{1}+\lambda_{2}z_{2}}dz_{1}\boxtimes dz_{2}\right)\right|_{z_{2}=w}\\
 &= e^{-(\lambda_{1}+\lambda_{2})w}\frac{1}{(n-1)!}\left.\left(\partial_{z_{1}}^{n-1}\left(g(z_{1},w)e^{\lambda_{1}z_{1}+\lambda_{2}w}\right)\right)\right|_{z_{1}=w}dw
 \\
 &=\frac{1}{(n-1)!}
\left.\left((\partial_{z_{1}}+\lambda_{1})^{n-1}g(z_{1},w)\right)\right|_{z_{1}=w}dw,
\end{align*}
so $\mu_2(\alpha)\in \omega_{\mathbb{A}^{1}}\otimes_{\kk[\lambda_{\star}]}\kk[\lambda_{1},\lambda_{2}].$
For example, when $g = 1$ this reduces to
\begin{equation}
  \mu_2\left(\frac{d z_{1} \boxtimes d z_{2}}{(z_{1}-z_{2})^{n}}\right) = \frac{\lambda_{1}^{n-1}}{(n-1)!} \otimes d w
\end{equation}

\begin{prop}
    $\mu_2$ defines a chiral operation, i.e.,
    $$    \mu_2\in \cP^{ch}_{\AA^1}[\omega_{\AA^1}](2)
    $$
\end{prop}
\begin{proof}
  We prove that $\mu_2$ is a $\mathcal{D}_{(\mathbb{A}^{1})^{2}}[(\mathbb{A}^{1})^{2}]$-module morphism.

Recall that the $\mathcal D$-module structure is determined by the rules
$$
\alpha \cdot z_{i}=z_{i}\alpha,\quad
\alpha\cdot \partial_{z_{i}}=-\partial_{z_{i}}\left(\frac{g(z_{1},z_{2})}{(z_{1}-z_{2})^{n}}\right)dz_{1}\boxtimes dz_{2} .
$$
Thus, we can directly compute:
\begin{align*}
  \mu_{2}(\alpha \cdot z_{i})&=\frac{1}{2\pi i}e^{-(\lambda_{1}+\lambda_{2})w}\left. \left(\oint_{z_{1}=z_2}\frac{z_{i}g(z_{1},z_{2})}{(z_{1}-z_{2})^{n}}e^{\lambda_{1}z_{1}+\lambda_{2}z_{2}}dz_{1}\boxtimes dz_{2}\right)\right|_{z_{2}=w}\\
&=\frac{1}{2\pi i}e^{-(\lambda_{1}+\lambda_{2})w}\left. \left(\oint_{z_{1}=z_2}\frac{g(z_{1},z_{2})}{(z_{1}-z_{2})^{n}}\partial_{\lambda_{i}}\left(e^{\lambda_{1}z_{1}+\lambda_{2}z_{2}}\right)dz_{1}\boxtimes dz_{2}\right)\right|_{z_{2}=w}\\
&=\frac{1}{2\pi i}(w+\partial_{\lambda_{i}})\left(e^{-(\lambda_{1}+\lambda_{2})w}\left. \left(\oint_{z_{1}=z_2}\frac{g(z_{1},z_{2})}{(z_{1}-z_{2})^{n}}e^{\lambda_{1}z_{1}+\lambda_{2}z_{2}}dz_{1}\boxtimes dz_{2}\right)\right|_{z_{2}=w}\right)\\
&=\mu_{2}(\alpha )\cdot z_{i},
\end{align*}
\begin{align*}
  \mu_{2}(\alpha \cdot \partial_{z_{i}})&=-\frac{1}{2\pi i}e^{-(\lambda_{1}+\lambda_{2})w}\left. \left(\oint_{z_{1}=z_2}\partial_{z_{i}}\left(\frac{g(z_{1},z_{2})}{(z_{1}-z_{2})^{n}}\right)e^{\lambda_{1}z_{1}+\lambda_{2}z_{2}}dz_{1}\boxtimes dz_{2}\right)\right|_{z_{2}=w}\\
&=-\frac{1}{2\pi i}e^{-(\lambda_{1}+\lambda_{2})w}\left. \left(\oint_{z_{1}=z_2}\partial_{z_{i}}\left(\frac{g(z_{1},z_{2})}{(z_{1}-z_{2})^{n}}e^{\lambda_{1}z_{1}+\lambda_{2}z_{2}}\right)dz_{1}\boxtimes dz_{2}\right)\right|_{z_{2}=w}\\
&+\frac{1}{2\pi i}e^{-(\lambda_{1}+\lambda_{2})w}\left. \left(\oint_{z_{1}=z_2}\frac{g(z_{1},z_{2})}{(z_{1}-z_{2})^{n}}\partial_{z_{i}}\left(e^{\lambda_{1}z_{1}+\lambda_{2}z_{2}}\right)dz_{1}\boxtimes dz_{2}\right)\right|_{z_{2}=w}\\
&=0+\frac{\lambda_{i}}{2\pi i}e^{-(\lambda_{1}+\lambda_{2})w}\left. \left(\oint_{z_{1}=z_2}\frac{g(z_{1},z_{2})}{(z_{1}-z_{2})^{n}}e^{\lambda_{1}z_{1}+\lambda_{2}z_{2}}dz_{1}\boxtimes dz_{2}\right)\right|_{z_{2}=w}\\
&=\mu_{2}(\alpha )\cdot \partial_{z_{i}}.
\end{align*}
\end{proof}

The special feature in dimension one is that the unit chiral algebra has trivial higher operations.
Notice that by simple degree reasons, the unit chiral algebra must necessarily have vanishing $k \geq 3$-ary
operations.
In fact, the operation $\mu_2$ satisfies the ordinary Jacobi identity.

\begin{prop}[\cite{BD}, see also \cite{HLdmodule}]
The operation $\mu_2$ satisfies the Jacobi identity, hence it defines the structure of a (strict, one-dimensional) chiral
algebra.
\end{prop}
In the definition of chiral operation for $\omega_{\mathbb{A}^{1}}$, we used the residue.
Now we introduce an alternative method to define the residue.
The advantage of this definition is that it lends itself to a natural higher-dimensional extension.
Thus, this is the approach we will take in the next section to define the higher-dimensional chiral operations (for the
unit chiral algebra).
The first step is to prove the following integral representation of the Cauchy kernel.

\begin{lem}
    $$
    \frac{1}{z}=-\int_{0}^{+\infty}e^{-zy}dy,
    $$
    where $y=\frac{\bar{z}}{t}$, and $dy$ is viewed as the differential form $$dy=\frac{1}{t} d\bar{z}-\frac{\bar{z}}{t^2} d t.$$
    The integration is over $t$. We call $t$ the Schwinger parameter.
\end{lem}
\begin{proof}
  \begin{equation}
        \frac{1}{z} =\frac{\bar{z}}{z\bar{z}}\\
    =
\int_{0}^{+\infty}\bar{z}e^{-z\bar{z}t}dt\\
        =
        -\int_{0}^{+\infty}e^{-\frac{z\bar{z}}{t}}\bar{z}d\left(\frac{1}{t}\right)\\
        =
-\int_{0}^{+\infty}e^{-zy}dy.
\end{equation}
\end{proof}

Let $\alpha=\frac{g(z)}{z^{n}}dz$. Using the above lemma, we can rephrase the residue as
\begin{align*}
\oint_{z=0}\alpha
&=
\oint_{z=0}\frac{g(z)}{z^{n}}dz \\
&=
\int_{(0,+\infty)^{n}}\int_{S^{1}}e^{-\left(\sum_{k=1}^{n}\frac{1}{t_{k}}\right)z\bar{z}}g(z)\prod_{k=1}^{n}dy_{k}dz,
\end{align*}
where $S^{1}=\{z\in \kk=\mathbb{C}\mid |z|^2=1\}$, $y_{k}=\frac{\bar{z}}{t_{k}}$.

Let $d_{\mathbb{A}^{1}}$, $d_{t}$ be the de Rham differential on $\mathbb{A}^{1}$ and $(0,+\infty)^{n}$ respectively.
Then the differential form appearing in the above formula is $(d_{\AA^{1}} + d_{t})$-closed:
$$
(d_{\mathbb{A}^{1}}+d_{t}) \left(e^{-\left(\sum_{k=1}^{n}\frac{1}{t_{k}}\right)z\bar{z}}g(z)\prod_{k=1}^{n}dy_{k}dz\right)=0.
$$
We will find an equivalent way to express the residue $\oint_{z=0} \alpha$ in terms of a homologous integration cycle.

Let
$$S^{+}((0,+\infty)^n)=\{(t_{1},\cdots,t_{n})\in(0,+\infty)^{n}|\sum_{k=1}^{n}t_{k}^{2}=1\} \subset (0,+\infty)^{n},
$$
using Stokes' theorem carefully gives
\begin{align*}
\oint_{z=0}\alpha
&=\int_{(0,+\infty)^{n}}\int_{S^{1}}e^{-\left(\sum_{k=1}^{n}\frac{1}{t_{k}}\right)z\bar{z}}g(z)\prod_{k=1}^{n}dy_{k}dz\\
&=\int_{S^{+}((0,+\infty)^n)}\int_{\mathbb{A}^{1}}e^{-\left(\sum_{k=1}^{n}\frac{1}{t_{k}}\right)z\bar{z}}g(z)\prod_{k=1}^{n}dy_{k}dz.
\end{align*}
Rather than constructing a homotopy between these two integration cycles, we provide a direct computation to obtain this result:
\begin{prop}
    Let $\alpha=\frac{g(z)}{z^{n}}dz$. Then
    $$
    \frac{1}{2\pi i}\int_{S^{+}((0,+\infty)^n)}\int_{\mathbb{A}^{1}}e^{-\left(\sum_{k=1}^{n}\frac{1}{t_{k}}\right)z\bar{z}}g(z)\prod_{k=1}^{n}dy_{k}dz=\frac{\partial_{z}^{n-1}g(0)}{(n-1)!}.
    $$
\end{prop}
\begin{proof}
Noting that the integral is of Gaussian type, we apply Wick's theorem to obtain
\begin{equation}
    \frac{1}{2\pi i}\int_{\mathbb{A}^{1}}e^{-\left(\sum_{k=1}^{n}\frac{1}{t_{k}}\right)z\bar{z}}g(z)\prod_{k=1}^{n}dy_{k}dz
   =(-1)^{n-1}\left.\frac{1}{\sum_{k=1}^{n}\frac{1}{t_{k}}}\iota_{\partial_{\bar{z}}}e^{\frac{1}{\sum_{k=1}^{n}\frac{1}{t_{k}}}\partial_{z}\partial_{\bar{z}}}\left(g(z)\prod_{k=1}^{n}dy_{k}\right)\right|_{z=0},
\end{equation}
where $\iota_{\partial_{\bar{z}}}$ denotes contraction with the vector field $\partial_{\bar{z}}$.

Notice that as operators acting on differential forms, we have the following relation
\begin{align*}
    e^{\frac{1}{\sum_{k=1}^{n}\frac{1}{t_{k}}}\partial_{z}\partial_{\bar{z}}}\circ \left(dy_{i} \wedge - \right) \circ e^{\frac{-1}{\sum_{k=1}^{n}\frac{1}{t_{k}}}\partial_{z}\partial_{\bar{z}}}=dy_{i} \wedge +d(\frac{\frac{1}{t_{i}}}{\sum_{k=1}^{n}\frac{1}{t_{k}}})\partial_{z}.
\end{align*}
Thus we can express the Gaussian integration as
\begin{align*}
    &\frac{1}{2\pi i}\int_{\mathbb{A}^{1}}e^{-\left(\sum_{k=1}^{n}\frac{1}{t_{k}}\right)z\bar{z}}g(z)\prod_{k=1}^{n}dy_{k}dz\\
    &=
    (-1)^{n-1}\sum_{k=1}^{n}\left((-1)^{k-1}\frac{\frac{1}{t_{k}}}{\sum_{k''=1}^{n}\frac{1}{t_{k''}}}\prod_{k'\neq k}d(\frac{\frac{1}{t_{k'}}}{\sum_{k''=1}^{n}\frac{1}{t_{k''}}})\right)\partial_{z}^{n-1}g(0).
\end{align*}

Therefore,
\begin{align*}
     &\frac{1}{2\pi i}\int_{S^{+}((0,+\infty)^n)}\int_{\mathbb{A}^{1}}e^{-\left(\sum_{k=1}^{n}\frac{1}{t_{k}}\right)z\bar{z}}g(z)\prod_{k=1}^{n}dy_{k}dz\\
     &=
     (-1)^{n-1}\partial_{z}^{n-1}g(0)\int_{S^{+}((0,+\infty)^n)}\sum_{k=1}^{n}\left((-1)^{k-1}\frac{\frac{1}{t_{k}}}{\sum_{k''=1}^{n}\frac{1}{t_{k''}}}\prod_{k'\neq k}d(\frac{\frac{1}{t_{k'}}}{\sum_{k''=1}^{n}\frac{1}{t_{k''}}})\right).
\end{align*}

To compute the integral over $S^{+}((0,+\infty)^n)$, we apply localization to the following $\mathbb{R}_+ = (0,+\infty)$-action on $(0,+\infty)^n$:
$$
\lambda\cdot (t_{1},\cdots,t_{n})=(\lambda t_{1},\cdots,\lambda t_{n}),\quad \lambda\in \mathbb{R}_+.
$$

Notice that
$$
\sum_{k=1}^{n}\left((-1)^{k-1}\frac{\frac{1}{t_{k}}}{\sum_{k''=1}^{n}\frac{1}{t_{k''}}}\prod_{k'\neq k}d(\frac{\frac{1}{t_{k'}}}{\sum_{k''=1}^{n}\frac{1}{t_{k''}}})\right)
$$
induces a well-defined differential form on $(0,+\infty)^{n}/\mathbb{R}_+$, so we have
\begin{align*}
&(-1)^{n-1}\int_{S^{+}((0,+\infty)^n)}\sum_{k=1}^{n}\left((-1)^{k-1}\frac{\frac{1}{t_{k}}}{\sum_{k''=1}^{n}\frac{1}{t_{k''}}}\prod_{k'\neq k}d(\frac{\frac{1}{t_{k'}}}{\sum_{k''=1}^{n}\frac{1}{t_{k''}}})\right)\\
&=
(-1)^{n-1}\int_{\{(t_{1},\cdots,t_{n})\in((0,+\infty)^n)|\sum_{k=1}^{n}\frac{1}{t_{k}}=1\}}\sum_{k=1}^{n}\left((-1)^{k-1}\frac{\frac{1}{t_{k}}}{\sum_{k''=1}^{n}\frac{1}{t_{k''}}}\prod_{k'\neq k}d(\frac{\frac{1}{t_{k'}}}{\sum_{k''=1}^{n}\frac{1}{t_{k''}}})\right)\\
&=
\int_{\{(u_{1},\cdots,u_{n})\in((0,+\infty)^n)|\sum_{k=1}^{n}u_{k}=1\}}\sum_{k=1}^{n}\left((-1)^{k-1}u_{k}\prod_{k'\neq k}du_{k'}\right)\\
&=\frac{1}{(n-1)!}.
\end{align*}

Thus, the conclusion follows.
\end{proof}

Although the expression
$$
\frac{1}{2\pi i}\int_{S^{+}((0,+\infty)^n)}\int_{\mathbb{A}^{1}}e^{-\left(\sum_{k=1}^{n}\frac{1}{t_{k}}\right)z\bar{z}}g(z)\prod_{k=1}^{n}dy_{k}dz
$$
appears to be more complicated than the purely holomorphic contour integral
$$
\frac{1}{2\pi i}\oint_{z=0}\frac{g(z)}{z^{n}}dz,
$$
we will show that it can be extended to higher dimensional affine spaces in the next subsection.

\begin{rem}
    This computation is parallel to the calculation in \cite{li2023vertex}, where the author identifies the QME in BV formalism with a Maurer-Cartan equation in a vertex algebra.
\end{rem}

\subsection{Holomorphic Feynman graph integrals}

In this subsection, we lay out the prerequisite work on Feynman graph integrals which are central to the concept of
residues that we develop for the Jouanolou model of $(\omega_{\mathbb{A}^{d}}[d])^{\boxtimes  I}$.

Let $\mathcal{A}^{0,\bullet}(\mathrm{Conf}_{ I}(\mathbb{A}^d))$ be the Dolbeault complex of the structure sheaf $\mathcal{O}(\mathrm{Conf}_{ I}(\mathbb{A}^{d}))$. Define the map
$$
\Phi_I:\mathbf{J}^{ I}_{\mathbb{A}^{d}}\rightarrow\mathcal{A}^{0,\bullet}(\mathrm{Conf}_{ I}(\mathbb{A}^d))
$$
$$
\Phi_I(z_{i}^{\sss}) = z_{i}^{\sss},\quad \Phi_I(x_{ij}^{\sss}) = \frac{\bar{z}_{i}^{\sss} - \bar{z}_{j}^{\sss}}{|z_{i} - z_{j}|^{2}},\quad \Phi_I(\mathbf{d}x_{ij}^{\sss}) = \bar{\partial}\left(\frac{\bar{z}_{i}^{\sss} - \bar{z}_{j}^{\sss}}{|z_{i} - z_{j}|^{2}}\right).
$$
Here:
\begin{itemize}
    \item \(\sss \in \{1,2,\dots,d\}\) is the coordinate index.
    \item \(i, j \in  I\) are elements in the ordered set.
    \item \(|z_{i} - z_{j}|^{2} = \sum\limits_{\ttt=1}^{d}(z_{i}^{\ttt} - z_{j}^{\ttt})(\bar{z}_{i}^{\ttt} - \bar{z}_{j}^{\ttt})\) is the squared norm.
    \item \(\bar{\partial}\) is the Dolbeault differential.
\end{itemize}
It is easily checked that this map is well-defined and is a morphism of commutative dg algebras.
The map $\Phi_I$ is dense in cohomology.\footnote{In fact, $\Phi_I$
is injective, but we will not need that.}

We also point out that $\Phi_I$ is a $\mathcal{D}_{(\mathbb{A}^{d})^{ I}}$-module morphism.
Similarly, we can define a natural dg $\mathcal{D}_{(\mathbb{A}^{d})^{ I}}$-module
morphism
\[
  \mathbf{J}^{ I}_{\mathbb{A}^{d}}((\omega_{\mathbb{A}^{d}}[d])^{\boxtimes I}) \to \mathcal{A}^{0,\bullet}
  (\mathrm{Conf}_{ I}(\mathbb{A}^d),(\omega_{\mathbb{A}^{d}}[d])^{\boxtimes I}) .
\]

In the following discussion, we identify the elements of $\mathbf{J}^{ I}_{\mathbb{A}^{d}}((\omega_{\mathbb{A}^{d}}[d])^{\boxtimes I})$ with the corresponding elements in $\mathcal{A}^{0,\bullet}(\mathrm{Conf}_{ I}(\mathbb{A}^d),(\omega_{\mathbb{A}^{d}}[d])^{\boxtimes I})$.

\subsubsection{Schwinger space}

In our previous discussion of residues when $d=1$, we represented the residue kernel $\frac{1}{z}$ as an integral over Schwinger parameters.
Now our goal is to introduce similar integral representations of $\Phi_I(x_{ij}^{\sss})$ and $\Phi_I(\mathbf{d}x_{ij}^{\sss})$:

\begin{lem}
    We have the following equalities:
    \begin{enumerate}
        \item
        $$
        \Phi_I(x_{ij}^{\sss})=-\int_{0}^{+\infty}e^{-(z_{ij}|y_{ij})}dy_{ij}^{\sss}
        $$
        \item
        $$
        \Phi_I(\mathbf{d}x_{ij}^{\sss})=-\int_{0}^{+\infty}e^{-(z_{ij}|y_{ij})}(z_{ij}|dy_{ij})dy_{ij}^{\sss}
        $$
    \end{enumerate}
    Here:
        \begin{itemize}
            \item $y_{ij}^{\sss}=\frac{\bar{z}_{i}^{\scriptstyle\sss}-\bar{z}_{j}^{\scriptstyle\sss}}{t}$, where $t$ is called the Schwinger parameter.
            \item $z_{ij}^{\sss}=z_{i}^{\sss}-z_{j}^{\sss}$.
            \item
            $(z_{ij}| y_{ij})=\sum\limits_{\scriptstyle\sss=1}^{d}z_{ij}^{\sss}y_{ij}^{\sss}$ is the dot product.
            \item $dy_{ij}^{\sss}$ is the differential form
            $$
            dy_{ij}^{\sss}=\frac{d\bar{z}_{i}^{\sss}-d\bar{z}_{j}^{\sss}}{t}-\frac{(\bar{z}_{i}^{\sss}-\bar{z}_{j}^{\sss})dt}{t^{2}}.
            $$
            \item both integrations are with respect to the Schwinger parameter $t$.
        \end{itemize}
\end{lem}
\begin{proof}
    We note
    \begin{align*}
        \Phi_I(x_{ij}^{\sss})
        &=
        \frac{\bar{z}_{i}^{\sss} - \bar{z}_{j}^{\sss}}{|z_{i} - z_{j}|^{2}}\\
        &=
        (\bar{z}_{i}^{\sss} - \bar{z}_{j}^{\sss})\int_{0}^{+\infty}e^{-|z_{i} - z_{j}|^{2}u}du\\
        &=
        -(\bar{z}_{i}^{\sss} - \bar{z}_{j}^{\sss})\int_{0}^{+\infty}e^{-\frac{|z_{i} - z_{j}|^{2}}{t}}d\left(\frac{1}{t}\right)\\
        &=
        -\int_{0}^{+\infty}e^{-(z_{ij}| y_{ij})}dy_{ij}^{\sss}.
    \end{align*}
    In the last line we have used the fact that only the $d t$ component contributes to the integral.

    By the dominated convergence theorem for derivatives, we can interchange the order of integration and differentiation when $z_{i}\neq z_{j}$. Thus
    \begin{align*}
        \Phi_I(\mathbf{d}x_{ij}^{\sss})
        &=
        -\bar{\partial}\left(\int_{0}^{+\infty}e^{-(z_{ij}| y_{ij})}dy_{ij}^{\sss}\right)\\
        &=
        \int_{0}^{+\infty}\bar{\partial}\left(e^{-(z_{ij}| y_{ij})}dy_{ij}^{\sss}\right)\\
        &=
        \int_{0}^{+\infty}(\bar{\partial}+d_{t})\left(e^{-(z_{ij}| y_{ij})}dy_{ij}^{\sss}\right)\\
        &=
        -\int_{0}^{+\infty}e^{-(z_{ij}| y_{ij})}(z_{ij}| dy_{ij})dy_{ij}^{\sss},
    \end{align*}
    where $d_t$ is the de Rham differential over the Schwinger parameter space.
    Again, we emphasize that only the term proportional to $dt$ will contribute to the integral.
\end{proof}

The following proposition is useful and follows by direct computation.

\begin{prop}\label{Lie and interior product}
    For $i,j \in  I$ define the vector field
    $$ V_{ij}=\sum_{\scriptstyle\sss=1}^{d}(\bar{z}_i^{\sss}\partial_{\bar{z}_{i}^{\scriptstyle\sss}}+\bar{z}_j^{\sss}\partial_{\bar{z}_{j}^{\scriptstyle\sss}})+t\partial_{t}.
    $$
    If $L_{V_{ij}}$ and $\iota_{V_{ij}}$ denote the Lie derivative and interior product operator, respectively, then we have the following identities:
$$
\left\{
\begin{array}{cc}
     L_{V_{ij}}e^{-(z_{ij}| y_{ij})}dy_{ij}^{\sss} & =0,  \\
      L_{V_{ij}}e^{-(z_{ij}| y_{ij})}\cdot(z_{ij}|dy_{ij})dy_{ij}^{\sss} & =0,\\
      \iota_{V_{ij}}e^{-(z_{ij}| y_{ij})}dy_{ij}^{\sss} & =0,\\
      \iota_{V_{ij}}e^{-(z_{ij}| y_{ij})}\cdot(z_{ij}|dy_{ij})dy_{ij}^{\sss} & =0.
\end{array}
\right.
$$
\end{prop}

\subsubsection{Feynman graph integrals}
Next, we introduce a graphical language to encode elements in the Jouanolou model.
To this end, we introduce some concepts from graph theory.
\begin{defn}
  A \textit{directed graph} $\vec{\Gamma}$ consists of the following data:
    \begin{enumerate}
        \item A set $\vec{\Gamma}_{0}$ of vertices. We use $|\Gamma_{0}|$ to denote the number of vertices.
        \item An ordered set $\vec{\Gamma}_{1}$ of directed edges. We use $|\Gamma_{1}|$ to denote the number of directed edges.
        \item Two maps
        $$
        t,h:\vec{\Gamma}_{1}\rightarrow\vec{\Gamma}_{0},
        $$
        which are assignments of tail and head to each directed edge. We require that
        $$
        t(e)\neq h(e)
        $$
        for any $e\in\vec{\Gamma}_{1}$, i.e., the graph $\vec{\Gamma}$ has no self-loops.
    \end{enumerate}
    Furthermore, we say that a directed graph $\vec{\Gamma}$ is \textit{decorated} if we have a special edge $e_{l}\in \vec{\Gamma}_{1}$ and a map
    $$
    \mmm \colon \vec{\Gamma}_{1}\rightarrow\{1,2,\dots,d\}.
    $$
\end{defn}

We will use $(\vec{\Gamma},\mmm,l)$ to denote a decorated directed graph. For simplicity, we will also call a decorated
directed graph a \textit{Feynman graph} in this article.

If $\vec{\Gamma}$ is a directed graph, we may use the following notations to describe $\vec{\Gamma}_{0}$ and $\vec{\Gamma}_{1}$:
$$
\begin{cases}
    \vec{\Gamma}_{0}=\{1<\cdots<n=|\Gamma_{0}|\},\\
    \vec{\Gamma}_{1}=\{e_{1}<\cdots<e_{|\Gamma_{1}|}\}.
\end{cases}
$$
\begin{defn}\label{dfn:poly}
    Let $(\vec{\Gamma},\mmm,l)$ be a Feynman graph. We define
    $$
    p_{(\vec{\Gamma},\mmm,l)}(x_{e}^{\mmm(e)},\mathbf{d}x_{e}^{\mmm(e)})=\prod_{e<e_{l}\in\vec{\Gamma}_{1}}x_{e}^{\mmm(e)}\prod_{e_l\leq e'\in\vec{\Gamma}_{1}}\mathbf{d}x_{e'}^{\mmm(e')}\in \kk[x_{ij}^{\sss},\mathbf{d}x_{ij}^{\sss}]_{i<j\in\vec{\Gamma}_{0}}^{1\leq \sss\leq d},
    $$
    where the order of the factors is determined by the order of $\vec{\Gamma}_{1}$. We have used the convention
    $$
    x_{e}^{\mmm(e)}=x_{t(e)h(e)}^{\mmm(e)}=-x_{h(e)t(e)}^{\mmm(e)}.
    $$
    We call $p_{(\vec{\Gamma},\mmm,l)}$ the corresponding monomial of $(\vec{\Gamma},\mmm,l)$.
\end{defn}

As we vary over all Feynman graphs, the monomials $\{p_{(\vec{\Gamma},\mmm,l)}\}$ span $\kk[x_{ij}^{\sss},\mathbf{d}x_{ij}^{\sss}]_{i<j\in\vec{\Gamma}_{0}}^{1\leq \sss\leq d}$, so we can use Feynman graphs to represent elements in $\kk[x_{ij}^{\sss},\mathbf{d}x_{ij}^{\sss}]_{i<j\in\vec{\Gamma}_{0}}^{1\leq \sss\leq d}$.

\begin{defn}
  Let $\vec{\Gamma}$ be a directed graph. The \textit{incidence matrix} $\rho=(\rho_{ei})_{e\in\vec{\Gamma}_{1},
  i\in\vec{\Gamma}_{0}}$ is
    $$
    \rho_{ei}=
    \begin{cases}
        1, &\text{if }t(e)=i,\\
        -1, &\text{if }h(e)=i,\\
        0, &\text{otherwise.}
    \end{cases}
    $$
    The \textit{weighted Laplacian matrix} is the $\Gamma_0 \times \Gamma_0$ matrix defined by
    $$
    M_{\vec{\Gamma}}(t)_{ij}=
    \sum_{e\in\vec{\Gamma}_{1}}\rho_{ei}\frac{1}{t_{e}}\rho_{ej}
    $$
    It is defined for $t = (t_e) \in (0,+\infty)^{\Gamma_1}$.
\end{defn}
\begin{prop}\label{Minverse}
    Let $\vec{\Gamma}$ be a connected directed graph.
    Then the matrix $M_{\vec{\Gamma}}(t)=(M_{\vec{\Gamma}}(t)_{ij})_{i,j\in\vec{\Gamma}_{0}-\{n\}}$ is invertible. We use $M^{-1}_{\vec{\Gamma}}(t)=(M^{-1}_{\vec{\Gamma}}(t)_{ij})_{i,j\in\vec{\Gamma}_{0}-\{n\}}$ to denote the inverse matrix.
\end{prop}
\begin{proof}
    See Corollary \ref{det of laplacian} for an explicit description of its inverse.
\end{proof}

    Thus, we can define the following graphical Green's function $d^{-1}_{\vec{\Gamma}}(t)=(d^{-1}_{\vec{\Gamma}}(t)
    _{ei})_{e\in\vec{\Gamma}_{1},i\in\vec{\Gamma}_{0}-\{n\}}$ by the formula
    $$
    d^{-1}_{\vec{\Gamma}}(t)_{ei}=\sum_{j=1}^{n-1}\frac{1}{t_{e}}\rho_{ej}M^{-1}_{\vec{\Gamma}}(t)_{ji}.
    $$
    It is defined for $t \in (0,+\infty)^{\Gamma_1}$.


Using these notations, we can finally describe the type of integrals that we are interested in.
\begin{defn}
    Let $(\vec{\Gamma},\mmm,l)$ be a Feynman graph and denote $\Gamma_{0}= I$. Given
    $$
    \beta\in (\omega_{\mathbb{A}^{d}}[d])^{\boxtimes I}
    $$
    and
    $$
    e^{\sum\limits_{i=1}^{n}(z_{i}|\lambda_{i})}=e^{\sum\limits_{\scriptstyle\sss=1}^{d}\sum\limits_{i=1}^{n}z^{\scriptstyle\sss}_{i}\lambda^{\scriptstyle\sss}_{i}}\in C^{\infty}((\mathbb{A}^{d})^{ I}\times (\mathbb{A}^{d})^{ I}),
    $$
    the Feynman graph integrand is defined by
    \begin{align*}
        &W((\vec{\Gamma},\mmm,l),\beta e^{\sum\limits_{i=1}^{n}(z_{i}|\lambda_{i})})\\
        &=
        \int_{(\mathbb{A}^{d})^{ I-\{n\}}}p_{(\vec{\Gamma},\mmm,l)}(-e^{-(z_{e}| y_{e})}dy_{e}^{\mmm(e)},-e^{-(z_{e}| y_{e})}(z_{e}|dy_{e})dy_{e}^{\mmm(e)})e^{\sum\limits_{i=1}^{n}(z_{i}|\lambda_{i})}\otimes \beta,
    \end{align*}
    where $z_{e}=z_{t(e)}-z_{h(e)}$, $y_{e}=\frac{\bar{z}_{t(e)}-\bar{z}_{h(e)}}{t_{e}}$.
\end{defn}

The following proposition follows from Wick's formula for Gaussian integrals.

\begin{prop}\label{explicit formula}
  For each $t = (t_{e})\in(0,+\infty)^{\Gamma_1}$ the integral
    $$
    W((\vec{\Gamma},\mmm,l),\beta e^{\sum\limits_{i=1}^{n}(z_{i}|\lambda_{i})})
    $$
    over $(\mathbb{A}^d)^{I-\{n\}}$ converges.
    Moreover,
    \begin{enumerate}
        \item If $\vec{\Gamma}$ is disconnected,
        $$
        W((\vec{\Gamma},\mmm,l),\beta e^{\sum\limits_{i=1}^{n}(z_{i}|\lambda_{i})})=0.
        $$
        \item If $\vec{\Gamma}$ is connected, we have the following explicit formula:
        \begin{align*}
            &W((\vec{\Gamma},\mmm,l),\beta e^{\sum\limits_{i=1}^{n}(z_{i}|\lambda_{i})})\\
            &=(-2\pi i)^{d(n-1)}\left(
        \iota_{\prod_{i=1}^{n-1}(d^{d}\tilde{z}_{i}d^{d}\bar{\tilde{z}}_{i})}\circ p_{(\vec{\Gamma},\mmm,l)}(-d\hat{y}_{e}^{\mmm(e)},-(\sum_{\scriptstyle\sss=1}^{d}(\partial_{\lambda_{e}^{\scriptstyle\sss}}+\tilde{z}_{e}^{\sss})\circ d\hat{y}_{e}^{\sss})\circ d\hat{y}_{e}^{\mmm(e)})
        \right.\\
        &\left.\left.(1\otimes \beta)\right)\right|_{\tilde{z}_{i}=0,1\leq i\leq n-1} e^{(\tilde{z}_{n}|\sum\limits_{i=1}^{n}\lambda_{i})}.
        \end{align*}
    \end{enumerate}
  \end{prop}
  In this proposition, we have introduced the following notations:
    \begin{itemize}
      \item We introduce the following coordinates on $(\mathbb{A}^{d})^{ I}=(\mathbb{A}^d)^n$:
$$
\begin{cases}
    z_{i}=\tilde{z}_{i}+\tilde{z}_{n}, &\text{if }i\neq n,\\
    z_{n}=\tilde{z}_{n}.
\end{cases}
$$
        \item $\iota_{\prod_{i=1}^{n-1}(d^{d}\tilde{z}_{i}d^{d}\bar{\tilde{z}}_{i})}$: if $
        \gamma\in \mathcal{A}^{d(n-1),d(n-1)}((\mathbb{A}^d)^{ I-\{n\}})
        $ is a top form, there exists a unique
        $
        \gamma'\in C^{\infty}((\mathbb{A}^d)^{ I-\{n\}}),
        $
        such that
        $
        \gamma=\gamma'\prod_{i=1}^{n-1}(d^{d}\tilde{z}_{i}d^{d}\bar{\tilde{z}}_{i}).
        $
        We define
        $$
        \iota_{\prod_{i=1}^{n-1}(d^{d}\tilde{z}_{i}d^{d}\bar{\tilde{z}}_{i})}\gamma=\gamma'.
        $$
        If $\gamma$ is not a top form, we define
        $$
        \iota_{\prod_{i=1}^{n-1}(d^{d}\tilde{z}_{i}d^{d}\bar{\tilde{z}}_{i})}\gamma=0.
        $$
        \item $d\hat{y}_{e}^{\sss}$: it is an operator defined by
        $$
        d\hat{y}_{e}^{\sss}=\sum_{j=1}^{n-1}\left(d^{-1}_{\vec{\Gamma}}(t)_{ej}d\bar{\tilde{z}}_{j}^{\sss}+d(d^{-1}_{\vec{\Gamma}}(t)_{ej})(\partial_{\tilde{z}_{j}^{\scriptstyle\sss}}+\lambda_{j}^{\sss})\right),
        $$
        where $e\in \vec{\Gamma}_{1}$, $1\leq \sss\leq d$.
    \end{itemize}
    \begin{proof}[Proof of Proposition \ref{explicit formula}]
    To simplify notation, we define
    \[
    \tilde{W}=p_{(\vec{\Gamma},\mmm,l)}(-e^{-(z_{e}| y_{e})}dy_{e}^{\mmm(e)},-e^{-(z_{e}| y_{e})}(z_{e}| dy_{e})dy_{e}^{\mmm(e)})e^{\sum\limits_{i=1}^{n}(z_{i}|\lambda_{i})}\otimes \beta.
  \]

    We first prove part (1). If $\vec{\Gamma}$ is disconnected, let $\vec{\Gamma}'$ be a connected component of $\vec{\Gamma}$, such that $n\notin \vec{\Gamma}'$. Let $V=\sum\limits_{i\in\vec{\Gamma}'_{0}}\partial_{\bar{z}_{i}}$, $\iota_V$ is the interior product operator of $V$. We can check that
    $$
    \iota_V(\tilde{W})=0,
    $$
    so $\tilde{W}$ is not a top form on $(\mathbb{A}^d)^{ I-\{n\}}$. Hence $$
    W((\vec{\Gamma},\mmm,l),\beta e^{\sum\limits_{i=1}^{n}(z_{i}|\lambda_{i})})=\int_{(\mathbb{A}^{d})^{ I-\{n\}}}\tilde{W}=0.
    $$

    Now we prove (2). If $\vec{\Gamma}$ is connected, by using coordinates $\{\tilde{z}_{i}\}$, we have
    \begin{align*}
        \tilde{W}&=e^{-\sum\limits_{i=1}^{n-1}\sum\limits_{j=1}^{n-1}(\tilde{z}_{i}| \bar{\tilde{z}}_{j})M_{\vec{\Gamma}}(t)_{ij}}p_{(\vec{\Gamma},\mmm,l)}(-dy_{e}^{\mmm(e)},-(\tilde{z}_{e}| dy_{e})dy_{e}^{\mmm(e)})e^{\sum\limits_{i=1}^{n-1}(\tilde{z}_{i}|\lambda_{i})+(\tilde{z}_{n}|\sum\limits_{i=1}^{n}\lambda_{i})}\otimes \beta.
    \end{align*}

    Since $\vec{\Gamma}$ is connected, $M_{\vec{\Gamma}}(t)$ is invertible by Proposition \ref{Minverse}.
    As our integral over $(\mathbb{A}^d)^{ I-\{n\}}$ is of Gaussian type, we apply Wick's formula to get
    \begin{align*}
        &W((\vec{\Gamma},\mmm,l),\beta e^{\sum\limits_{i=1}^{n}(z_{i}|\lambda_{i})})\\
        &=\int_{(\mathbb{A}^d)^{ I-\{n\}}}e^{-\sum\limits_{i=1}^{n-1}\sum\limits_{j=1}^{n-1}(\tilde{z}_{i}| \bar{\tilde{z}}_{j})M_{\vec{\Gamma}}(t)_{ij}}\prod_{i=1}^{n-1}(d^{d}\tilde{z}_{i}d^{d}\bar{\tilde{z}}_{i})\wedge \iota_{\prod_{i=1}^{n-1}(d^{d}\tilde{z}_{i}d^{d}\bar{\tilde{z}}_{i})}\\
        &\left(p_{(\vec{\Gamma},\mmm,l)}(-dy_{e}^{\mmm(e)},-(\tilde{z}_{e}|dy_{e})dy_{e}^{\mmm(e)})e^{\sum\limits_{i=1}^{n-1}(\tilde{z}_{i}|\lambda_{i})+(\tilde{z}_{n}|\sum\limits_{i=1}^{n}\lambda_{i})}\otimes \beta\right)\\
        &=(-2\pi i)^{d(n-1)}\frac{1}{(\det M_{\vec{\Gamma}}(t))^{d}}\left(
        \iota_{\prod_{i=1}^{n-1}(d^{d}\tilde{z}_{i}d^{d}\bar{\tilde{z}}_{i})}
        \circ
        e^{\sum\limits_{\scriptstyle\sss=1}^{d}\sum\limits_{i=1}^{n-1}\sum\limits_{j=1}^{n-1} M^{-1}_{\vec{\Gamma}}(t)_{ij}\partial_{\tilde{z}_{i}^{\scriptstyle\sss}}\partial_{\bar{\tilde{z}}_{j}^{\scriptstyle\sss}}}
        \right.\\
        &\left.\left.p_{(\vec{\Gamma},\mmm,l)}(-dy_{e}^{\mmm(e)},-(\tilde{z}_{e}| dy_{e})dy_{e}^{\mmm(e)})e^{\sum\limits_{i=1}^{n-1}(\tilde{z}_{i}|\lambda_{i})+(\tilde{z}_{n}|\sum\limits_{i=1}^{n}\lambda_{i})}\otimes \beta\right)\right|_{\tilde{z}_{i}=0,1\leq i\leq n-1}.
    \end{align*}
To simplify the notation further, it is convenient to choose the following basis for one-forms:
    $$
    \begin{cases}
        d\tilde{z}_{i}=d\tilde{z}_{i}&\text{for }1\leq i\leq n,\\
        d\bar{\tilde{z}}_{i}=\sum\limits_{j=1}^{n-1}d\left(
        M^{-1}_{\vec{\Gamma}}(t)_{ij}\bar{z'}_{j}
        \right)=\sum\limits_{j=1}^{n-1}\left(
        d(M^{-1}_{\vec{\Gamma}}(t)_{ij})\bar{z'}_{j}+M^{-1}_{\vec{\Gamma}}(t)_{ij}d\bar{z'}_{j}
        \right) &\text{for }1\leq i\leq n-1,\\
        d\bar{\tilde{z}}_{n}=d\bar{\tilde{z}}_{n}\\
        dt_{e}=dt_{e}&\text{for }e\in\vec{\Gamma}_{1},
    \end{cases}
    $$
    where $\bar{z'}_{i}$ is a smooth vector-valued function defined by
    $$
    \bar{z'}_{i}=\sum_{j=1}^{n-1}M_{\vec{\Gamma}}(t)_{ij}\bar{z}_{j}.
    $$
    The dual basis for vectors transforms as follows:
    $$
    \begin{cases}
        \partial_{\tilde{z}_{i}}\rightarrow \partial_{\tilde{z}_{i}}&\text{for }1\leq i\leq n,\\
        \partial_{\bar{\tilde{z}}_{i}}\rightarrow \sum\limits_{j=1}^{n-1}
        M_{\vec{\Gamma}}(t)_{ij}\partial_{\bar{z'}_{j}}
         &\text{for }1\leq i\leq n-1,\\
        \partial_{\bar{\tilde{z}}_{n}}\rightarrow \partial_{\bar{\tilde{z}}_{n}}\\
        \partial_{t_{e}}\rightarrow\partial_{t_{e}}-\sum\limits_{i,j,k\in  I-\{n\}}M_{\vec{\Gamma}}(t)_{ik}\partial_{t_{e}}(M^{-1}_{\vec{\Gamma}}(t)_{kj})(\bar{z'}_{j}|\partial_{\bar{z'}_{i}})&\text{for }e\in\vec{\Gamma}_{1}.

    \end{cases}
    $$

    Notice that
    $$
    \iota_{\prod_{i=1}^{n-1}(d^{d}\tilde{z}_{i}d^{d}\bar{\tilde{z}}_{i})}=(\det M_{\vec{\Gamma}}(t))^{d}\iota_{\prod_{i=1}^{n-1}(d^{d}\tilde{z}_{i}d^{d}\bar{z'}_{i})},
    $$
    so the Feynman graph integrand becomes
    \begin{align*}
        &W((\vec{\Gamma},\mmm,l),\beta e^{\sum\limits_{i=1}^{n}(z_{i}|\lambda_{i})})\\
        &=(-2\pi i)^{d(n-1)}\left(
        \iota_{\prod_{i=1}^{n-1}(d^{d}\tilde{z}_{i}d^{d}\bar{z'}_{i})}
        \circ
        e^{\sum\limits_{\scriptstyle\sss=1}^{d}\sum\limits_{i=1}^{n-1}\partial_{\tilde{z}_{i}^{\scriptstyle\sss}}\partial_{\bar{z'}_{i}^{\scriptstyle\sss}}}
        \right.\\
        &\left.\left.p_{(\vec{\Gamma},\mmm,l)}(-dy_{e}^{\mmm(e)},-(\tilde{z}_{e}| dy_{e})dy_{e}^{\mmm(e)})e^{\sum\limits_{i=1}^{n-1}(\tilde{z}_{i}|\lambda_{i})+(\tilde{z}_{n}|\sum\limits_{i=1}^{n}\lambda_{i})}\otimes \beta\right)\right|_{\tilde{z}_{i}=0,1\leq i\leq n-1}.
    \end{align*}

    Notice that, when applied to the exponential, we can trade multiplication by the coordinate function $\tilde{z}^{\sss}_e$
    with the derivative with respect to $\lambda_e^{\sss}=\lambda_{t(e)}^{\sss}-\lambda_{h(e)}^{\sss}$:
    $$
    \tilde{z}^{\sss}_{e}e^{\sum\limits_{i=1}^{n-1}(\tilde{z}_{i}|\lambda_{i})+(\tilde{z}_{n}|\sum\limits_{i=1}^{n}\lambda_{i})}=\partial_{\lambda_{e}^{\scriptstyle\sss}}e^{\sum\limits_{i=1}^{n-1}(\tilde{z}_{i}|\lambda_{i})+(\tilde{z}_{n}|\sum\limits_{i=1}^{n}\lambda_{i})}=(\partial_{\lambda_{t(e)}^{\scriptstyle\sss}}-\partial_{\lambda_{h(e)}^{\scriptstyle\sss}})e^{\sum\limits_{i=1}^{n-1}(\tilde{z}_{i}|\lambda_{i})+(\tilde{z}_{n}|\sum\limits_{i=1}^{n}\lambda_{i})},
    $$
    It follows that
    \begin{align*}
        &p_{(\vec{\Gamma},\mmm,l)}(-dy_{e}^{\mmm(e)},-(\tilde{z}_{e}| dy_{e})dy_{e}^{\mmm(e)})e^{\sum\limits_{i=1}^{n-1}(\tilde{z}_{i}|\lambda_{i})+(\tilde{z}_{n}|\sum\limits_{i=1}^{n}\lambda_{i})}\\
        &=
        p_{(\vec{\Gamma},\mmm,l)}(-dy_{e}^{\mmm(e)},-(\sum_{\scriptstyle\sss=1}^{d}\partial_{\lambda_{e}^{\scriptstyle\sss}}\circ dy_{e}^{\sss})\circ dy_{e}^{\mmm(e)})e^{\sum\limits_{i=1}^{n-1}(\tilde{z}_{i}|\lambda_{i})+(\tilde{z}_{n}|\sum\limits_{i=1}^{n}\lambda_{i})}.
    \end{align*}
    We also notice that, as operators, one has
    \begin{align*}
        e^{\sum\limits_{\scriptstyle\sss=1}^{d}\sum\limits_{i=1}^{n-1}\partial_{\tilde{z}_{i}^{\scriptstyle\sss}}\partial_{\bar{z'}_{i}^{\scriptstyle\sss}}}\circ dy_{e}^{\sss}=d\tilde{y}_{e}^{\sss}\circ
        e^{\sum\limits_{\scriptstyle\sss=1}^{d}\sum\limits_{i=1}^{n-1}\partial_{\tilde{z}_{i}^{\scriptstyle\sss}}\partial_{\bar{z'}_{i}^{\scriptstyle\sss}}},
    \end{align*}
    where
    $$
    d\tilde{y}_{e}^{\sss}=\sum_{j=1}^{n-1}\left(
        d\left(d^{-1}_{\vec{\Gamma}}(t)_{ej}\bar{z'}_{j}^{\sss}\right)+d\left(d^{-1}_{\vec{\Gamma}}(t)_{ej}\right)\partial_{\tilde{z}_{j}^{\scriptstyle\sss}}
        \right)
    $$
    Thus we have
    \begin{align*}
        &e^{\sum\limits_{\scriptstyle\sss=1}^{d}\sum\limits_{i=1}^{n-1}\partial_{\tilde{z}_{i}^{\scriptstyle\sss}}\partial_{\bar{z'}_{i}^{\scriptstyle\sss}}}\circ
        p_{(\vec{\Gamma},\mmm,l)}(-dy_{e}^{\mmm(e)},-(\sum\limits_{\scriptstyle\sss=1}^{d}\partial_{\lambda_{e}^{\scriptstyle\sss}}\circ dy_{e}^{\sss})\circ dy_{e}^{\mmm(e)})\\
        &=
        p_{(\vec{\Gamma},\mmm,l)}(-d\tilde{y}_{e}^{\mmm(e)},-(\sum\limits_{\scriptstyle\sss=1}^{d}\partial_{\lambda_{e}^{\scriptstyle\sss}}\circ d\tilde{y}_{e}^{\sss})\circ d\tilde{y}_{e}^{\mmm(e)})\circ
        e^{\sum\limits_{\scriptstyle\sss=1}^{d}\sum\limits_{i=1}^{n-1}\partial_{\tilde{z}_{i}^{\scriptstyle\sss}}\partial_{\bar{z'}_{i}^{\scriptstyle\sss}}}.
    \end{align*}
    Applying these relations, the Feynman graph integrand becomes
    \begin{align*}
        &W((\vec{\Gamma},\mmm,l),\beta e^{\sum\limits_{i=1}^{n}(z_{i}|\lambda_{i})})\\
        &=(-2\pi i)^{d(n-1)}\left(
        \iota_{\prod_{i=1}^{n-1}(d^{d}\tilde{z}_{i}d^{d}\bar{z'}_{i})}\circ p_{(\vec{\Gamma},\mmm,l)}(-d\tilde{y}_{e}^{\mmm(e)},-(\sum_{\scriptstyle\sss=1}^{d}\partial_{\lambda_{e}^{\scriptstyle\sss}}\circ d\tilde{y}_{e}^{\sss})\circ d\tilde{y}_{e}^{\mmm(e)})
        \right.\\
        &\left.\left.\circ e^{\sum\limits_{i=1}^{n-1}(\tilde{z}_{i}|\lambda_{i})+(\tilde{z}_{n}|\sum\limits_{i=1}^{n}\lambda_{i})}(1\otimes \beta)\right)\right|_{\tilde{z}_{i}=0,1\leq i\leq n-1}.
    \end{align*}
   The factor $e^{\sum\limits_{\scriptstyle\sss=1}^{d}\sum\limits_{i=1}^{n-1}\partial_{\tilde{z}_{i}^{\scriptstyle\sss}}\partial_{\bar{z'}_{i}^{\scriptstyle\sss}}}$ disappears since $e^{\sum\limits_{i=1}^{n-1}(\tilde{z}_{i}|\lambda_{i})+(\tilde{z}_{n}|\sum\limits_{i=1}^{n}\lambda_{i})}(1\otimes \beta)$ is holomorphic. We further notice that
    $$
    d\tilde{y}_{e}^{\sss}\circ e^{\sum\limits_{i=1}^{n-1}(\tilde{z}_{i}|\lambda_{i})+(\tilde{z}_{n}|\sum\limits_{i=1}^{n}\lambda_{i})}=e^{\sum\limits_{i=1}^{n-1}(\tilde{z}_{i}|\lambda_{i})+(\tilde{z}_{n}|\sum\limits_{i=1}^{n}\lambda_{i})}\circ d\tilde{\tilde{y}}_{e}^{\sss},
    $$
    where
    $$
    d\tilde{\tilde{y}}_{e}^{\sss}=d\tilde{y}_{e}^{\sss}+\sum_{j=1}^{n-1}d\left(d^{-1}_{\vec{\Gamma}}(t)_{ej}\right)\lambda_{j}^{\sss}.
    $$
    Then we have
    \begin{align*}
        &p_{(\vec{\Gamma},\mmm,l)}(-d\tilde{y}_{e}^{\mmm(e)},-(\sum_{\scriptstyle\sss=1}^{d}\partial_{\lambda_{e}^{\scriptstyle\sss}}\circ d\tilde{y}_{e}^{\sss})\circ d\tilde{y}_{e}^{\mmm(e)})\circ e^{\sum\limits_{i=1}^{n-1}(\tilde{z}_{i}|\lambda_{i})+(\tilde{z}_{n}|\sum\limits_{i=1}^{n}\lambda_{i})}\\
        &=e^{\sum\limits_{i=1}^{n-1}(\tilde{z}_{i}|\lambda_{i})+(\tilde{z}_{n}|\sum\limits_{i=1}^{n}\lambda_{i})}\circ p_{(\vec{\Gamma},\mmm,l)}(-d\tilde{\tilde{y}}_{e}^{\mmm(e)},-(\sum_{\scriptstyle\sss=1}^{d}(\partial_{\lambda_{e}^{\scriptstyle\sss}}+\tilde{z}_{e}^{\sss})\circ d\tilde{\tilde{y}}_{e}^{\sss})\circ d\tilde{\tilde{y}}_{e}^{\mmm(e)}).
    \end{align*}

    Finally, we get
    \begin{align*}
        &W((\vec{\Gamma},\mmm,l),\beta e^{\sum\limits_{i=1}^{n}(z_{i}|\lambda_{i})})\\
        &=(-2\pi i)^{d(n-1)}\left(
        \iota_{\prod_{i=1}^{n-1}(d^{d}\tilde{z}_{i}d^{d}\bar{z'}_{i})}\circ p_{(\vec{\Gamma},\mmm,l)}(-d\tilde{\tilde{y}}_{e}^{\mmm(e)},-(\sum_{\scriptstyle\sss=1}^{d}(\partial_{\lambda_{e}^{\scriptstyle\sss}}+\tilde{z}_{e}^{\sss})\circ d\tilde{\tilde{y}}_{e}^{\sss})\circ d\tilde{\tilde{y}}_{e}^{\mmm(e)})
        \right.\\
        &\left.\left.(1\otimes \beta)\right)\right|_{\tilde{z}_{i}=0,1\leq i\leq n-1} e^{(\tilde{z}_{n}|\sum\limits_{i=1}^{n}\lambda_{i})}\\
        &=(-2\pi i)^{d(n-1)}\left(
        \iota_{\prod_{i=1}^{n-1}(d^{d}\tilde{z}_{i}d^{d}\bar{\tilde{z}}_{i})}\circ p_{(\vec{\Gamma},\mmm,l)}(-d\hat{y}_{e}^{\mmm(e)},-(\sum_{\scriptstyle\sss=1}^{d}(\partial_{\lambda_{e}^{\scriptstyle\sss}}+\tilde{z}_{e}^{\sss})\circ d\hat{y}_{e}^{\sss})\circ d\hat{y}_{e}^{\mmm(e)})
        \right.\\
        &\left.\left.(1\otimes \beta)\right)\right|_{\tilde{z}_{i}=0,1\leq i\leq n-1} e^{(\tilde{z}_{n}|\sum\limits_{i=1}^{n}\lambda_{i})}.\\
    \end{align*}
\end{proof}

We have shown that the expression $W((\vec{\Gamma},\mmm,l),\beta e^{\sum\limits_{i=1}^{n}(z_{i}|\lambda_{i})})$ is a
well-defined differential form in the variables $t = (t_e) \in (0,+\infty)^{|\Gamma_1|}$.
Now we prove that this integrand is integrable over
$$
S^{+}((0,+\infty)^{|\Gamma_{1}|})=\{(t_{e_{1}},\dots,t_{e_{|\Gamma_{1}|}})\in(0,+\infty)^{|\Gamma_{1}|}|\sum_{i=1}^{|\Gamma_{1}|}t_{i}^{2}=1\}.
$$
We use the compactification technique of Schwinger spaces from \cite{Wang:2024tjf,wang2024feynman}.
We refer to $(0,+\infty)^{\vec{\Gamma}_{1}}$ as the \textit{Schwinger space} of the graph $\vec{\Gamma}$.
In order to integrate, we fix the orientation by the following
    $$
    \int_{(0,L)^{|\Gamma_{1}|}}\prod_{e\in\vec{\Gamma}_{1}}dt_{e}=L^{|\Gamma_{1}|},
    $$
    where $L>0$.

There is a natural partial compactification of $(0,+\infty)^{|\Gamma_{1}|}$, which is constructed by iterated real
blow up along corners of $[0,+\infty)^{|\Gamma_{1}|}$. We collect basic properties of the partially compactified
Schwinger spaces in the Appendix \ref{Schwinger spaces}. We use $\widetilde{[0,+\infty)}^{|\Gamma_{1}|}$ to denote the partially compactified Schwinger space of $\vec{\Gamma}$.
One of the key properties of this partial compactification is that the closure of $S^{+}((0,+\infty)^{|\Gamma_{1}|})$
in $\widetilde{[0,+\infty)}^{|\Gamma_{1}|}$, which we denote by
    $$
    \bar{S}^{+}((0,+\infty)^{|\Gamma_{1}|}),
    $$
    is compact.
    Additionally, we will use the fact that the graphical Green's function
    \[
      d^{-1}_{\vec{\Gamma}}(t)_{ei}, \quad i\in \vec{\Gamma}_{0}-\{n\},\quad e\in \vec{\Gamma}_{1} ,
    \]
    which is originally defined on $(0,+\infty)^{|\Gamma_1|}$, can be extended to a smooth function on $\widetilde{[0,
    +\infty)}^{\vec{\Gamma}_{1}}$, see Lemma \ref{extended functions} in the appendix.

\begin{prop}
    Let $(\vec{\Gamma},\mmm,l)$ be a Feynman graph and denote $\Gamma_{0}= I$.
    Given
    $$
    \beta\in (\omega_{\mathbb{A}^{d}}[d])^{\boxtimes I}
    $$
    and
    $$
    e^{\sum\limits_{i=1}^{n}(z_{i}|\lambda_{i})}=e^{\sum\limits_{\scriptstyle\sss=1}^{d}\sum\limits_{i=1}^{n}z^{\scriptstyle\sss}_{i}\lambda^{\scriptstyle\sss}_{i}}\in C^{\infty}((\mathbb{A}^{d})^{ I}\times (\mathbb{A}^{d})^{ I}),
    $$
    the Feynman graph integrand
    $$
    W((\vec{\Gamma},\mmm,l),\beta e^{\sum\limits_{i=1}^{n}(z_{i}|\lambda_{i})})
    $$
    can be extended to a smooth differential form on $\widetilde{[0,+\infty)}^{\vec{\Gamma}_{1}}$. In particular,
    $$
    \int_{S^{+}((0,+\infty)^{\vec{\Gamma}_{1}})}W((\vec{\Gamma},\mmm,l),\beta e^{\sum\limits_{i=1}^{n}(z_{i}|\lambda_{i})})
    $$
    is convergent.
\end{prop}
\begin{proof}
    By Proposition \ref{explicit formula}, we have an explicit formula for $
    W((\vec{\Gamma},\mmm,l),\beta e^{\sum\limits_{i=1}^{n}(z_{i}|\lambda_{i})})
    $. Since $d^{-1}_{\vec{\Gamma}}(t)_{ei}$ is smooth on $\widetilde{[0,+\infty)}^{\vec{\Gamma}_{1}}$, the operator $d\hat{y}_{e}^{\sss}$ is also smooth. So $
    W((\vec{\Gamma},\mmm,l),\beta e^{\sum\limits_{i=1}^{n}(z_{i}|\lambda_{i})})
    $ can be extended to a smooth differential form. Finally, since $\bar{S}^{+}((0,+\infty)^{\vec{\Gamma}_{1}})$ is compact, the integral
    $$
    \int_{S^{+}((0,+\infty)^{\vec{\Gamma}_{1}})}W((\vec{\Gamma},\mmm,l),\beta e^{\sum\limits_{i=1}^{n}(z_{i}|\lambda_{i})})=
    \int_{\bar{S}^{+}((0,+\infty)^{\vec{\Gamma}_{1}})}W((\vec{\Gamma},\mmm,l),\beta e^{\sum\limits_{i=1}^{n}(z_{i}|\lambda_{i})})
    $$
    is convergent.
\end{proof}

\subsubsection{Properties of Feynman graph integrals}

The next result is about the scaling invariance of the Feynman graph integrals that we are studying.

\begin{prop}\label{useful property}
    Let $(\vec{\Gamma},\mmm,l)$ be a Feynman graph such that $\vec{\Gamma}_{0}= I$. Given
    $$
    \beta\in (\omega_{\mathbb{A}^{d}}[d])^{\boxtimes I}
    $$
    and
    $$
e^{\sum\limits_{i=1}^{n}(z_{i}|\lambda_{i})}=e^{\sum\limits_{\scriptstyle\sss=1}^{d}\sum\limits_{i=1}^{n}z^{\scriptstyle\sss}_{i}\lambda^{\scriptstyle\sss}_{i}}\in C^{\infty}((\mathbb{A}^{d})^{ I}\times (\mathbb{A}^{d})^{ I}),
    $$
    then we have the following consequences:
    \begin{enumerate}
        \item
        Let $\mathbb{R}_{+}$ act on Schwinger space by rescaling
        $$
        \lambda\cdot(t_{e_{1}},\dots,t_{e_{|\Gamma_{1}|}})=(\lambda\cdot t_{e_{1}},\dots,\lambda\cdot t_{e_{|\Gamma_{1}|}}),
        $$
        where $\lambda\in \mathbb{R}_+$, $(t_{e_{1}},\dots,t_{e_{|\Gamma_{1}|}})\in (0,+\infty)^{\vec{\Gamma}_{1}}$.
        Then, the Feynman graph integrand
        $$
        W((\vec{\Gamma},\mmm,l),\beta e^{\sum\limits_{i=1}^{n}(z_{i}|\lambda_{i})})
        $$
        is $\mathbb{R}_+$-invariant and hence induces a smooth differential form on $(0,+\infty)^{\vec{\Gamma}_{1}}/\mathbb{R}_+$.
        \item Let $S\subset(0,+\infty)^{\vec{\Gamma}_{1}}$ be a submanifold, such that the natural map $S\rightarrow(0,+\infty)^{\vec{\Gamma}_{1}}/\mathbb{R}_+$ is a diffeomorphism. Then we have
        $$
        \int_{S}W((\vec{\Gamma},\mmm,l),\beta e^{\sum\limits_{i=1}^{n}(z_{i}|\lambda_{i})})=
        \int_{S^{+}((0,+\infty)^{\vec{\Gamma}_{1}})}W((\vec{\Gamma},\mmm,l),\beta e^{\sum\limits_{i=1}^{n}(z_{i}|\lambda_{i})}).
        $$
    \end{enumerate}
\end{prop}
\begin{proof}
  Part (2) is a direct consequence of (1), so we prove (1).
We need to show that the weight is an equivariant, or basic, differential form with respect to this group action.
    Let $V=V_{1}+V_{2}$, where $V_{1}=\sum\limits_{\scriptstyle\sss=1}^{d}\sum\limits_{i=1}^{n-1}\bar{\tilde{z}}_{i}^{\sss}
    \partial_{\bar{\tilde{z}}_{i}^{\scriptstyle\sss}}$, $V_{2}=\sum\limits_{e\in\vec{\Gamma}_{1}}t_{e}\partial_{t_{e}}$.
    Notice that $V_2$ is the infinitesimal generator for the $\mathbb{R}_+$-action as defined in the statement.

    As before, for
    shorthand notation define
    $$
    \tilde{W}=p_{(\vec{\Gamma},\mmm,l)}(-e^{-(z_{e}| y_{e})}dy_{e}^{\mmm(e)},-e^{-(z_{e}| y_{e})}(z_{e}|dy_{e})dy_{e}^{\mmm(e)})e^{\sum\limits_{i=1}^{n}(z_{i}|\lambda_{i})}\otimes \beta.
    $$
    By Proposition \ref{Lie and interior product}, we have
    $$
    \begin{cases}
        L_V\tilde{W}=0,\\
        \iota_V\tilde{W}=0.
    \end{cases}
    $$
    Thus
    \begin{align*}
        L_{V_2}W((\vec{\Gamma},\mmm,l),\beta e^{\sum\limits_{i=1}^{n}(z_{i}|\lambda_{i})})=\int_{(\mathbb{A}^{d})^{ I-\{n\}}}L_{V_2}\tilde{W}=-\int_{(\mathbb{A}^{d})^{ I-\{n\}}}L_{V_1}\tilde{W}=0
    \end{align*}
    by Stokes' theorem.
    Similarly,
    $$
    \iota_{V_2}W((\vec{\Gamma},\mmm,l),\beta e^{\sum\limits_{i=1}^{n}(z_{i}|\lambda_{i})})=\pm\int_{(\mathbb{A}^{d})^{ I-\{n\}}}\iota_{V_1}\tilde{W}=0
    $$
    by type reasons.
\end{proof}
We turn to the dependency of the integration $\int_{S^{+}((0,+\infty)^{\vec{\Gamma}_{1}})}W((\vec{\Gamma},\mmm,l),\beta
e^{\sum\limits_{i=1}^{n}(z_{i}|\lambda_{i})})$ on the actual Feynman graph $(\vec{\Gamma},\mmm,l)$.
This is summarized in the following proposition.
\begin{prop}\label{well-defineness}
        Let $(\vec{\Gamma},\mmm,l)$ be a Feynman graph such that $\vec{\Gamma}_{0}= I$. Given
    $$
    \beta\in (\omega_{\mathbb{A}^{d}}[d])^{\boxtimes I}
    $$
    and
    $$
    e^{\sum\limits_{i=1}^{n}(z_{i}|\lambda_{i})}=e^{\sum\limits_{\scriptstyle\sss=1}^{d}\sum\limits_{i=1}^{n}z^{\scriptstyle\sss}_{i}\lambda^{\scriptstyle\sss}_{i}}\in C^{\infty}((\mathbb{A}^{d})^{ I}\times (\mathbb{A}^{d})^{ I}),
    $$
    then we have the following consequences:
    \begin{enumerate}
        \item Let $1\leq l_{1},l_{2}< l$ and $l\leq l'_{1},l'_{2}\leq |\Gamma_{1}|$. We use $\sigma_{l_{1}l_{2}}(\vec{\Gamma})$ ($\sigma_{l'_{1}l'_{2}}(\vec{\Gamma})$) to denote the Feynman graph $\vec{\Gamma}$ with the order of $e_{l_{1}},e_{l_{2}}\in\vec{\Gamma}_{1}$ ($e_{l'_{1}},e_{l'_{2}}\in\vec{\Gamma}_{1}$) interchanged. Then we have
        $$
        \begin{cases}
            \int_{S^{+}((0,+\infty)^{\sigma_{l_{1}l_{2}}(\vec{\Gamma}_{1})})}W((\sigma_{l_{1}l_{2}}(\vec{\Gamma}),\mmm,l),\beta e^{\sum\limits_{i=1}^{n}(z_{i}|\lambda_{i})})=
        \int_{S^{+}((0,+\infty)^{\vec{\Gamma}_{1}})}W((\vec{\Gamma},\mmm,l),\beta e^{\sum\limits_{i=1}^{n}(z_{i}|\lambda_{i})})\\
        \int_{S^{+}((0,+\infty)^{\sigma_{l'_{1}l'_{2}}(\vec{\Gamma}_{1})})}W((\sigma_{l'_{1}l'_{2}}(\vec{\Gamma}),\mmm,l),\beta e^{\sum\limits_{i=1}^{n}(z_{i}|\lambda_{i})})=
        -\int_{S^{+}((0,+\infty)^{\vec{\Gamma}_{1}})}W((\vec{\Gamma},\mmm,l),\beta e^{\sum\limits_{i=1}^{n}(z_{i}|\lambda_{i})})
        \end{cases}
        $$
        \item For $1\leq \sss\leq d$ and $i,j\in \vec{\Gamma}_{0}$, let $(\vec{\Gamma}',\mmm^{\sss},l)$ be the Feynman graph which satisfies the following:
        \begin{itemize}
            \item $\vec{\Gamma}\subset\vec{\Gamma}'$ is a directed subgraph.
            \item $\vec{\Gamma}'_{0}=\vec{\Gamma}_{0}$.
            \item $\vec{\Gamma}'_{1}=\{e_{0}\}\cup\vec{\Gamma}_{1}$, where $t(e_{0})=i$, $h(e_{0})=j$.
            \item $\mmm^{\sss}(e_{0})=\sss$.
        \end{itemize}
        Then we have
        \begin{align}\label{well-defined formula}
        \sum_{\scriptstyle\sss=1}^{d}\int_{S^{+}((0,+\infty)^{\vec{\Gamma}'_{1}})}W((\vec{\Gamma}',\mmm^{\sss},l),z_{ij}^{\sss}\beta e^{\sum\limits_{i=1}^{n}(z_{i}|\lambda_{i})})=
        \int_{S^{+}((0,+\infty)^{\vec{\Gamma}_{1}})}W((\vec{\Gamma},\mmm,l),\beta e^{\sum\limits_{i=1}^{n}(z_{i}|\lambda_{i})}).
        \end{align}
        \item For $1\leq \sss\leq d$ and $i,j\in \vec{\Gamma}_{0}$, let $(\vec{\Gamma}'',\mmm^{\sss},l)$ be the Feynman graph which satisfies the following:
        \begin{itemize}
            \item $\vec{\Gamma}\subset\vec{\Gamma}''$ is a directed subgraph.
            \item $\vec{\Gamma}''_{0}=\vec{\Gamma}_{0}$.
            \item $\vec{\Gamma}''_{1}=\vec{\Gamma}_{1}\cup\{e_{|\vec{\Gamma}_{1}|+1}\}$, where $t(e_{|\vec{\Gamma}_{1}|+1})=i$, $h(e_{|\vec{\Gamma}_{1}|+1})=j$.
            \item $\mmm^{\sss}(e_{|\vec{\Gamma}_{1}|+1})=\sss$.
        \end{itemize}
        Then we have
        $$
        \sum_{\scriptstyle\sss=1}^{d}\int_{S^{+}((0,+\infty)^{\vec{\Gamma}''_{1}})}W((\vec{\Gamma}'',\mmm^{\sss},l),z_{ij}^{\sss}\beta e^{\sum\limits_{i=1}^{n}(z_{i}|\lambda_{i})})=0.
        $$
    \end{enumerate}
\end{prop}
\begin{proof}
    Part (1) is a direct observation.
    We prove part (2). First assume $\vec{\Gamma}_{1}\neq\emptyset$. To simplify notation, we define
    $$
    \tilde{W}=p_{(\vec{\Gamma},\mmm,l)}(-e^{-(z_{e}|y_{e})}dy_{e}^{\mmm(e)},-e^{-(z_{e}| y_{e})}(z_{e}|dy_{e})dy_{e}^{\mmm(e)})e^{\sum\limits_{i=1}^{n}(z_{i}|\lambda_{i})}\otimes \beta.
    $$
    Consider the submanifold
    $$
    S=S^{+}((0,+\infty)^{\vec{\Gamma}_{1}})\times(0,+\infty)\subset S^{+}((0,+\infty)^{\vec{\Gamma}'_{1}}).
    $$
    Notice that the natural map from $S$ to $(0,+\infty)^{\vec{\Gamma}'_{1}}/\mathbb{R}_+$ is a diffeomorphism. Hence
    \begin{align*}
        &\sum_{\scriptstyle\sss=1}^{d}\int_{S^{+}((0,+\infty)^{\vec{\Gamma}'_{1}})}W((\vec{\Gamma}',\mmm^{\sss},l),z_{ij}^{\sss}\beta e^{\sum\limits_{i=1}^{n}(z_{i}|\lambda_{i})})\\
        &=
        \sum_{\scriptstyle\sss=1}^{d}\int_{S^{+}((0,+\infty)^{\vec{\Gamma}_{1}})}\int_{(0,+\infty)}W((\vec{\Gamma}',\mmm^{\sss},l),z_{ij}^{\sss}\beta e^{\sum\limits_{i=1}^{n}(z_{i}|\lambda_{i})})\\
        &=\sum_{\scriptstyle\sss=1}^{d}\int_{S^{+}((0,+\infty)^{\vec{\Gamma}_{1}})}\int_{(\mathbb{A}^{d})^{ I-\{n\}}}z_{ij}^{\sss}x_{ij}^{\sss}\tilde{W}\\
        &=
        \int_{S^{+}((0,+\infty)^{\vec{\Gamma}_{1}})}W((\vec{\Gamma},\mmm,l),\beta e^{\sum\limits_{i=1}^{n}(z_{i}|\lambda_{i})}).
    \end{align*}
    When $\vec{\Gamma}_{1}=\emptyset$, $|\Gamma_{0}|\geq3$, $\vec{\Gamma}$ and $\vec{\Gamma}'$ are disconnected, so both sides of formula~\eqref{well-defined formula} are zero. When $\vec{\Gamma}_{1}=\emptyset$, $i=1$, $j=2$, $\vec{\Gamma}_{0}=\{1,2\}$, we have
    \begin{align*}
        &\sum_{\scriptstyle\sss=1}^{d}\int_{S^{+}((0,+\infty)^{\vec{\Gamma}'_{1}})}W((\vec{\Gamma}',\mmm^{\sss},l),z_{ij}^{\sss}\beta e^{\sum\limits_{i=1}^{n}(z_{i}|\lambda_{i})})\\
        &=
        \left.\left(
        -\int_{\mathbb{A}^{d}}e^{-(z_{12}| y_{12})}(z_{12}| dy_{12})\beta e^{(z_{1}| \lambda_{1})+(z_{2}|\lambda_{2})}
        \right)\right|_{t_{12}=1}\\
        &=
        \left.\left(
        \int_{\mathbb{A}^{d}}\bar{\partial}\left(e^{-(z_{12}| y_{12})}\beta e^{(z_{1}| \lambda_{1})+(z_{2}|\lambda_{2})}
        \right)\right)\right|_{t_{12}=1}\\
        &=0.
    \end{align*}
    Since $\vec{\Gamma}$ is disconnected,
    \begin{align*}
        &\int_{S^{+}((0,+\infty)^{\vec{\Gamma}_{1}})}W((\vec{\Gamma},\mmm,l),\beta e^{\sum\limits_{i=1}^{n}(z_{i}|\lambda_{i})})\\
        &=0\\
        &=
        \sum_{\scriptstyle\sss=1}^{d}\int_{S^{+}((0,+\infty)^{\vec{\Gamma}'_{1}})}W((\vec{\Gamma}',\mmm^{\sss},l),z_{ij}^{\sss}\beta e^{\sum\limits_{i=1}^{n}(z_{i}|\lambda_{i})}).
    \end{align*}
    Finally, part (3) follows by similar arguments.
\end{proof}

\subsection{Higher residues from Feynman graphs}

Finally, we arrive at the definition of the residue on the model $\mathbf{J}^{ I}_{\mathbb{A}^{d}}((\omega_{\mathbb{A}
^{d}}[d])^{\boxtimes I})$ for configuration space.

\begin{defn}\label{dfn:res}
    Let
    $$
    \alpha=p(x_{ij}^{\sss},\mathbf{d}x_{ij}^{\sss})\otimes\beta\in \mathbf{J}_{ I}^{\mathbb{A}^{d}}((\omega_{\mathbb{A}^{d}}[d])^{\boxtimes I}),
    $$
    where $p(x_{ij}^{\sss},\mathbf{d}x_{ij}^{\sss})$ is a monomial with coefficient $1$.
    Let
    $$
    e^{\sum\limits_{i=1}^{n}(z_{i}|\lambda_{i})}=e^{\sum\limits_{\scriptstyle\sss=1}^{d}\sum\limits_{i=1}^{n}z^{\scriptstyle\sss}_{i}\lambda^{\scriptstyle\sss}_{i}}\in C^{\infty}((\mathbb{A}^{d})^{ I}\times (\mathbb{A}^{d})^{ I}),
    $$
    then the \textit{residue} of $\alpha e^{\sum\limits_{i=1}^{n}(z_{i}|\lambda_{i})}$ is defined by
    \begin{align*}
        &\frac{-1}{(-2\pi i)^{d(n-1)}}\oint_{z_{1},\dots z_{n-1}=z_{n}}\alpha e^{\sum\limits_{i=1}^{n}(z_{i}|\lambda_{i})}\\
        &\define\frac{(-1)^{\frac{1}{2}(|\Gamma_{1}|-l)(|\Gamma_{1}|-l+1)+1}}{(-2\pi i)^{d(n-1)}}\int_{S^{+}((0,+\infty)^{\vec{\Gamma}_{1}})}W((\vec{\Gamma},\mmm,l),\beta e^{\sum\limits_{i=1}^{n}(z_{i}|\lambda_{i})}),
    \end{align*}
    where $(\vec{\Gamma},\mmm,l)$ is a Feynman graph with corresponding polynomial $p_{(\vec{\Gamma},\mmm,l)}=p(x_{ij}^{\sss},
    \mathbf{d}x_{ij}^{\sss})$, see definition \ref{dfn:poly}.
    We define residues for general elements in $\mathbf{J}^{ I}_{\mathbb{A}^{d}}((\omega_{\mathbb{A}^{d}}[d])^{\boxtimes
    I})$ by linearity.
\end{defn}

We observe that by Proposition \ref{well-defineness}, residues are well-defined.
The following proposition highlights key features of our higher residue defined on the Jouanolou model.
In part, it implies that the formal $\mathcal{D}_{(\mathbb{A}^{d})^{ I}}$-module structure on $\mathbf{J}^{ I}_{\mathbb{A}^{d}}$ coincides with the action by derivatives on the residues we have just defined.

\begin{prop}\label{residue 0}
    Let $\alpha\in \mathbf{J}^{ I}_{\mathbb{A}^{d}}((\omega_{\mathbb{A}^{d}}[d])^{\boxtimes I})$, $e^{\sum\limits_{i=1}^{n}(z_{i}|\lambda_{i})}\in C^{\infty}((\mathbb{A}^{d})^{ I}\times (\mathbb{A}^{d})^{ I})$, and $1\leq \sss\leq d$. We have

    \begin{enumerate}
        \item If $i\in I-\{n\}$, then
        $$
        \frac{-1}{(-2\pi i)^{d(n-1)}}\oint_{z_{1},\dots z_{n-1}=z_{n}}\left(\alpha e^{\sum\limits_{i=1}^{n}(z_{i}|\lambda_{i})}\right)\cdot \partial_{z_{i}^{\scriptstyle\sss}}=0.
        $$
        \item If $i=n$, then
        \begin{align*}
            &\frac{-1}{(-2\pi i)^{d(n-1)}}\oint_{z_{1},\dots z_{n-1}=z_{n}}\left(\alpha e^{\sum\limits_{i=1}^{n}(z_{i}|\lambda_{i})}\right)\cdot \partial_{z_{i}^{\scriptstyle\sss}}\\
            &=\left(\frac{-1}{(-2\pi i)^{d(n-1)}}\oint_{z_{1},\dots z_{n-1}=z_{n}}\alpha e^{\sum\limits_{i=1}^{n}(z_{i}|\lambda_{i})}\right)\cdot \partial_{z_{n}^{\scriptstyle\sss}}.
        \end{align*}
    \end{enumerate}
    \begin{proof}
        Without loss of generality, we assume $\alpha=p_{(\vec{\Gamma},\mmm,l)}\otimes\beta$, where $(\vec{\Gamma},\mmm,l)$ is a Feynman graph. We can also assume $\vec{\Gamma}_{1}\neq\emptyset$, $t(e_{1})=i$.

        Let $\vec{\Gamma}'\subset\vec{\Gamma}$ be a subgraph, such that $\vec{\Gamma}_{0}'=\vec{\Gamma}_{0}$, $\vec{\Gamma}_{1}'=\vec{\Gamma}_{1}-\{e_{1}\}$. Let
        $$
        \tilde{W}'=p_{(\vec{\Gamma}',\mmm|_{\vec{\Gamma}'},l)}(-e^{-(z_{e}| y_{e})}dy_{e}^{\mmm(e)},-e^{-(z_{e}| y_{e})}(z_{e}| dy_{e})dy_{e}^{\mmm(e)})e^{\sum\limits_{i=1}^{n}(z_{i}|\lambda_{i})}\otimes \beta.
        $$
        When $1<l$, we have
        $$
        \partial_{z_{i}^{\scriptstyle\sss}}x_{e_{1}}^{\mmm(e_{1})}=-x_{e_{1}}^{\sss}x_{e_{1}}^{\mmm(e_{1})}.
        $$
        Let's prove
        \begin{align*}
            &\int_{S^{+}((0,+\infty)^{|\Gamma_{1}|+1})}\int_{(\mathbb{A}^{d})^{ I-\{n\}}}
            e^{-\frac{(z_{e_{1}}| \bar{z}_{e_{1}})}{t_{e_{0}}}}d(\frac{\bar{z}_{e_{1}}^{\sss}}{t_{e_{0}}})e^{-\frac{(z_{e_{1}}|\bar{z}_{e_{1}})}{t_{e_{1}}}}d(\frac{\bar{z}_{e_{1}}^{\mmm(e_{1})}}{t_{e_{1}}})\tilde{W}'\\
            &=
            -\int_{S^{+}((0,+\infty)^{\vec{\Gamma}_{1}})}\int_{(\mathbb{A}^{d})^{ I-\{n\}}}\partial_{z_{i}^{\scriptstyle\sss}}\left(e^{-\frac{(z_{e_{1}}|\bar{z}_{e_{1}})}{t_{e_{1}}}}d(\frac{\bar{z}_{e_{1}}^{\mmm(e_{1})}}{t_{e_{1}}})\right)\tilde{W}'.
        \end{align*}

        When $\vec{\Gamma}_{1}'\neq\emptyset$, let
        $$
        S=S^{+}((0,+\infty)^{\vec{\Gamma}_{1}'})\times(0,+\infty)^{2}\subset(0,+\infty)^{|\Gamma_{1}|+1}.
        $$
        Since the natural map from $S$ to $(0,+\infty)^{|\Gamma_{1}|+1}/\mathbb{R}_+$ is a diffeomorphism, we have
        \begin{align*}
            &\int_{S^{+}((0,+\infty)^{|\Gamma_{1}|+1})}\int_{(\mathbb{A}^{d})^{ I-\{n\}}}
            e^{-\frac{(z_{e_{1}}| \bar{z}_{e_{1}})}{t_{e_{0}}}}d(\frac{\bar{z}_{e_{1}}^{\sss}}{t_{e_{0}}})e^{-\frac{(z_{e_{1}}| \bar{z}_{e_{1}})}{t_{e_{1}}}}d(\frac{\bar{z}_{e_{1}}^{\mmm(e_{1})}}{t_{e_{1}}})\tilde{W}'\\
            &=
            \int_{S^{+}((0,+\infty)^{\vec{\Gamma}'_{1}})}\int_{(\mathbb{A}^{d})^{ I-\{n\}}}\Phi_I(x_{e_{1}}^{\sss})\Phi_I(x_{e_{1}}^{\mmm(e_{1})})\tilde{W}'\\
            &=
            -\int_{S^{+}((0,+\infty)^{\vec{\Gamma}'_{1}})}\int_{(\mathbb{A}^{d})^{ I-\{n\}}}\partial_{z_{i}^{\scriptstyle\sss}}(\Phi_I(x_{e_{1}}^{\mmm(e_{1})}))\tilde{W}'\\
            &=
            -\int_{S^{+}((0,+\infty)^{\vec{\Gamma}_{1}})}\int_{(\mathbb{A}^{d})^{ I-\{n\}}}\partial_{z_{i}^{\scriptstyle\sss}}\left(e^{-\frac{(z_{e_{1}}| \bar{z}_{e_{1}})}{t_{e_{1}}}}d(\frac{\bar{z}_{e_{1}}^{\mmm(e_{1})}}{t_{e_{1}}})\right)\tilde{W}'.
        \end{align*}
        When $\vec{\Gamma}_{1}'=\emptyset$, this can be proved by direct computation.

        Similar arguments prove the case when $l=1$.

        Now, (1) follows from the fact that the integral of a total derivative is zero. By the dominated convergence theorem for derivatives, we can interchange the order of taking derivative and integration. This proves (2).
    \end{proof}
\end{prop}
\subsection{Construction of the unit chiral algebra on $\mathbb{A}^d$} \label{s:unit}

We conclude this section with a construction of an $L_{\infty}$ chiral algebra structure on $\omega_{\mathbb{A}^{d}}[d-1]$ by using residues.
This is precisely the \textit{unit chiral algebra} in dimension $d$.

We first introduce the following definition.
\begin{defn}
    Let $\alpha\in \mathbf{J}^{ I}_{\mathbb{A}^{d}}((\omega_{\mathbb{A}^{d}}[d])^{\boxtimes I})$, the shifted $n$-ary operation
    $$
    \tilde{\mu}_{ I}:\mathbf{J}^{ I}_{\mathbb{A}^{d}}((\omega_{\mathbb{A}^{d}}[d])^{\boxtimes I})\rightarrow \omega_{\mathbb{A}^{d}}[d]\otimes_{\kk[\lambda_{\star}]}\kk[\lambda_{1},\dots,\lambda_{n}]
    $$
    is given by
    $$
    \tilde{\mu}_{ I}(\alpha)=\frac{-e^{-(\lambda_{\star}| w)}}{(-2\pi i)^{d(n-1)}}\left.\left(\oint_{z_{1},\cdots z_{n-1}=z_{n}}\alpha e^{\sum\limits_{i=1}^{n}(z_{i}|\lambda_{i})}\right)\right|_{z_n=w},
    $$
    where $\lambda_{\star}=\sum\limits_{i=1}^{n}\lambda_{i}$.
\end{defn}
\begin{rem}
    By Proposition \ref{explicit formula}, we know $\tilde{\mu}_{ I}(\alpha)$ is a well-defined element of $\omega_{\mathbb{A}^{d}}[d]\otimes_{\kk[\lambda_{\star}]}\kk[\lambda_{1},\dots,\lambda_{n}]$, i.e., $\tilde{\mu}_{ I}(\alpha)$ is a polynomial with respect to $\{\lambda_{i}\}_{i\in I}$.
\end{rem}

As mentioned in Remark \ref{rem:CompareJouanolouPoly}, the Jouanolou model has the advantage that it carries a $\mathrm{GL}_d$-action. Here we give the detailed definition.
\begin{defn}
   The Jouanolou model $\mathbf{J}^{ I}_{\mathbb{A}^{d}}((\omega_{\mathbb{A}^{d}}[d])^{\boxtimes I})$ and the pushforward $\omega_{\mathbb{A}^{d}}[d]\otimes_{\kk[\lambda_{\bullet}]}\kk[\lambda_{1},\dots,\lambda_{n}]$ carry a natural $\mathrm{GL}_d$-action which can be described as follows:
   $$
   (z_i^{\1},\dots,z_i^{\dd})\mapsto  (z_i^{\1},\dots,z_i^{\dd})A,\quad i=1,\dots,n,
   $$
    $$
   (dz_i^{\1},\dots,dz_i^{\dd})\mapsto  (dz_i^{\1},\dots,dz_i^{\dd})A,\quad i=1,\dots,n,
   $$
    $$
   (x_{ij}^{\1},\dots,x_{ij}^{\dd})\mapsto  (x_{ij}^{\1},\dots,x_{ij}^{\dd})(A^{T})^{-1},\quad i,j=1,\dots,n,
   $$
 $$
   (\lambda_i^{\1},\dots,\lambda_i^{\dd})\mapsto  (\lambda_i^{\1},\dots,\lambda_i^{\dd})(A^{T})^{-1},\quad i=\star,1,\dots,n,
   $$
   where $A\in \mathrm{GL}_d$.
\end{defn}
The following proposition says that $\tilde{\mu}_{ I}(\alpha)$ defines a degree 1 $\mathrm{GL}_d$-equivariant chiral operation.
\begin{prop}
    $\tilde{\mu}_{ I}(\alpha)$ defines a cohomological degree $1$ element of
    \[
      \mathrm{Hom}_{\mathcal{D}_{(\mathbb{A}^d)^{I}}}\left(\mathbf{J}_{\mathbb{A}^d}^{I}((\omega_{\mathbb{A}^{d}}[d]))^{\boxtimes I}),\Delta^{I/\{\star\}}_*(\omega_{\mathbb{A}^{d}}[d])\right)
    \]
which is also $\mathrm{GL}_d$-equivariant.
\end{prop}
\begin{proof}
    We first prove that $\tilde{\mu}_{ I}$ is a $\mathcal{D}_{(\mathbb{A}^{d})^{ I}}$-module map.
    First, observe that functions act in the following way:
    \begin{align*}
        &\tilde{\mu}_{ I}(\alpha\cdot z_{i}^{\sss})\\
        &=
        \frac{-e^{-(\lambda_{\star}| w)}}{(-2\pi i)^{d(n-1)}}\left.\left(\oint_{z_{1},\cdots z_{n-1}=z_{n}}z_{i}^{\sss}\alpha e^{\sum\limits_{j=1}^{n}(z_{j}|\lambda_{j})}\right)\right|_{z_n=w}\\
        &=
        \frac{-e^{-(\lambda_{\star}| w)}}{(-2\pi i)^{d(n-1)}}\left.\left(\oint_{z_{1},\cdots z_{n-1}=z_{n}}\alpha \partial_{\lambda_{i}^{\scriptstyle\sss}}e^{\sum\limits_{j=1}^{n}(z_{j}|\lambda_{j})}\right)\right|_{z_n=w}\\
        &=(w+\partial_{\lambda_{i}^{\scriptstyle\sss}})\frac{-e^{-(\lambda_{\star}| w)}}{(-2\pi i)^{d(n-1)}}\left.\left(\oint_{z_{1},\cdots z_{n-1}=z_{n}}\alpha e^{\sum\limits_{j=1}^{n}(z_{j}|\lambda_{j})}\right)\right|_{z_n=w}\\
        &=
        \tilde{\mu}_{ I}(\alpha)\cdot z_{i}^{\sss},
    \end{align*}
    Similarly, derivatives act as
    \begin{align*}
        &\tilde{\mu}_{ I}(\alpha\cdot \partial_{z_{i}^{\scriptstyle\sss}})\\
        &=
        \frac{e^{-(\lambda_{\star}| w)}}{(-2\pi i)^{d(n-1)}}\left.\left(\oint_{z_{1},\cdots z_{n-1}=z_{n}}\partial_{z_{i}^{\scriptstyle\sss}}(\alpha) e^{\sum\limits_{j=1}^{n}(z_{j}|\lambda_{j})}\right)\right|_{z_n=w}\\
        &=
        \frac{e^{-(\lambda_{\star}| w)}}{(-2\pi i)^{d(n-1)}}\left.\left(\oint_{z_{1},\cdots z_{n-1}=z_{n}}\partial_{z_{i}^{\scriptstyle\sss}}(\alpha e^{\sum\limits_{j=1}^{n}(z_{j}|\lambda_{j})})\right)\right|_{z_n=w}\\
        &-
        \frac{e^{-(\lambda_{\star}| w)}}{(-2\pi i)^{d(n-1)}}\left.\left(\oint_{z_{1},\cdots z_{n-1}=z_{n}}\lambda_{i}^{\sss}\alpha e^{\sum\limits_{j=1}^{n}(z_{j}|\lambda_{j})}\right)\right|_{z_n=w}\\
        &=
        \frac{e^{-(\lambda_{\star}| w)}}{(-2\pi i)^{d(n-1)}}\left.\left(\oint_{z_{1},\cdots z_{n-1}=z_{n}}\partial_{z_{i}^{\scriptstyle\sss}}(\alpha e^{\sum\limits_{j=1}^{n}(z_{j}|\lambda_{j})})\right)\right|_{z_n=w}+
        \tilde{\mu}_{ I}(\alpha)\cdot \partial_{z_{i}^{\scriptstyle\sss}}.
    \end{align*}
    Using Proposition~\ref{residue 0}, the first term is $0$ when $i\neq n$. When $i=n$, the first term is
    $$
    (-\lambda_{\star}+\partial_{w})\frac{e^{-(\lambda_{\star}| w)}}{(-2\pi i)^{d(n-1)}}\left.\left(\oint_{z_{1},\cdots z_{n-1}=z_{n}}\alpha e^{\sum\limits_{j=1}^{n}(z_{j}|\lambda_{j})}\right)\right|_{z_n=w},
    $$
    which is zero as an element in
    $$
    \omega_{\mathbb{A}^{d}}[d]\otimes_{\kk[\lambda_{\star}]}\kk[\lambda_{1},\dots,\lambda_{n}].
    $$

    We next check that $\tilde{\mu}$ is of cohomological degree $1$.
    Assume
    $$
    \alpha=p\otimes\beta\in\mathbf{J}^{ I}_{\mathbb{A}^{d}}\otimes_{\mathcal{O}_{(\mathbb{A}^{d})^{ I}}}(\omega_{\mathbb{A}^{d}}[d])^{\boxtimes I},
    $$
    we use $|p|$ and $|\beta|$ to denote the homological degree of $p$ and $\beta$ respectively. To make sure the integral
    $$
    \oint_{z_{1},\cdots z_{n-1}=z_{n}}p\otimes\beta e^{\sum\limits_{j=1}^{n}(z_{j}|\lambda_{j})}
    $$
    is non-zero, we have
    $$
    \begin{cases}
        |\Gamma_{1}|+|p|=d(|\Gamma_{0}|-1)+|\Gamma_{1}|-1\\
        |\beta|=-d|\Gamma_{0}|
    \end{cases}.
    $$
    So $|\tilde{\mu}_{ I}(\alpha)|=-d=|p|+|\beta|+1=|\alpha|+1$.

    For the $\mathrm{GL}_d$-equivariance, we use Proposition \ref{explicit formula} (2). In that explicit formula, we can formally treat $\tilde{z}_i$ and $\bar{\tilde{z}}_i$ as independent variables and declare that under the action $A\in \mathrm{GL}_d$
$$
   (\tilde{z}_i^{\1},\dots,\tilde{z}_i^{\dd})\mapsto  (\tilde{z}_i^{\1},\dots,\tilde{z}_i^{\dd})A,\quad i=1,\dots,n,
   $$
   $$
(\bar{\tilde{z}}_i^{\1},\dots,\bar{\tilde{z}}_i^{\dd})\mapsto  (\bar{\tilde{z}}_i^{\1},\dots,\bar{\tilde{z}}_i^{\dd})(A^{T})^{-1},\quad i=1,\dots,n.
   $$
   Then in the expression of $W((\vec{\Gamma},\mmm,l),\beta e^{\sum\limits_{i=1}^{n}(z_{i}|\lambda_{i})})$, we see that both $\iota_{\prod_{i=1}^{n-1}(d^{d}\tilde{z}_{i}d^{d}\bar{\tilde{z}}_{i})}$ and $\sum\limits_{\scriptstyle\sss=1}^{d}(\partial_{\lambda_{e}^{\scriptstyle\sss}}+\tilde{z}_{e}^{\sss})\circ d\hat{y}_{e}^{\sss}$ are invariant, and for $e\in \Gamma_1$
$$
(d\hat{y}_{e}^{\1},\dots,d\hat{y}_{e}^{\dd})\mapsto(d\hat{y}_{e}^{\1},\dots,d\hat{y}_{e}^{\dd})(A^{T})^{-1}
$$
   has the expected transformation rule.
\end{proof}

Thus, $\{\tilde \mu_I\}$ defines a chiral operation of the appropriate degree.
For such operations to define a chiral algebra, we need to check that $\tilde{\mu}_{ I}$ satisfies the following shifted $L_{\infty}$ relations.
This is the main result of this section.

\begin{thm}
    $\{\tilde{\mu}_{ I}\}$ satisfies the following properties:
    \begin{enumerate}
        \item $\tilde{\mu}_{ I}=\tilde{\mu}_{ I_{\sigma_{ii'}}}\circ{\sigma_{ii'}}$, where $\sigma_{ii'}$ is the permutation of $i\neq i'\in I$ (recall it is defined in Definition \ref{SymmetricGroupAction}).
        \item For any finite set $I$ one has (recall the notation in the proof of Theorem \ref{DgOperadStructure})
        $$
        -\tilde{\mu}_{ I}\circ \mathbf{d}=\sum_{ I'\subset I}\tilde{\mu}_{\{\bullet\}\cup I- I'}\circ \tilde{\mu}_{ I'\subset I}
        $$
        where the sum is over subsets $I' \subset I$.
    \end{enumerate}
\end{thm}
\begin{proof}
    We first prove (1). When $i,i'\neq n$, it is trivial. Assume $i<i'=n$. Since $\tilde{\mu}_{ I}$ is a $\mathcal{D}_{(\mathbb{A}^{d})^{ I}}$-module morphism, we only need to prove this in the case when
    $$
    \alpha=p(x_{i}^{\sss},\mathbf{d}x_{i}^{\sss})\otimes\prod_{i=1}^{n}d^{d}z_{i}.
    $$
    In this case,
    $$
    \tilde{\mu}_{ I}(\alpha)=\frac{-e^{-(\lambda_{\star}| w)}}{(-2\pi i)^{d(n-1)}}\left.\left(\oint_{z_{1},\cdots z_{n-1}=z_{n}}\alpha e^{\sum\limits_{j=1}^{n}(z_{j}|\lambda_{j})}\right)\right|_{z_{n}=w}.
    $$
    Using the coordinate transformation
    $$
    z_{j}=z'_{j}+2w-z'_{n}-z'_{i},
    $$
    we have
    $$
    \tilde{\mu}_{ I}(\alpha)=\frac{-e^{-(\lambda_{\star}| w)}}{(-2\pi i)^{d(n-1)}}\left.\left(\oint_{z'_{1},\cdots, z'_{i-1},z'_{n}, z'_{i+1},\cdots, z'_{n-1}=z'_{i}}(-1)^{d^2}\alpha e^{\sum\limits_{j=1}^{n}(z'_{j}|\lambda_{j})+(\lambda_{\star}|(z'_i-z'_n)})\right)\right|_{z'_{i}=w}.
    $$
    By Proposition \ref{explicit formula} and the fact that $d^{d}w\cdot\lambda_{\star}^{\sss}=0$, we have
    $$
    \tilde{\mu}_{ I}(\alpha)=\frac{-e^{-(\lambda_{\star}| w)}}{(-2\pi i)^{d(n-1)}}\left.\left(\oint_{z'_{1},\cdots, z'_{i-1},z'_{n}, z'_{i+1},\cdots, z'_{n-1}=z'_{i}}{\sigma_{in}}(\alpha) e^{\sum\limits_{j=1}^{n}(z'_{j}| \lambda_{j})}\right)\right|_{z'_{i}=w}=\tilde{\mu}_{ I_{\sigma_{in}}}\circ{\sigma_{in}}(\alpha).
    $$

    Now we prove (2). Assume $\alpha=p_{(\vec{\Gamma},\mmm,l)}\otimes\beta$, where $(\vec{\Gamma},\mmm,l)$ is a Feynman graph. Let
    $$
    \tilde{W}=p_{(\vec{\Gamma},\mmm,l)}(-e^{-(z_{e}| y_{e})}dy_{e}^{\mmm(e)},-e^{-(z_{e}| y_{e})}(z_{e}|dy_{e})dy_{e}^{\mmm(e)})e^{\sum\limits_{i=1}^{n}(z_{i}|\lambda_{i})}\otimes\beta.
    $$
    By Stokes' theorem and Proposition \ref{boundary description},
    \begin{align*}
        &-\tilde{\mu}_{ I}(\mathbf{d}\alpha)\\
        &=\frac{e^{-(\lambda_{\star}| w)}(-1)^{\frac{1}{2}(|\Gamma_{1}|-l)(|\Gamma_{1}|-l+1)}}{(-2\pi i)^{d(n-1)}}\left.\left(\int_{S^{+}((0,+\infty)^{\vec{\Gamma}_{1}})}\int_{(\mathbb{A}^{d})^{ I-\{n\}}}(\bar{\partial}+d_t)\tilde{W}\right)\right|_{z_n=w}\\
        &=
        \frac{e^{-(\lambda_{\star}| w)}(-1)^{\frac{1}{2}(|\Gamma_{1}|-l)(|\Gamma_{1}|-l+1)}}{(-2\pi i)^{d(n-1)}}\left.\left(\int_{S^{+}((0,+\infty)^{\vec{\Gamma}_{1}})}d_{t}\left(\int_{(\mathbb{A}^{d})^{ I-\{n\}}}\tilde{W}\right)\right)\right|_{z_n=w}\\
        &=
        \frac{e^{-(\lambda_{\star}| w)}(-1)^{\frac{1}{2}(|\Gamma_{1}|-l)(|\Gamma_{1}|-l+1)}}{(-2\pi i)^{d(n-1)}}\left(\sum_{\vec{\Gamma}'\subset \vec{\Gamma}}\mathrm{sgn}(\sigma_{\vec{\Gamma}'_{1}\subset \vec{\Gamma}_{1}})\int_{S^{+}((0,+\infty)^{\vec{\Gamma}_{1}- \vec{\Gamma}'_{1}})}\int_{S^{+}((0,+\infty)^{\vec{\Gamma}'_{1}})}\right.\\
        &
        \left.\left.\int_{(\mathbb{A}^{d})^{\{\bullet\}\cup\vec{\Gamma}_{0}-\vec{\Gamma}'_{0}-\{n\}}}\int_{(\mathbb{A}^{d})^{|\Gamma'_{0}|-1}}\tilde{W}\right)\right|_{z_n=w}\\
        &=
        \sum_{ I'\subset I}\tilde{\mu}_{\{\bullet\}\cup I- I'}\circ \tilde{\mu}_{ I'\subset I}(\alpha).
    \end{align*}
\end{proof}

Finally, using the canonical d\'{e}calage isomorphism
$$
\mathrm{Sym}^{\boxtimes I}(\omega_{\mathbb{A}^{d}}[d])\cong \wedge^{\boxtimes I}(\omega_{\mathbb{A}^{d}}[d-1])[n],
$$
the shifted $L_{\infty}$ chiral algebra structure $\{\tilde{\mu}_{ I}\}_{ I\in \mathbf{fSet}}$ on $\omega_{\mathbb{A}^{d}}[d]$ corresponds to an $L_{\infty}$ chiral algebra structure on $\omega_{\mathbb{A}^{d}}[d-1]$.
\begin{defn}
  The $L_{\infty}$-algebra structure on $\omega_{\mathbb{A}^{d}}[d-1]$ is called the \textit{unit} $L_{\infty}$ chiral
  algebra structure. We denote it by $\omega_{\mathbb{A}^d}^{\blacklozenge}$. The corresponding $L_{\infty}$ operations are
  denoted by $\{\mu^{\omega}_{ I}\}_{ I\in \mathbf{fSet}}$.
\end{defn}

\section{Free-field realization via chiral algebras}\label{s:example}

We provide examples of chiral algebras in the Jouanolou model.
These include commutative chiral algebras and free ghost chiral algebras.
Additionally, we provide a construction of a chiral algebra enhancement of the higher-dimensional Kac--Moody algebra as
constructed in \cite{FHK}.
This is a generalization of the relationship between the (universal) WZW/current chiral algebras and affine Kac--Moody algebras. We will assume that $\kk=\mathbb{C}$ in this section. However, everything works for the general ground field $\mathbf{k}$ if one accepts the existence of a unit chiral algebra over $\kk$.

For simplicity of notation, we write $\omega$ instead of $\omega_{\mathbb{A}^d}$ throughout this section.
\subsection{Commutative chiral algebras}
Following \cite{BD}, we define the notion of commutative chiral algebras.
\begin{defn}
   Let $\mathcal{A}$ be a chiral algebra with zero differential on $\mathbb{A}^d$. We say that $\mathcal{A}$ is
   \textit{commutative} if the composition
$$
\Gamma\left((\mathbb{A}^d)^{{I}},\mathcal{A}^{\boxtimes {I}}\right)\hookrightarrow \mathbf{J}^{{I}}_{\mathbb{A}^d}(\mathcal{A})\xrightarrow{\mu_{{I}}}\mathcal{A}\otimes_{\mathbf{k}[\lambda_{\star}]}\mathbf{k}[\lambda_I]
$$
    vanishes for all $I$ with $|I|\geq 2$.
    Here, the first map is the inclusion of global sections into derived global sections.
\end{defn}

A graded commutative $\mathcal{D}_{\mathbb{A}^d}$-algebra $\mathbf{B}$ is a left $\mathcal{D}_{\mathbb{A}^d}$-module with a commutative product
$$
m\colon\mathbf{B}\otimes\mathbf{B}\rightarrow \mathbf{B}
$$
which is a $\mathcal{D}_{\mathbb{A}^d}$-module map.
\begin{prop}\label{prop:commop}
  Let $\mathbf{B}$ be a commutative $\mathcal{D}_{\mathbb{A}^d}$-algebra.
  There is a chiral algebra structure on $\mathcal{B}\define\mathbf{B}\otimes_{\mathcal{O}_{\mathbb{A}^d}} \omega_{\mathbb{A}^d}^{\blacklozenge}$ which makes $\mathcal{B}$ into a commutative chiral algebra.
\end{prop}
\begin{proof}
  We define
  $$
  \mu^{\mathcal{B}}_{{I}}:\mathbf{J}^{{I}}_{\mathbb{A}^d}(\mathcal{B}^{\boxtimes I})\rightarrow \mathcal{B}\otimes_{\mathbf{k}[\lambda_{\star}]}\mathbf{k}[\lambda_{I}]
  $$
  to be the composition of the following sequence of maps
  $$
\mathbf{J}^{{I}}_{\mathbb{A}^d}(\mathcal{B}^{\boxtimes{I}})=  \mathbf{J}^{{I}}_{\mathbb{A}^d}((\omega_{\mathbb{A}^d}^{\blacklozenge})^{\boxtimes I})\otimes_{\mathcal{O}_{(\mathbb{A}^d)^I}} \mathbf{B}^{\boxtimes{I}}\xrightarrow{\mu^{\omega}_{{I}}\otimes \mathrm{Id}}\left(\omega_{\mathbb{A}^d}^{\blacklozenge}\otimes_{\mathbf{k}[\lambda_{\star}]}\mathbf{k}[\lambda_{I}]\right)\otimes_{\mathcal{O}_{(\mathbb{A}^d)^I}}\mathbf{B}^{\boxtimes{I}}
  $$
  $$
  \xrightarrow{\sim}\omega_{\mathbb{A}^d}^{\blacklozenge}\otimes_{\mathbf{k}[\lambda_{\star}]}\mathbf{k}[\lambda_{I}]\otimes_{\mathbf{k}[\lambda_{I}]} \left( \mathbf{B}^{\boxtimes I}\otimes \mathbf{k}[\lambda_{I}]\right)  \xrightarrow{\sim}\omega_{\mathbb{A}^d}^{\blacklozenge}\otimes_{\mathbf{k}[\lambda_{\star}]} \left( \mathbf{B}^{\otimes{I}}\otimes \mathbf{k}[\lambda_{I}]\right)
  $$
  $$
  \xrightarrow{m^{({I})}}\omega_{\mathbb{A}^d}^{\blacklozenge}\otimes_{\mathbf{k}[\lambda_{\star}]} \left( \mathbf{B}\otimes \mathbf{k}[\lambda_{I}]\right)\xrightarrow{\sim}\mathcal{B}\otimes_{\mathbf{k}[\lambda_{\star}]}\mathbf{k}[\lambda_{I}].
  $$
 Here $m^{(I)}$ is the iterated commutative product. The fact that $\{\mu_I^{\mathcal{B}}\}$ satisfy the $L_\infty$ relations follows immediately from the
  $L_\infty$-relations of the unit chiral algebra.
\end{proof}

\subsection{Free ghost chiral algebras}
The following set of examples are immediate generalizations of familiar
\(d=1\) ghost-type chiral algebras, such as the \(\beta-\gamma\) and \(b-c\)
systems. Let \(F\) be a locally free \(\mathbb Z\)-graded
\(\mathcal O_{\mathbb A^d}\)-module, and let
\[
  F_{\mathcal D}
  \define
  F\otimes_{\mathcal O_{\mathbb A^d}}\mathcal D_{\mathbb A^d}
\]
be the induced \(\mathcal D_{\mathbb A^d}\)-module. We assume that \(F\) is
equipped with a non-degenerate graded skew-symmetric pairing of graded
\(\mathcal O_{\mathbb A^d}\)-modules
\begin{equation}\label{eq:pairing}
  \langle -,-\rangle:
  F\otimes_{\mathcal O_{\mathbb A^d}}F
  \longrightarrow
  \omega_{\mathbb A^d}[d-1].
\end{equation}

We introduce the following graded commutative $\mathcal{D}$-algebra
$$
\mathbf{F} =\mathrm{Sym}\left(F_{\mathcal{D}} \otimes \omega^{-1}\right)
$$

\begin{rem}
  In the case that $F= \omega_{\mathbb{A}^d}\oplus \mathcal{O}[1-d]$, with the obvious pairing, the chiral algebra we are about to define is associated to the
  following first-order action functional in physics
  \begin{equation}\label{}
    \int_{\mathbb{C}^d} \beta \overline{\partial} \gamma
  \end{equation}
  where $\beta$ is a form of Dolbeault type $(d,d-1)$ and $\gamma$ is a smooth function.
  Indeed, when $d=1$ this returns the ordinary $\beta-\gamma$ system used in superstring theory.
  To obtain the $b-c$ system we must introduce an additional $\mathbb{Z}/2$ grading and consider the parity shift of this
  example.
  The same generalization can be made in the higher-dimensional setting as well, though we will not pursue it here.
\end{rem}

\begin{defn}{\label{WickContraction}}
Define the $\mathcal{D}$-module map
$$
e^{\mathcal{P}_{\{1,2\}}}\colon \mathbf{J}^{\{1,2\}}_{\mathbb{A}^d}(\mathbf{F}^{\boxtimes 2})\rightarrow \mathbf{J}^{\{1,2\}}
_{\mathbb{A}^d}(\mathbf{F}^{\boxtimes 2})
$$
by extending the $\langle - , - \rangle$ pairing
$$
\mathcal{P}_{\{1,2\}} = \langle - , - \rangle\cdot P_{12}  \colon(F\otimes\omega^{-1})\boxtimes F[1-d]\rightarrow \mathbf{J}^{\{1,2\}}_{\mathbb{A}^d}(\mathcal{O}\boxtimes \mathcal{O})
$$
by the Leibniz rule. Here we define the element
$$
P_{12}=\sum^d_{\sss=1}(-1)^{\sss-1}x_{12}^s\cdot \mathbf{d}x_{12}^{\1}\cdots \widehat{\mathbf{d}x_{12}^{\sss}}\cdots \mathbf{d}x_{12}^{\dd},
$$
 which under the map $x^{\sss}_{12}\mapsto \frac{\bar{z}_{1}^{\scriptstyle{\sss}} - \bar{z}_{2}^{
 \scriptstyle{\sss}}}{|z_{1} - z_{2}|^{2}}$ is exactly the Bochner-Martinelli kernel.
\end{defn}
Denote
\[
  \mathcal P_I
  \define
  \sum_{\{i,j\}\subset I}
  \mathcal P_{\{i,j\}}.
\]
If \(I=I'\bigsqcup I''\), we write
\[
  \mathcal P_{I'I''}
  \define
  \sum_{i'\in I',\,i''\in I''}
  \mathcal P_{\{i',i''\}}.
\]
Thus
\[
  \mathcal P_I
  =
  \mathcal P_{I'}
  +
  \mathcal P_{I'I''}
  +
  \mathcal P_{I''}.
\]

\begin{defn}\label{FreeChiralDefn}
  We define the chiral operations $\{\mu_I^{\mathcal{F}}\}$ on $\mathcal{F}=\mathbf{F}\otimes_{\mathcal{O}_{\mathbb{A}^d}} \omega_{\mathbb{A}^d}^{\blacklozenge}$ as follows.
  Let $\{\mu_I^{\cF,comm}\}$ be the commutative chiral operations as constructed in Proposition \ref{prop:commop}.
  Then, define the new chiral operations $\mu_I^{\mathcal{F}}$ by the composition:
  \[
\mu^{\mathcal F}_{I}
=
\mu^{\mathcal F,\mathrm{comm}}_{I}
\circ
e^{\mathcal P_I}
:
\mathbf J_{\mathbb A^d}^{I}
\left(
\mathcal F^{\boxtimes I}
\right)
\longrightarrow
\mathcal F\otimes_{\mathbf k[\lambda_\star]}\mathbf k[\lambda_I].
\]
\end{defn}

We have the following lemma.

\begin{lem}
    Let $\alpha_{{I'}}\in \mathbf{J}_{\mathbb{A}^d}^{{I'}}((\omega_{\mathbb{A}^d}^{\blacklozenge})^{\boxtimes {I'}}) $,
    $b'\in \mathbf{F}^{\boxtimes {I'}}$ and $b''\in \mathbf{F}^{\boxtimes {I''}}$. We have
    \begin{multline}
    m^{(I')}\left(\mu^{\omega}_{{I'}}(\alpha_{{I'}})(    e^{{\mathcal{P}}_{{I'}{I''}}}(b'\boxtimes
    b''\cdot f ))\right)=\\ e^{{\mathcal{P}}_{\{\bullet\}{I''}}}m^{(I')}\left(\mu^{\omega}_{{I'}}(\alpha_{{I'}})(b'\boxtimes b''\cdot f) \right)\in \mathbf{J}^{\{\bullet\}\bigsqcup{I''}}_{\mathbb{A}^d}(\mathcal{F}\boxtimes \mathcal{F}^{\boxtimes{I''}})\otimes_{k[\lambda_{\bullet}]}k[\lambda_{I'}]
  \end{multline}
    for all $f\in \mathbf{J}^{({I'}),{{I''}}}_{\mathbb{A}^d}(\mathcal{O}_{\mathbb{A}^d}^{\boxtimes{I'}}\boxtimes(\omega_{\mathbb{A}^d}^{\blacklozenge})^{\boxtimes{I''}})
$ (recall the notation in the proof of Theorem \ref{DgOperadStructure}).
\end{lem}
\begin{proof}
Since $  e^{{\mathcal{P}}_{{I'}{I''}}}$ is a $\mathcal{D}$-module map, we have
$$
m^{(I')}\left(\mu^{\omega}_{{I'}}(\alpha_{{I'}})(    e^{{\mathcal{P}}_{{I'}{I''}}}(b'\boxtimes b''\cdot f ))\right)=m^{(I')}\left(e^{{\mathcal{P}}_{{I'}{I''}}}(\mu^{\omega}_{{I'}}(\alpha_{{I'}})(    b'\boxtimes b''\cdot f ))\right).
$$
Then the lemma follows from the fact that (recall the proof of Proposition \ref{prop:commop})
\[
  m^{(I')}
  \left(
  e^{\mathcal P_{I'I''}}(-)
  \right)
  =
  e^{\mathcal P_{\{\bullet\}I''}}
  m^{(I')}(-).
\]

\end{proof}

We can now prove the following.
\begin{prop}
The chiral operations \(\{\mu^{\mathcal F}_{I}\}\) define an
\(L_\infty\)-chiral algebra structure on
\[
  \mathcal F
  =
  \mathbf F\otimes_{\mathcal O_{\mathbb A^d}}
  \omega_{\mathbb A^d}^{\blacklozenge}
\]
for every locally free graded \(\mathcal O_{\mathbb A^d}\)-module \(F\)
equipped with a pairing as in \eqref{eq:pairing}.
\end{prop}
\begin{proof}
From the definition of  $\{\mu^{\mathcal{F}}_{{I}}\}$ and the above lemma, we have
\begin{align*}
&\mu^{\mathcal{F}}_{{I}}\left(\mathbf{d}(\alpha_{{I'}}b'\boxtimes b''\cdot f)\right)\\
&=\mu^{\mathcal{F},comm}_{{I}}\left(e^{{\mathcal{P}}_{{I}}}\mathbf{d}(\alpha_{{I'}}b'\boxtimes b''\cdot f)\right)
=\mu^{\mathcal{F},comm}_{{I}}\left(\mathbf{d}(e^{{\mathcal{P}}_{{I}}}(\alpha_{{I'}}b'\boxtimes b''\cdot f))\right)\\
&=\sum \mu^{\mathcal{F},comm}_{\bullet\bigsqcup {I}-{I'}}\circ \mu^{\mathcal{F},comm}_{{I'}\subset{I}}\left(e^{{\mathcal{P}}_{{I'}}}e^{{\mathcal{P}}_{{I'}{I''}}}e^{{\mathcal{P}}_{{I''}}}(\alpha_{{I'}}b'\boxtimes b''\cdot f) \right)\\
&=\sum \mu^{\mathcal{F},comm}_{\bullet\bigsqcup {I}-{I'}}\circ \left(e^{{\mathcal{P}}_{\bullet{I''}}}e^{{\mathcal{P}}
_{{I''}}} \mu^{\mathcal{F},comm}_{{I'}\subset {I}}(e^{{\mathcal{P}}_{{I'}}}\alpha_{{I'}}b'\boxtimes b''\cdot f)\right)\\
&=\sum \mu^{\mathcal{F}}_{\bullet\bigsqcup{I}-{I'}}\circ \mu^{\mathcal{F}}_{{I'}\subset {I}}\left(\alpha_{{I'}}b'\boxtimes b''\cdot f\right),
\end{align*}
which proves the theorem.
 \end{proof}
\begin{rem}
  This is the Wick theorem in higher dimensions, which may be more familiar in the context of the operator product
  expansion for vertex algebras.
  See \cite{gui2023,gui2022elliptictracemapchiral} for the $d=1$ Wick theorem in the language of chiral algebras.  
\end{rem}
\begin{rem}
  Note that while we used the commutative chiral operations as part of the definition of $\{\mu_I^{\mathcal{F}}\}$,
  this new chiral algebra is \textit{not} commutative.
\end{rem}

From Proposition \ref{prop:derham} we know that the de Rham functor applied to $\mathcal{F}$ yields an $L_\infty$
algebra $h(\mathcal{F})$.
In this case, this $L_\infty$ algebra is, in fact, a graded Lie algebra.
Explicitly, this graded Lie algebra has underlying graded vector space
\begin{equation}\label{}
  \mathrm{Sym} \left(F \otimes \omega^{-1}\right) \otimes \omega [d-1] .
\end{equation}
The bracket is the natural extension of $\langle - , - \rangle$ by the graded Leibniz rule.
In particular, $h(\mathcal{F})$ is a graded Poisson algebra.

\subsection{The Faonte-Hennion-Kapranov central extension via chiral operations}

In this section, denote $\mathsf{A}^\bu_d \define \mathbf{J}_{\mathring{\mathbb{A}}^d}$ the Jouanolou model of
punctured affine space~$\mathring{\mathbb{A}}^d = \mathbb{A}^d \setminus \{0\}$.
Equivalently, this is the quotient of the Jouanolou model of two points in $\mathbb{A}^d$ by the translations.
It is a commutative dg algebra model for the derived global sections of the structure sheaf of the punctured affine
space $\mathring{\mathbb{A}}^d$.
For a Lie algebra $\mathfrak{g}$, introduce the following dg Lie algebra
$$
\mathfrak{g}^{\bullet}_d=\mathfrak{g}\otimes_{\mathbf{k}} \mathsf{A}_d^{\bullet}.
$$
This dg Lie algebra was introduced in \cite{FHK} as a higher-dimensional generalization of current algebras which
appear in symmetries of chiral conformal field theory \cite{GWkm}.
One of the main results of \cite{FHK} is that invariant polynomials of degree $(d+1)$ on the Lie algebra $\mathfrak{g}$ give
rise to nontrivial $L_\infty$ central extensions of this dg Lie algebra.
The centrally extended algebras provide higher-dimensional, or multivariate, generalizations of the famous affine Kac--Moody algebras.

Tied to the usual affine Kac--Moody algebra is a well-known chiral (and vertex) algebra.
In fact, it is a special case of a general enveloping algebra construction of a chiral algebra from any Lie$^\star$ algebra.
The notion of a Lie$^\star$ algebra can be generalized to higher dimensions.
While we do not pursue the construction here, we do expect a universal enveloping-type construction which associates to
such a higher-dimensional Lie$^\star$ algebra a higher chiral algebra in the sense of this work.
We leave this problem to future work.

Instead, we will give a more direct interpretation of the higher-dimensional Kac--Moody algebra in terms of its "free
field realization" using the higher Wick's theorem of the previous section.

Let $V$ be a $\mathfrak{g}$-representation and consider the graded $\mathcal{O}$-module
\[
  F = \underline{V^*} \oplus \underline{V} \otimes
\omega [d-1] = V^* \otimes \mathcal{O}\oplus V \otimes \omega [d-1]
\]
equipped with the obvious pairing in the sense of \eqref{eq:pairing}.
(We will assume that $V$ is an ordinary vector space, rather than a graded one, for simplicity.)
The free ghost chiral algebra $\mathcal{F}$ associated to $V$, as defined in the previous section, can be identified,
as graded $\mathcal{D}$-module, with
\begin{equation}\label{GhostChiralV}
 \mathcal{F}=\mathrm{Sym}(\underline{V^{\vee}}_{\mathcal{D}}\otimes\omega^{-1}\oplus \underline{V} [1-d]\otimes\omega_{\mathcal{D}}\otimes\omega^{-1})\otimes\omega[d-1].   
\end{equation}

Given an \(L_{\infty}\)-chiral algebra \(\mathcal A\), we consider the following higher-dimensional generalization of the modes Lie algebra
(compare with \cite[\S 4.1]{BZF}, where the notation
\(\mathrm U'(-)\) is used):
\[
  \modes(\mathcal A)
 \define
  \frac{
  \mathsf A_d^\bu\otimes_{\mathcal O}\mathcal A
  }{
  \operatorname{Im}\,\partial
  } .
\]
Here \(\operatorname{Im}\,\partial\) is the space of total derivatives, spanned
by the right action of the vector fields
\(\partial_{z^{\sss}}\), \(\sss=1,\ldots,d\). This construction is simply
Proposition~\ref{prop:derham} applied to the restriction of \(\mathcal A\) to
the punctured affine space \(\mathring{\mathbb A}^d\), with
\(\mathsf A_d^\bu\) serving as the Jouanolou model for functions on
\(\mathring{\mathbb A}^d\). Thus the \(L_{\infty}\)-chiral algebra structure on
\(\mathcal A\) induces an ordinary \(L_{\infty}\)-algebra structure on
\(\modes(\mathcal A)\).

The unit embedding
\[
  \omega_{\mathbb A^d}^{\blacklozenge}
  \hookrightarrow
  \mathcal F
\]
induces a morphism of \(L_{\infty}\)-algebras
\[
  \modes(\omega_{\mathbb A^d}^{\blacklozenge})
  \longrightarrow
  \modes(\mathcal F).
\]
The \(L_{\infty}\)-algebra
\(\modes(\omega_{\mathbb A^d}^{\blacklozenge})\) is abelian and central in
\(\modes(\mathcal F)\).

The \(\mathfrak g\)-representation \(V\) defines a cochain map
\[
  \rho_{\mathrm{aff}}
  :
  \mathfrak g_d^\bullet
  \longrightarrow
  (V^\vee\otimes_{\mathbf k}V)\otimes_{\mathbf k}\mathsf A_d^\bu
  \hookrightarrow
  \modes(\mathcal F).
\]
Composing with the quotient map gives
\[
  \overline{\rho}_{\mathrm{aff}}
  :
  \mathfrak g_d^\bullet
  \longrightarrow
  \modes(\mathcal F)/
  \modes(\omega_{\mathbb A^d}^{\blacklozenge}).
\]

\begin{prop}
The map
\[
  \overline{\rho}_{\mathrm{aff}}
  :
  \mathfrak g_d^\bullet
  \longrightarrow
  \modes(\mathcal F)/
  \modes(\omega_{\mathbb A^d}^{\blacklozenge})
\]
is a strict morphism of \(L_{\infty}\)-algebras. Here ``strict'' means that only
the linear component is nonzero, and this linear component intertwines the
differentials and all multi-brackets.
\end{prop}

\begin{proof}
The map \(\rho_{\mathrm{aff}}\) is a cochain map by construction. For the binary
bracket, the one-contraction term in \(\modes(\mathcal F)\) gives the ordinary
current bracket on \(\mathfrak g_d^\bullet\), while the remaining scalar term
lies in the central subalgebra
\(\modes(\omega_{\mathbb A^d}^{\blacklozenge})\). Hence the binary bracket is
intertwined after passing to the quotient.

For arity \(n\geq 3\), the source \(\mathfrak g_d^\bullet\) has no higher
brackets. On the target side, the higher brackets of the affine currents land in
\(\modes(\omega_{\mathbb A^d}^{\blacklozenge})\), and therefore vanish in the
quotient. Thus \(\overline{\rho}_{\mathrm{aff}}\) is a strict
\(L_{\infty}\)-morphism.
\end{proof}

Consider the pullback square
\[
\begin{tikzcd}
\widetilde{\mathfrak g_d^\bullet} \arrow[dd] \arrow[rr]
&  &
\modes(\mathcal F) \arrow[dd]
\\
&  &
\\
\mathfrak g_d^\bullet \arrow[rr, "\overline{\rho}_{\mathrm{aff}}"']
&  &
{\modes(\mathcal F)/
\modes(\omega_{\mathbb A^d}^{\blacklozenge})}.
\end{tikzcd}
\]
We obtain a central extension
\[
0
\longrightarrow
\modes(\omega_{\mathbb A^d}^{\blacklozenge})
\longrightarrow
\widetilde{\mathfrak g_d^\bullet}
\longrightarrow
\mathfrak g_d^\bullet
\longrightarrow
0.
\]
Using the residue map
\[
  \operatorname{Res}
  :
  \modes(\omega_{\mathbb A^d}^{\blacklozenge})
  =
  \frac{
  \mathsf A_d^\bu\otimes_{\mathcal O}
  \omega_{\mathbb A^d}^{\blacklozenge}
  }{
  \operatorname{Im}\,\partial
  }
  \longrightarrow
  \mathbb C,
\]
we push out the above central extension and obtain a one-dimensional
\(\mathbb C\)-central extension of \(\mathfrak g_d^\bullet\):
\[
\begin{tikzcd}
0 \arrow[r]
&
{\modes(\omega_{\mathbb A^d}^{\blacklozenge})}
\arrow[r] \arrow[d, "\operatorname{Res}"']
&
\widetilde{\mathfrak g_d^\bullet}
\arrow[r] \arrow[d]
&
\mathfrak g_d^\bullet
\arrow[r] \arrow[d, no head, equal]
&
0
\\
0 \arrow[r]
&
\mathbb C
\arrow[r]
&
\mathfrak g_d^{\bullet\flat}
\arrow[r]
&
\mathfrak g_d^\bullet
\arrow[r]
&
0 .
\end{tikzcd}
\]

We will use the following standard way of extracting a Lie algebra cocycle from
such a central extension. Suppose that
\[
0
\longrightarrow
\mathbb C
\longrightarrow
\mathfrak l^\flat
\xrightarrow{\pi}
\mathfrak l
\longrightarrow
0
\]
is a central extension of \(L_{\infty}\)-algebras, where \(\mathfrak l\) is an
ordinary dg Lie algebra and the morphisms are strict. Assume also that
\(\mathfrak l^\flat\) and \(\mathfrak l\) are \(\mathbb Z_{\geq 0}\)-graded.
For the extensions considered here, we choose a cochain splitting
\[
  s \colon \mathfrak l\longrightarrow \mathfrak l^\flat
\]
of graded vector spaces. Define cochains
\[
  \mathsf{c}_n\in
  \operatorname{Hom}^{2-n}
  \left(
  \bigwedge^n\mathfrak l,\mathbb C
  \right),
  \qquad n\geq 2,
\]
by
\[
  {\mathsf{c}_2}(x_1,x_2)
  \define
  [s(x_1),s(x_2)]^{\flat}_2
  -
  s\left([x_1,x_2]_2\right),
\]
and, for \(n\geq 3\),
\[
  \mathsf{c}_n(x_1,\ldots,x_n)
  \define
  [s(x_1),\ldots,s(x_n)]^{\flat}_n .
\]
The \(L_{\infty}\)-identities for \(\mathfrak l^\flat\) imply that
\(\mathsf{c}= (\mathsf{c}_n)_{n\geq 2}\) is a Chevalley--Eilenberg cocycle. Changing the splitting
changes \(\mathsf{c}\) by a coboundary. Hence the central extension determines a
cohomology class
\[
  [\mathsf{c}]\in
  \mathbb H^2_{\mathrm{Lie}}(\mathfrak l;\mathbb C).
\]
We compare this cohomology class to one which defines a higher Kac--Moody extension in the sense of \cite{FHK}.
To every symmetric degree $d+1$ invariant polynomial $\theta$ on $\lie{g}$ there is a functional
\[
  \gamma_{\theta} \colon \left(\mathfrak{g}_d^{\bullet}[1]\right)^{\otimes_{\mathbf{k}}d+1}\rightarrow \mathbf{k}
\]
defined by
\[
  \gamma_{\theta} \left((x_0\otimes f_0)\otimes(x_1\otimes f_1)\otimes \cdots \otimes (x_d\otimes f_d) \right)=\theta (x_0,\dots,x_d)\cdot \mathrm{Res}\left(f_0\cdot \partial f_1\wedge\cdots \wedge \partial f_d\right).
\]
This functional is a cocycle and induces a cohomology class 
\[
  [\gamma_{\theta}]\in \mathbb{H}^2_{\mathrm{Lie}}(\mathfrak{g}_d^{\bullet};\mathbb{C}) .
\]

Define the symmetric polynomial
$$
p^V_d(x_0,\dots,x_{d})=\frac{1}{(d+1)!}\sum_{\sigma\in S_{d+1}}\mathrm{tr}_V\left(x_{\sigma(0)}\cdots x_{\sigma(d)}\right).
$$
Where $\op{tr}_V$ is the trace taken in the representation $V$.
\begin{prop}\label{PureAffine}
    The cohomology class associated to the central extension
    \[
    0\rightarrow \mathbb{C}\rightarrow{\mathfrak{g}_{2}^{\bullet\flat}}\rightarrow {\mathfrak{g}_{2}^{\bullet}}\rightarrow 0
    \]
    is equal to the class $[\gamma_{p^V_2}]\in \mathbb{H}^2_{\mathrm{Lie}}(\mathfrak{g}_{2}^{\bullet};\mathbb{C})$ of \cite{FHK}.
\end{prop}
\begin{proof}
We use the notation of Sections \ref{PropagatorNotation}, \ref{ResidueOp},
and \ref{Formulas}. Label the three inputs by \(\{o,1,2\}\), with \(o\)
corresponding to \(a_0\). Write
\[
a_i=x_i\otimes f_i\in \mathfrak g^\bullet_2,\qquad i=0,1,2.
\]
We use the splitting induced by
\[
s(a)=\rho_{\mathrm{aff}}(a)\in h_{\mathring{\AA}^2}(\mathcal F).
\]

By Definition \ref{FreeChiralDefn}, the operation on \(\mathcal F\) is
\[
\mu_I^{\mathcal F}
=
\mu_I^{\mathcal F,\mathrm{comm}}\circ e^{P_I}.
\]
Thus it is computed by Wick contraction followed by the unit chiral operation.
In dimension \(2\), the arity-\(n\) unit operation is zero unless the scalar
input contains exactly \(2n-3\) propagator factors. Therefore the arity-three
operation kills all Wick terms with fewer than three propagators.

Let us first note that the central cochain has no arity-two contribution. The
one-contraction term is exactly the ordinary current bracket and hence is
accounted for by the strict map \(\rho_{\mathrm{aff}}\). The only possible scalar
term is the double contraction. Its scalar factor is
\[
P_{o1}P_{1o}.
\]
But in dimension \(2\), using \(P_{1o}=P_{o1}\), this product is
\[
P_{o1}^2=0
\]
by graded commutativity, since \(P_{o1}\) is a Jouanolou one-form. Hence
\(C^{\mathrm{aff}}_2=0\). Similarly, for arity \(n>3\), a complete contraction
of \(n\) quadratic currents contains at most \(n\) propagators, whereas the
unit \(n\)-operation requires \(2n-3\) propagators. Since \(n<2n-3\), no
central term occurs. Thus only \(C^{\mathrm{aff}}_3\) remains.

For three quadratic affine currents, the complete Wick contractions are exactly
the two cyclic contractions
\[
(o\to 1\to 2\to o),
\qquad
(o\to 2\to 1\to o).
\]
The first gives the trace factor \(\operatorname{tr}_V(x_0x_1x_2)\) and the
scalar propagator
\[
P_{1o}P_{12}P_{2o}.
\]
By formula \eqref{OneloopConst},
\[
\mu_{\{o,1,2\}}
\left(
P_{1o}P_{12}P_{2o}\,
d\mathbf z^{\blacklozenge}_{\{o,1,2\}}
\right)
=
\frac12\,\lambda_1\wedge\lambda_2\,
d\mathbf z^{\blacklozenge}_o .
\]
Therefore this oriented loop contributes
\[
\frac12\operatorname{tr}_V(x_0x_1x_2)\,
\operatorname{Res}
\left(
f_0\,\partial f_1\wedge\partial f_2
\right).
\]

The second cyclic contraction is obtained by interchanging \(1\) and \(2\). It
gives the trace factor \(\operatorname{tr}_V(x_0x_2x_1)\). Applying
\eqref{OneloopConst} to the ordered triple \((o,2,1)\) gives
\[
\frac12\operatorname{tr}_V(x_0x_2x_1)\,
\operatorname{Res}
\left(
f_0\,\partial f_2\wedge\partial f_1
\right)
\]
as a cochain on \(a_0\wedge a_2\wedge a_1\). Rewriting it as a cochain on
\(a_0\wedge a_1\wedge a_2\) introduces the transposition sign, which cancels
the sign from
\[
\partial f_2\wedge\partial f_1
=
-\partial f_1\wedge\partial f_2.
\]
Hence the second loop contributes
\[
\frac12\operatorname{tr}_V(x_0x_2x_1)\,
\operatorname{Res}
\left(
f_0\,\partial f_1\wedge\partial f_2
\right).
\]

Combining the two cyclic contractions gives
\[
  \frac12
\left(
\operatorname{tr}_V(x_0x_1x_2)
+
\operatorname{tr}_V(x_0x_2x_1)
\right)
\operatorname{Res}
\left(
f_0\,\partial f_1\wedge\partial f_2
\right).
\]
By cyclicity of the trace, we note
\[
\frac12
\left(
\operatorname{tr}_V(x_0x_1x_2)
+
\operatorname{tr}_V(x_0x_2x_1)
\right)
=
\frac16
\sum_{\sigma\in S_3}
\operatorname{tr}_V
\left(
x_{\sigma(0)}x_{\sigma(1)}x_{\sigma(2)}
\right).
\]
Thus, a cocycle representative for the class associated to our central extension is
\[
a_0 \wedge a_1 \wedge a_2 \mapsto p_2^V(x_0,x_1,x_2)\,
\operatorname{Res}
\left(
f_0\,\partial f_1\wedge\partial f_2
\right)
=
\gamma_{p_2^V}(a_0\wedge a_1\wedge a_2).
\]
Since all other arity components vanish, the free-field central extension
represents the Faonte--Hennion--Kapranov class $[\gamma_{p_2^V}]$.
\end{proof}
\begin{rem}
We anticipate that an analogous result should be valid in arbitrary dimensions, by comparison with \cite{GWkm}.
Nevertheless, this comparison is not completely straightforward, and we leave a detailed analysis to future work.
\end{rem}

\subsection{Higher Virasoro structures}

Let $\mathfrak{witt}_d^\bu$ denote the Jouanolou model of the derived global sections of the tangent sheaf on
punctured affine space $\mathring{\mathbb{A}}^d$.
In \cite{GWvir} a homomorphism 
\begin{equation}
    H^{2d+2} (BGL_d) \to \mathbb{H}^2_{Lie}(\lie{witt}_d^\bu) 
\end{equation}
is produced and is shown to be an isomorphism when $d=2$.
Similarly, as in the previous section, we now provide an explicit construction of such central extensions using chiral operations of a chiral algebra of free-field type.

Recall that $\mathfrak{witt}^{\bullet}_{d}\simeq \oplus^d_{\sss=1}\mathbf{J}^{\bullet}_{\mathring{\mathbb{A}}^d}\cdot \partial_{z^{{\scriptstyle \sss}}}$. We take $V$ to be an $r$-dimensional vector space and denote $\mathcal{F}$ to be the corresponding free ghost chiral algebra defined by (\ref{GhostChiralV}). We construct the following map
$$
\rho_{vir}\colon \mathfrak{witt}^{\bullet}_{d}\rightarrow \modes(\mathcal{F})
$$
defined by (we choose a basis $\{\mathbf{v}_{\ell}\}^r_{\ell=1}$ of $V$ and its dual $\{\mathbf{v}^{\vee}_{\ell}\}^r_{\ell=1}$)
$$
\rho_{vir}(\sum^d_{\sss=1}T^{\sss}(z)\partial_{z^{\sss}})=\sum^r_{\ell=1}\sum^d_{{\sss}=1}T^{\sss}(z)\cdot(\mathbf{v}_{\ell}^{\vee}\otimes \partial_{z^{{\scriptstyle \sss}}})\cdot \mathbf{v}_{\ell},\quad T^{\sss}(z)\in \mathbf{J}^{\bullet}_{\mathring{\mathbb{A}}^d}.
$$

Parallel to the previous section, we have a central extension
\[
0\rightarrow\C\rightarrow \mathfrak{witt}^{\bullet\flat}_{d}\rightarrow \mathfrak{witt}^{\bullet}_{d}\rightarrow 0
\]
and the corresponding cohomology class $[\mathsf{c}^{\mathrm{vir}}]\in \mathbb{H}^2_{\mathrm{Lie}}(\mathfrak{witt}^{\bullet}_{d};\mathbb{C})$.


We focus on the case $d=2$. Recall that in \cite{GWvir}, there are two distinguished cohomology classes 
\[
(\mathrm{ch}^{\mathcal{T}}_1)^3,\mathrm{ch}^{\mathcal{T}}_1\mathrm{ch}^{\mathcal{T}}_2 \in \mathbb{H}^2_{\mathrm{Lie}}(\mathfrak{witt}^{\bullet}_{2};\mathbb{C}) \cong H^6(BGL_2)
\]
which take the explicit form
\[
(\mathrm{ch}^{\mathcal{T}}_1)^3\left(T_1\wedge T_2\wedge T_3\right)=2\sum_{\sigma\in \mathrm{Cyc}}\mathrm{Res}(\mathrm{div}(T_{\sigma(1)})\cdot \partial\mathrm{div}(T_{\sigma(2)})\cdot\partial\mathrm{div}(T_{\sigma(3)})),
\]
\[
\mathrm{ch}^{\mathcal{T}}_1\mathrm{ch}^{\mathcal{T}}_2\left(T_1\wedge T_2\wedge T_3\right)=\mathrm{Res}(\mathrm{div}(T_1)\cdot \langle\partial T_2,\partial T_3\rangle+\mathrm{div}(T_2)\cdot \langle\partial T_3,\partial T_1\rangle+\mathrm{div}(T_3)\cdot \langle\partial T_1,\partial T_2\rangle).
\]
    Here $\mathrm{div}(T)=\sum^2\limits_{\sss=1}\partial_{z^{{\scriptstyle \sss}}} T^{\sss}(z)$ is the divergence and \(\langle \partial T_i,\partial T_j\rangle
\define
\sum\limits_{\rrr,\sss=1}^{2}
\partial(\partial_{z^{\scriptstyle\rrr}}T_i^{\sss})
\wedge
\partial(\partial_{z^{{\scriptstyle \sss}}}T_j^{\rrr}).
\)

    \begin{prop}\label{PureVira}
        We have
        \[
        [\mathsf{c}^{\mathrm{vir}}]=r\cdot [\mathrm{Todd}(\mathcal{T})]_6\quad \in \mathbb{H}^2_{\mathrm{Lie}}(\mathfrak{witt}^{\bullet}_{2};\mathbb{C}),
        \]
        where
        \[
        [\mathrm{Todd}(\mathcal{T})]_6 = \left[\frac{1}{48}(\mathrm{ch}^{\mathcal{T}}_1)^3-\frac{1}{24}\mathrm{ch}^{\mathcal{T}}_1\mathrm{ch}^{\mathcal{T}}_2\right]
        \].
    \end{prop}
\begin{proof}
We use the splitting induced by
\[
s(T)=\rho_{vir}(T)\in \modesS(\mathcal F).
\]
The Wick-contraction and propagator-counting argument in the proof of
Proposition \ref{PureAffine} applies verbatim to these quadratic fields. Hence the
central cochain is supported only in arity three.

For three Virasoro fields, the complete contractions are the two oriented cyclic
one-loop contractions. The contraction of the \(V\)-indices gives the factor
\(\dim V=r\). The only difference from Proposition \ref{PureAffine} is that each
insertion of
\[
\rho_{vir}
\left(
\sum_{\sss=1}^2T^{\sss}(z)\partial_{z^{{\scriptstyle \sss}}}
\right)
=
\sum_{\ell=1}^{r}\sum_{\sss=1}^2
T^{\sss}(z)\cdot
(\mathbf v_\ell^\vee\otimes\partial_{z^{{\scriptstyle \sss}}})\cdot\mathbf v_\ell
\]
places one first-order derivative on the corresponding propagator. Thus the scalar
part of the one-loop contraction is the operation on three first-order differentiated
propagators computed in \eqref{Oneloop3order}.

We interpret the resulting polynomial in the variables \(\lambda_i^{\sss}\) by the
diagonal \(\mathcal D\)-module convention: \(\lambda_i^{\sss}\) acts as
\(\partial_{z_i^{\scriptstyle\sss}}\) on the coefficient of the \(i\)-th labelled input; after this
action one restricts to the diagonal and applies \(\operatorname{Res}\). Since
\(\operatorname{Res}\) factors through the quotient by total derivatives, the
coefficient calculation recorded in Appendix \ref{app:ViraCoefficient} gives
\[
\mathsf{c}^{\mathrm{vir}}_3
=
r\left(
\frac{1}{48}(\mathrm{ch}^{\mathcal T}_1)^3
-
\frac{1}{24}\mathrm{ch}^{\mathcal T}_1\mathrm{ch}^{\mathcal T}_2
\right).
\]
All other arity components vanish by the same argument as in
Proposition \ref{PureAffine}. Therefore
\[
[\mathsf{c}^{\mathrm{vir}}]
=
r\cdot[\mathrm{Todd}(\mathcal T)]_6
\in
\mathbb H^2_{\mathrm{Lie}}
(\mathfrak{witt}^{\bullet}_{2};\mathbb C).
\]
\end{proof}
\begin{figure}[htbp]
    \centering

\tikzset{every picture/.style={line width=0.75pt}} 

\begin{tikzpicture}[x=0.75pt,y=0.75pt,yscale=-1,xscale=1]

\draw  [dash pattern={on 0.84pt off 2.51pt}]  (100,129) -- (294,57) ;
\draw [shift={(294,57)}, rotate = 339.64] [color={rgb, 255:red, 0; green, 0; blue, 0 }  ][fill={rgb, 255:red, 0; green, 0; blue, 0 }  ][line width=0.75]      (0, 0) circle [x radius= 3.35, y radius= 3.35]   ;
\draw [shift={(100,129)}, rotate = 339.64] [color={rgb, 255:red, 0; green, 0; blue, 0 }  ][fill={rgb, 255:red, 0; green, 0; blue, 0 }  ][line width=0.75]      (0, 0) circle [x radius= 3.35, y radius= 3.35]   ;
\draw  [dash pattern={on 0.84pt off 2.51pt}]  (100,129) -- (280,168) ;
\draw [shift={(280,168)}, rotate = 12.23] [color={rgb, 255:red, 0; green, 0; blue, 0 }  ][fill={rgb, 255:red, 0; green, 0; blue, 0 }  ][line width=0.75]      (0, 0) circle [x radius= 3.35, y radius= 3.35]   ;
\draw [shift={(100,129)}, rotate = 12.23] [color={rgb, 255:red, 0; green, 0; blue, 0 }  ][fill={rgb, 255:red, 0; green, 0; blue, 0 }  ][line width=0.75]      (0, 0) circle [x radius= 3.35, y radius= 3.35]   ;
\draw  [dash pattern={on 0.84pt off 2.51pt}]  (100,129) -- (376,129) ;
\draw [shift={(376,129)}, rotate = 0] [color={rgb, 255:red, 0; green, 0; blue, 0 }  ][fill={rgb, 255:red, 0; green, 0; blue, 0 }  ][line width=0.75]      (0, 0) circle [x radius= 3.35, y radius= 3.35]   ;
\draw [shift={(100,129)}, rotate = 0] [color={rgb, 255:red, 0; green, 0; blue, 0 }  ][fill={rgb, 255:red, 0; green, 0; blue, 0 }  ][line width=0.75]      (0, 0) circle [x radius= 3.35, y radius= 3.35]   ;
\draw  [color={rgb, 255:red, 155; green, 155; blue, 155 }  ,draw opacity=1 ][fill={rgb, 255:red, 155; green, 155; blue, 155 }  ,fill opacity=0.31 ] (226,115.5) .. controls (226,66.07) and (266.07,26) .. (315.5,26) .. controls (364.93,26) and (405,66.07) .. (405,115.5) .. controls (405,164.93) and (364.93,205) .. (315.5,205) .. controls (266.07,205) and (226,164.93) .. (226,115.5) -- cycle ;
\draw    (324,221) -- (324,263) ;
\draw [shift={(324,266)}, rotate = 270] [fill={rgb, 255:red, 0; green, 0; blue, 0 }  ][line width=0.08]  [draw opacity=0] (10.72,-5.15) -- (0,0) -- (10.72,5.15) -- (7.12,0) -- cycle    ;
\draw  [dash pattern={on 0.84pt off 2.51pt}]  (136,314) -- (342,300) ;
\draw [shift={(342,300)}, rotate = 356.11] [color={rgb, 255:red, 0; green, 0; blue, 0 }  ][fill={rgb, 255:red, 0; green, 0; blue, 0 }  ][line width=0.75]      (0, 0) circle [x radius= 3.35, y radius= 3.35]   ;
\draw [shift={(136,314)}, rotate = 356.11] [color={rgb, 255:red, 0; green, 0; blue, 0 }  ][fill={rgb, 255:red, 0; green, 0; blue, 0 }  ][line width=0.75]      (0, 0) circle [x radius= 3.35, y radius= 3.35]   ;
\draw    (325,327) -- (325,348) ;
\draw [shift={(325,351)}, rotate = 270] [fill={rgb, 255:red, 0; green, 0; blue, 0 }  ][line width=0.08]  [draw opacity=0] (10.72,-5.15) -- (0,0) -- (10.72,5.15) -- (7.12,0) -- cycle    ;
\draw  [dash pattern={on 0.84pt off 2.51pt}]  (326,366) ;
\draw [shift={(326,366)}, rotate = 0] [color={rgb, 255:red, 0; green, 0; blue, 0 }  ][fill={rgb, 255:red, 0; green, 0; blue, 0 }  ][line width=0.75]      (0, 0) circle [x radius= 3.35, y radius= 3.35]   ;
\draw [shift={(326,366)}, rotate = 0] [color={rgb, 255:red, 0; green, 0; blue, 0 }  ][fill={rgb, 255:red, 0; green, 0; blue, 0 }  ][line width=0.75]      (0, 0) circle [x radius= 3.35, y radius= 3.35]   ;

\draw (308,37.4) node [anchor=north west][inner sep=0.75pt]  [font=\scriptsize]  {$\rho _{vir}( T_{1})$};
\draw (349,98.4) node [anchor=north west][inner sep=0.75pt]  [font=\scriptsize]  {$\rho _{vir}( T_{2})$};
\draw (298,171.4) node [anchor=north west][inner sep=0.75pt]  [font=\scriptsize]  {$\rho _{vir}( T_{3})$};
\draw (351,235.4) node [anchor=north west][inner sep=0.75pt]  [font=\scriptsize]  {$\mu _{3}^{\mathcal{F}}( \rho _{vir}( T_{1}) \boxtimes \rho _{vir}( T_{2}) \boxtimes \rho _{vir}( T_{3}))$};
\draw (392,335.4) node [anchor=north west][inner sep=0.75pt]  [font=\scriptsize]  {$\mathrm{R} es( -)$};

\end{tikzpicture}
    \caption{Virasoro central extension via the chiral operation}
    \label{ViraFig}
\end{figure}

\begin{rem} 
 This result is compatible with a local universal index theorem in the spirit of \cite{Beilinson:1986zw}. More generally, in dimension $d$, we expect an isomorphism $H^{2d+2}(BU(d)) \cong
H^2_{\op{Lie}} (\mathfrak{witt}_d)$.
Thus, we conjecture a complete characterization of central extensions of the $d$-dimensional Witt
algebra in terms of universal characteristic classes.
\end{rem}

We can further combine the Virasoro and Kac-Moody central extensions. We can use the natural action of $ \mathfrak{witt}^{\bullet}_{d}$ on $\mathfrak{g}_d^{\bullet}$ to form a semi-direct product of dg Lie algebras \(\mathfrak{witt}^{\bullet}_{d}\ltimes \mathfrak{g}_d^{\bullet}\). We can define
$$
\rho_{\mathrm{vir}\ltimes \mathrm{aff}}=\rho_{\mathrm{vir}}\oplus \rho_{ \mathrm{aff}}\colon \mathfrak{witt}^{\bullet}_{d}\ltimes \mathfrak{g}_d^{\bullet}\rightarrow \modes(\mathcal{F}).
$$
As before, we get a central extension of $L_{\infty}$-algebras
\[
0\rightarrow\C\rightarrow \left(\mathfrak{witt}^{\bullet}_{d}\ltimes \mathfrak{g}_d^{\bullet}\right)^{\flat}\rightarrow \mathfrak{witt}^{\bullet}_{d}\ltimes \mathfrak{g}_d^{\bullet}\rightarrow 0,
\]
which produces a cohomology class $[\mathsf{c}^{\mathrm{vir}\ltimes \mathrm{aff}}]\in \mathbb{H}^2_{\mathrm{Lie}}(\mathfrak{witt}^{\bullet}_{d}\ltimes \mathfrak{g}_d^{\bullet};\mathbb{C})$.

Specializing in the case $d=2$, we construct
\[
\mathrm{ch}^{V}_3 \define \gamma_{p_2^V}( (x_0\otimes f_0)\wedge (x_1\otimes f_1)\wedge (x_2\otimes f_2))=\frac{1}{6}\sum_{\sigma\in S_{3}}\mathrm{tr}_V\left(x_{\sigma(0)} x_{\sigma(1)}x_{\sigma(2)}\right)\cdot \mathrm{Res}\left(f_0\cdot \partial f_1 \wedge \partial f_2\right),
\]
\[
\mathrm{ch}^{\mathcal T}_1\mathrm{ch}^{V}_2
\left(
T\wedge (x_0\otimes f_0)\wedge (x_1\otimes f_1)
\right)
=
\mathrm{tr}_V(x_0x_1)\,
\mathrm{Res}
\left(
\mathrm{div}(T)\,
\partial f_0\wedge\partial f_1
\right),
\]
\[
\mathrm{ch}^{V}_1\mathrm{ch}^{\mathcal T}_2
\left(
(x_0\otimes f_0)\wedge T_1\wedge T_2
\right)
=
\mathrm{tr}_V(x_0)\,
\mathrm{Res}
\left(
f_0\,
\langle \partial T_1,\partial T_2\rangle
\right),
\]
\[
\mathrm{ch}^{V}_1(\mathrm{ch}^{\mathcal T}_1)^2
\left(
(x_0\otimes f_0)\wedge T_1\wedge T_2
\right)
=
2\,\mathrm{tr}_V(x_0)\,
\mathrm{Res}
\left(
f_0\,
\partial\mathrm{div}(T_1)\wedge
\partial\mathrm{div}(T_2)
\right).
\]
Proposition \ref{PureAffine} and \ref{PureVira} can be enhanced to the following theorem.
\begin{thm}\label{FullCentral}
    Define the following class in $ \mathbb{H}^2_{\mathrm{Lie}}(\mathfrak{witt}^{\bullet}_{2}\ltimes \mathfrak{g}_2^{\bullet};\mathbb{C})$
    \[
     [\mathrm{Todd}(\mathcal{T})\mathrm{Ch}(V)]_6 \define [\mathrm{ch}^{V}_3]+\frac{1}{2}[\mathrm{ch}^{\mathcal{T}}_1\mathrm{ch}^{V}_2]+\left[\frac{1}{8}\mathrm{ch}^{V}_1(\mathrm{ch}^{\mathcal{T}}_1)^2-\frac{1}{12}\mathrm{ch}^{V}_1\mathrm{ch}^{\mathcal{T}}_2\right]+ r\cdot \left[\frac{1}{48}(\mathrm{ch}^{\mathcal{T}}_1)^3-\frac{1}{24}\mathrm{ch}^{\mathcal{T}}_1\mathrm{ch}^{\mathcal{T}}_2\right].
    \]
    Then
    \[
    [\mathsf{c}^{\mathrm{vir}\ltimes \mathrm{aff}}]=[\mathrm{Todd}(\mathcal{T})\mathrm{Ch}(V)]_6\in \mathbb{H}^2_{\mathrm{Lie}}(\mathfrak{witt}^{\bullet}_{2}\ltimes \mathfrak{g}_2^{\bullet};\mathbb{C}).
    \]
\end{thm}
\begin{proof}
We use the splitting induced by
\[
s=\rho_{\mathrm{vir}\ltimes\mathrm{aff}}
=
\rho_{\mathrm{vir}}\oplus\rho_{\mathrm{aff}}
:
\mathfrak{witt}^{\bullet}_{2}\ltimes\mathfrak g^{\bullet}_{2}
\longrightarrow
\modesS(\mathcal F).
\]
The Wick-contraction and propagator-counting argument in the proof of
Proposition \ref{PureAffine} applies verbatim. Hence the associated central
cochain is supported only in arity three.

It remains to evaluate the arity-three complete cyclic contractions. By
skew-symmetry, it is enough to consider the four types of inputs
\[
(A,A,A),\qquad (T,A,A),\qquad (A,T,T),\qquad (T,T,T),
\]
where \(A=x\otimes f\in \mathfrak g^\bullet_2\) and
\(T\in\mathfrak{witt}^\bullet_2\). The \((A,A,A)\)-term is exactly
Proposition \ref{PureAffine}, and the \((T,T,T)\)-term is exactly
Proposition \ref{PureVira}. For the two mixed terms, the only new point is the
number of first-order derivatives placed on the propagators: one derivative in
the \((T,A,A)\)-case and two derivatives in the \((A,T,T)\)-case. Thus the
scalar coefficients are the ones computed in \eqref{Oneloop1order} and
\eqref{Oneloop2order}.

Interpreting the resulting \(\lambda\)-polynomials by the diagonal
\(\mathcal D\)-module convention, restricting to the diagonal, and applying
\(\operatorname{Res}\), the coefficient calculation recorded in
Appendix \ref{app:CentralCoefficientCalculation} gives
\[
C^{\mathrm{vir}\ltimes\mathrm{aff}}_{3}
=
\mathrm{ch}^{V}_{3}
+\frac12\,\mathrm{ch}^{\mathcal T}_{1}\mathrm{ch}^{V}_{2}
+
\left(
\frac18\,\mathrm{ch}^{V}_{1}(\mathrm{ch}^{\mathcal T}_{1})^{2}
-
\frac1{12}\,\mathrm{ch}^{V}_{1}\mathrm{ch}^{\mathcal T}_{2}
\right)
+
r\left(
\frac1{48}(\mathrm{ch}^{\mathcal T}_{1})^{3}
-
\frac1{24}\mathrm{ch}^{\mathcal T}_{1}\mathrm{ch}^{\mathcal T}_{2}
\right).
\]
All other arity components vanish by the same argument as in
Proposition \ref{PureAffine}. Therefore
\[
[\mathsf{c}^{\mathrm{vir}\ltimes\mathrm{aff}}]
=
[\mathrm{Todd}(\mathcal T)\mathrm{Ch}(V)]_{6}
\in
\mathbb H^2_{\mathrm{Lie}}
(\mathfrak{witt}^{\bullet}_{2}\ltimes \mathfrak g^{\bullet}_{2};\mathbb C).
\]
\end{proof}

\section{Chiral algebras on $\mathbb A^2$ in more detail}
\label{s:A2}

In this section we unpack the chiral operations of the unit chiral algebra on
\(\mathbb A^2\).  For simplicity, we write \(\mu_k\) for the \(k\)-ary chiral
operation of the unit chiral algebra \(\omega_{\mathbb A^2}[1]\) defined in
Section~\ref{s:unit}.  Throughout this section we work over \(\kk=\mathbb C\).
The main recursive statement, Theorem~\ref{MainTheoremRecursiveRS}, is
independent of this specialization: it holds over a general ground field
\(\mathbf k\), provided the unit chiral algebra over \(\mathbf k\) is available.

\paragraph{Conventions for Section~\ref{s:A2}.}
Vertex labels are denoted by \(i,j,k,o,\star,\bullet\).  Component indices in
\(\mathbb A^2\) are displayed using the fixed component-index font:
\[
  \sss\in\{\1,\2\}.
\]
Thus
\[
  z_i^{\1},z_i^{\2},\qquad
  x_{ij}^{\1},x_{ij}^{\2},\qquad
  \lambda_i^{\1},\lambda_i^{\2}
\]
denote components, not powers.  For any two ordered pairs
\(f=(f^{\1},f^{\2})\) and \(g=(g^{\1},g^{\2})\), we use
\[
  (f|g)\define f^{\1}g^{\1}+f^{\2}g^{\2},
  \qquad
  f\wedge g\define f^{\1}g^{\2}-f^{\2}g^{\1}.
\]
When expressions involve both the \(D\)-module variables \(\lambda^{\sss}\) and
the coordinate variables \(z^{\sss},x^{\sss}\), we move the \(\lambda\)'s to the
right using the \(D\)-module convention.  In particular, the antisymmetric
combinations satisfy
\[
  \lambda\wedge z
  =
  \lambda^{\1}z^{\2}-\lambda^{\2}z^{\1}
  =
  z^{\2}\lambda^{\1}-z^{\1}\lambda^{\2}
  =
  -\,z\wedge\lambda
\]
and
\[
  \lambda\wedge x
  =
  \lambda^{\1}x^{\2}-\lambda^{\2}x^{\1}
  =
  x^{\2}\lambda^{\1}+x^{\1}x^{\2}
  -
  x^{\1}\lambda^{\2}-x^{\1}x^{\2}
  =
  -\,x\wedge\lambda .
\]
Finally, for \(i\in I\), we denote the shifted top form by
\[
  d\mathbf z_i^{\blacklozenge}
  \define
  \bigl(dz_i^{\1}\wedge dz_i^{\2}\bigr)[1],
  \qquad
  d\mathbf z_I^{\blacklozenge}
  \define
  \boxtimes_{i\in I}d\mathbf z_i^{\blacklozenge}.
\]
\subsection{The propagator}\label{PropagatorNotation}

We begin with the basic propagator.

\begin{defn}\label{dfn:prop}
  The propagator \(P_{12}\) is the element of
  \(\mathbf J_{\mathbb A^2}^{\{1,2\}}\) given by
  \[
    P_{12}
    =
    x_{12}^{\1}\mathbf d x_{12}^{\2}
    -
    x_{12}^{\2}\mathbf d x_{12}^{\1}.
  \]
  More generally, for any ordered pair of distinct vertices \(i,j\), we use
  \[
    P_{ij}
    \define
    x_{ij}^{\1}\mathbf d x_{ij}^{\2}
    -
    x_{ij}^{\2}\mathbf d x_{ij}^{\1}.
  \]
\end{defn}

With the coordinate convention above, \(x_{ji}=-x_{ij}\), and hence
\[
  P_{ji}=P_{ij}.
\]
Products of propagators are products of odd one-forms, so the displayed order of
the factors is part of the convention.

Our first result is the simplest nontrivial computation of the chiral two-ary operation.

\begin{prop}
  We have the following identity:
  \[
    \mu_{2}\bigl(
      P_{12}
      d\mathbf z_{1}^{\blacklozenge}
      \boxtimes
      d\mathbf z_{2}^{\blacklozenge}
    \bigr)
    =
    d\mathbf z^{\blacklozenge}.
  \]
\end{prop}

\begin{proof}
  The element \(P_{12}\) contains two terms. We first identify the Feynman graphs
  corresponding to these two monomials. Let \(\vec{\Gamma}\) be the directed graph
  constructed as follows:
  \begin{itemize}
    \item \(\vec{\Gamma}_{0}=\{1,2\}\) and \(\vec{\Gamma}_{1}=\{e_{1},e_{2}\}\).
    \item \(t(e_{1})=t(e_{2})=1\) and \(h(e_{1})=h(e_{2})=2\).
  \end{itemize}
  We consider two Feynman graphs
  \((\vec{\Gamma},{\mmm}_{1},2)\) and \((\vec{\Gamma},{\mmm}_{2},2)\), where
  \({\mmm}_{1}\) and \({\mmm}_{2}\) are defined by
  \[
    \begin{cases}
      {\mmm}_{1}(e_{1})=\1,\quad {\mmm}_{1}(e_{2})=\2,\\
      {\mmm}_{2}(e_{1})=\2,\quad {\mmm}_{2}(e_{2})=\1.
    \end{cases}
  \]
  Then
  \[
    p_{(\vec{\Gamma},{\mmm}_{1},2)}
    =
    x_{12}^{\1}\mathbf d x_{12}^{\2},
    \qquad
    p_{(\vec{\Gamma},{\mmm}_{2},2)}
    =
    x_{12}^{\2}\mathbf d x_{12}^{\1}.
  \]
  The graphical Green's function is
  \[
    \begin{cases}
      d^{-1}_{\vec{\Gamma}}(t)_{e_{1}1}
      =
      \dfrac{t_{e_{2}}}{t_{e_{1}}+t_{e_{2}}},\\[0.8em]
      d^{-1}_{\vec{\Gamma}}(t)_{e_{2}1}
      =
      \dfrac{t_{e_{1}}}{t_{e_{1}}+t_{e_{2}}}.
    \end{cases}
  \]

  By Proposition \ref{explicit formula}, we have
  \begin{align*}
    &\mu_{2}\bigl(
      p_{(\vec{\Gamma},{\mmm}_{1},2)}
      d\mathbf z^{\blacklozenge}_{1}
      \boxtimes
      d\mathbf z^{\blacklozenge}_{2}
    \bigr) \\
    &=
    \left.
    \left(
      \int_{S^{+}((0,+\infty)^{\vec{\Gamma}_{1}})}
      \iota_{d^{2}\tilde z_{1}d^{2}\bar{\tilde z}_{1}}
      \circ
      (-d\hat y_{e_{1}}^{\1})
      \circ
      \left(
        -
        \sum_{\sss=1}^{2}
        \bigl(
          \partial_{\lambda_{e_{2}}^{\scriptstyle \sss}}
          +
          \tilde z_{e_{2}}^{\sss}
        \bigr)
        \circ
        d\hat y_{e_{2}}^{\sss}
      \right)
      \circ
      d\hat y_{e_{2}}^{\2}
      \bigl(
        d^{2}z_{1}\boxtimes d^{2}z_{2}
      \bigr)
    \right)
    \right|_{\tilde z_{1}=0,\tilde z_{2}=w} \\
    &=
    \left(
      \int_{S^{+}((0,+\infty)^{\vec{\Gamma}_{1}})}
      \frac{t_{e_{1}}}{t_{e_{1}}+t_{e_{2}}}
      \left(
        \frac{t_{e_{1}}}{t_{e_{1}}+t_{e_{2}}}
        d\left(
          \frac{t_{e_{2}}}{t_{e_{1}}+t_{e_{2}}}
        \right)
        -
        \frac{t_{e_{2}}}{t_{e_{1}}+t_{e_{2}}}
        d\left(
          \frac{t_{e_{1}}}{t_{e_{1}}+t_{e_{2}}}
        \right)
      \right)
    \right)d^{2}w.
  \end{align*}
  Similarly, we have
  \begin{align*}
    &\mu_{2}\bigl(
      -p_{(\vec{\Gamma},{\mmm}_{2},2)}
      d\mathbf z^{\blacklozenge}_{1}
      \boxtimes
      d\mathbf z^{\blacklozenge}_{2}
    \bigr)\\
    &=
    \left(
      \int_{S^{+}((0,+\infty)^{\vec{\Gamma}_{1}})}
      \frac{t_{e_{2}}}{t_{e_{1}}+t_{e_{2}}}
      \left(
        \frac{t_{e_{1}}}{t_{e_{1}}+t_{e_{2}}}
        d\left(
          \frac{t_{e_{2}}}{t_{e_{1}}+t_{e_{2}}}
        \right)
        -
        \frac{t_{e_{2}}}{t_{e_{1}}+t_{e_{2}}}
        d\left(
          \frac{t_{e_{1}}}{t_{e_{1}}+t_{e_{2}}}
        \right)
      \right)
    \right)d^{2}w.
  \end{align*}
  Combining these computations and using Proposition \ref{useful property}, we obtain
  \begin{align*}
    &\mu_{2}\bigl(
      P_{12}
      d\mathbf z^{\blacklozenge}_{1}
      \boxtimes
      d\mathbf z^{\blacklozenge}_{2}
    \bigr)\\
    &=
    \left(
      \int_{S^{+}((0,+\infty)^{\vec{\Gamma}_{1}})}
      \left(
        \frac{t_{e_{1}}}{t_{e_{1}}+t_{e_{2}}}
        d\left(
          \frac{t_{e_{2}}}{t_{e_{1}}+t_{e_{2}}}
        \right)
        -
        \frac{t_{e_{2}}}{t_{e_{1}}+t_{e_{2}}}
        d\left(
          \frac{t_{e_{1}}}{t_{e_{1}}+t_{e_{2}}}
        \right)
      \right)
    \right)d^{2}w\\
    &=
    \left(
      \int_{\{(t_{e_{1}},t_{e_{2}})\in (0,+\infty)^2
      \mid t_{e_{1}}+t_{e_{2}}=1\}}
      \left(
        \frac{t_{e_{1}}}{t_{e_{1}}+t_{e_{2}}}
        d\left(
          \frac{t_{e_{2}}}{t_{e_{1}}+t_{e_{2}}}
        \right)
        -
        \frac{t_{e_{2}}}{t_{e_{1}}+t_{e_{2}}}
        d\left(
          \frac{t_{e_{1}}}{t_{e_{1}}+t_{e_{2}}}
        \right)
      \right)
    \right)d^{2}w.
  \end{align*}
  The integral above produces \(1\), as desired.
\end{proof}

\begin{rem}
  This element is closed. In fact, any closed two-operation is a multiple of
  \(\mu_{2}\).
\end{rem}
\subsection{Residues as two-ary chiral operations}\label{ResidueOp}

Our next result computes the value of the two-ary chiral operation on arbitrary
derivatives of the propagator. To do this, we introduce a generating function for
the derivatives of \(P_{12}\). Instead of carrying out an explicit Feynman diagram
calculation, we characterize the expression using the \(D\)-module relations,
together with the computation from the previous section.

Introduce new formal variables
\[
  \mathfrak z_{12}
  =
  \bigl(\mathfrak z_{12}^{\1},\mathfrak z_{12}^{\2}\bigr)
\]
and form the generating series
\[
  P_{12}(\mathfrak z_{12})
  \define
  \sum_{r,k\geq 0}
  \frac{1}{r!k!}
  \bigl(\partial_{z^{\scriptstyle\1}_1}\bigr)^r
  \bigl(\partial_{z^{\scriptstyle\2}_1}\bigr)^k
  P_{12}\cdot
  \bigl(\mathfrak z_{12}^{\1}\bigr)^r
  \bigl(\mathfrak z_{12}^{\2}\bigr)^k .
\]
Here \(\mathfrak z_{12}\) is only a formal variable. It is not the coordinate
difference \(z_1-z_2\).

For notational convenience, set
\[
  (x_{12}\mid \mathfrak z_{12})
  \define
  x_{12}^{\1}\mathfrak z_{12}^{\1}
  +
  x_{12}^{\2}\mathfrak z_{12}^{\2}.
\]

\begin{lem}
  We have
  \[
    P_{12}(\mathfrak z_{12})
    =
    \frac{P_{12}}
    {\bigl(1+(x_{12}\mid \mathfrak z_{12})\bigr)^2},
  \]
  and, similarly,
  \[
    x_{12}^{\sss}(\mathfrak z_{12})
    =
    \frac{x_{12}^{\sss}}
    {1+(x_{12}\mid \mathfrak z_{12})},
    \qquad
    \sss\in\{\1,\2\}.
  \]
  These equalities are understood as formal power series in
  \(\mathfrak z_{12}^{\1},\mathfrak z_{12}^{\2}\).
\end{lem}

\begin{proof}
  By Proposition \ref{prop:dmodulestr}, for
  \(\ii_1,\ldots,\ii_\ell\in\{\1,\2\}\), we have
  \[
    (-x_{12}^{\ii_1})\cdots(-x_{12}^{\ii_\ell})P_{12}
    =
    \frac{
      \partial_{z^{\scriptstyle\ii_1}_1}\cdots
      \partial_{z^{\scriptstyle\ii_\ell}_1}P_{12}
    }{(\ell+1)!},
  \]
  and 
  \[
    (-x_{12}^{\ii_1})\cdots(-x_{12}^{\ii_\ell})x_{12}^{\sss}
    =
    \frac{
      \partial_{z^{\scriptstyle\ii_1}_1}\cdots
      \partial_{z^{\scriptstyle\ii_\ell}_1}x_{12}^{\sss}
    }{\ell!},
    \qquad
    \sss\in\{\1,\2\}.
  \]
  Hence
  \[
    P_{12}(\mathfrak z_{12})
    =
    \sum_{\ell\geq 0}
    (\ell+1)
    \bigl(-(x_{12}\mid \mathfrak z_{12})\bigr)^\ell
    P_{12}
    =
    \frac{P_{12}}
    {\bigl(1+(x_{12}\mid \mathfrak z_{12})\bigr)^2},
  \]
  and
  \[
    x_{12}^{\sss}(\mathfrak z_{12})
    =
    \sum_{\ell\geq 0}
    \bigl(-(x_{12}\mid \mathfrak z_{12})\bigr)^\ell
    x_{12}^{\sss}
    =
    \frac{x_{12}^{\sss}}
    {1+(x_{12}\mid \mathfrak z_{12})}.
  \]
\end{proof}

\begin{prop}
  We have the following explicit formula for \(\mu_2\):
  \[
    \mu_2\left(
      P_{12}(\mathfrak z_{12})\cdot
      d\mathbf z_{1}^{\blacklozenge}
      \boxtimes
      d\mathbf z_{2}^{\blacklozenge}
    \right)
    =
    d\mathbf z^{\blacklozenge}
    \otimes
    e^{-(\lambda_1\mid\mathfrak z_{12})},
  \]
  where
  \[
    (\lambda_1\mid\mathfrak z_{12})
    \define
    \lambda_1^{\1}\mathfrak z_{12}^{\1}
    +
    \lambda_1^{\2}\mathfrak z_{12}^{\2}.
  \]
\end{prop}

\begin{proof}
  By the \(D\)-module relation for the chiral operation, we have
  \[
    \mu_2\left(
      \partial_{z^{\sss}_1}Q\cdot
      d\mathbf z_{1}^{\blacklozenge}
      \boxtimes
      d\mathbf z_{2}^{\blacklozenge}
    \right)
    =
    -\,
    \mu_2\left(
      Q\cdot
      d\mathbf z_{1}^{\blacklozenge}
      \boxtimes
      d\mathbf z_{2}^{\blacklozenge}
    \right)\lambda_1^{\sss},
    \qquad
    \sss\in\{\1,\2\}.
  \]
  Applying this repeatedly to \(Q=P_{12}\), and using the computation
  \[
    \mu_2\left(
      P_{12}\cdot
      d\mathbf z_{1}^{\blacklozenge}
      \boxtimes
      d\mathbf z_{2}^{\blacklozenge}
    \right)
    =
    d\mathbf z^{\blacklozenge}
  \]
  from the previous section, gives
  \[
    \mu_2\left(
      \partial_{z^{\scriptstyle\ii_1}_1}\cdots
      \partial_{z^{\scriptstyle\ii_\ell}_1}P_{12}\cdot
      d\mathbf z_{1}^{\blacklozenge}
      \boxtimes
      d\mathbf z_{2}^{\blacklozenge}
    \right)
    =
    (-\lambda_1^{\ii_1})\cdots
    (-\lambda_1^{\ii_\ell})
    d\mathbf z^{\blacklozenge},
    \qquad
    \ii_1,\ldots,\ii_\ell\in\{\1,\2\}.
  \]
  Summing the resulting formal power series in
  \(\mathfrak z_{12}^{\1},\mathfrak z_{12}^{\2}\), we obtain
  \[
    \mu_2\left(
      P_{12}(\mathfrak z_{12})\cdot
      d\mathbf z_{1}^{\blacklozenge}
      \boxtimes
      d\mathbf z_{2}^{\blacklozenge}
    \right)
    =
    d\mathbf z^{\blacklozenge}
    \otimes
    e^{-(\lambda_1\mid\mathfrak z_{12})}.
  \]
\end{proof}

We next record the Arnold relation in the Jouanolou model. In this paragraph,
\(z_{ij}=z_i-z_j\) denotes the actual coordinate difference, not the formal variable
\(\mathfrak z_{12}\). We use the convention
\[
  x_{ij}\wedge x_{kl}
  \define
  x_{ij}^{\1}x_{kl}^{\2}
  -
  x_{ij}^{\2}x_{kl}^{\1}.
\]

\begin{prop}[Arnold relation]
  We have
  \[
    P_{12}P_{13}
    +
    P_{23}P_{12}
    +
    P_{13}P_{23}
    =
    \mathbf d\left(
      P_{13}\cdot x_{12}\wedge x_{23}
      +
      P_{12}\cdot x_{13}\wedge x_{23}
      +
      P_{23}\cdot x_{12}\wedge x_{13}
    \right).
  \]
\end{prop}

\begin{proof}
  We use the Jouanolou relation
  \[
    x_{23}^{\1}z_{23}^{\1}
    +
    x_{23}^{\2}z_{23}^{\2}
    =
    1
  \]
  and the identity
  \[
    z_{23}^{\sss}=z_{13}^{\sss}-z_{12}^{\sss},
    \qquad
    \sss\in\{\1,\2\}.
  \]
  Then
  \begin{align*}
    P_{12}P_{13}
    &=
    P_{12}P_{13}
    \bigl(
      x_{23}^{\1}z_{23}^{\1}
      +
      x_{23}^{\2}z_{23}^{\2}
    \bigr)\\
    &=
    P_{12}P_{13}
    \left(
      x_{23}^{\1}
      (z_{13}^{\1}-z_{12}^{\1})
      +
      x_{23}^{\2}
      (z_{13}^{\2}-z_{12}^{\2})
    \right)\\
    &=
    P_{13}\cdot \mathbf d x_{12}\wedge x_{23}
    -
    P_{12}\cdot \mathbf d x_{13}\wedge x_{23}.
  \end{align*}
  The analogous identities obtained by permuting the vertices \(1,2,3\), together
  with the fact that the propagators are odd one-forms, give the displayed identity.
\end{proof}

\begin{cor}\label{RecursiveLemma}
  We have
  \[
    P_{12}P_{13}P_{23}
    =
    \mathbf d\left(
      P_{13}P_{23}\cdot x_{12}\wedge x_{23}
      +
      P_{12}P_{23}\cdot x_{13}\wedge x_{23}
    \right).
  \]
\end{cor}

\begin{proof}
  Multiply the Arnold relation by \(P_{23}\) and use \(P_{23}^2=0\), keeping the
  displayed order of the odd factors. Equivalently,
  \[
    \mathbf d\left(
      P_{13}P_{23}\cdot x_{12}\wedge x_{23}
      +
      P_{12}P_{23}\cdot x_{13}\wedge x_{23}
    \right)
    =
    P_{12}P_{13}P_{23}.
  \]
\end{proof}

\begin{rem}
  The derived Arnold relation also holds for general \(\mathbb C^d\), \(d\geq 2\).
  This is proved in \cite{FGY} for the polysimplicial model. For comparison, one
  can replace \(x_{ij}^{\ii}\) by \(
    \frac{u_{ij}^{\scriptstyle\ii}}{z_{ij}^{\scriptstyle\ii}}
  \)  to obtain the Arnold relation in the Jouanolou model.
\end{rem}
\subsection{Formulas for chiral operations}\label{Formulas}

In Section~\ref{s:feynman}, we constructed a family of chiral operations using
Feynman integrals and showed that they satisfy the \(L_{\infty}\)-relations.
These are precisely the \(L_{\infty}\)-relations underlying the unit chiral
algebra on \(\mathbb A^d\).
In the case of \(\mathbb A^2\), the \(L_{\infty}\)-relations impose very strong
constraints on the chiral operations and can sometimes determine them
completely. In this section, we derive recursive formulas for chiral
operations purely from the $L_{\infty}$-relation and find a perfect match with numerous calculations that have appeared
in the work of physics
\cite{budzik2023feynman}.

\subsubsection{Recursive formula for Type I' Laman graphs}

For the standard graph terminology used in this section, we refer the reader to
\cite{graver1993combinatorial}. Here we use a slightly different graph
convention from the one used in Section~\ref{s:feynman}: a graph \(\Gamma\) is
simply a pair \(\Gamma=(V(\Gamma),E(\Gamma))\) consisting of a vertex set and an
edge set, and we do not allow multiple edges between two vertices.

\begin{defn}[Laman]
Let \(\Gamma=(V(\Gamma),E(\Gamma))\) be a graph with vertex set \(V(\Gamma)\)
and edge set \(E(\Gamma)\). We say that \(\Gamma\) is a \textit{Laman graph} if
\[
|E(\Gamma)|=2|V(\Gamma)|-3
\]
and
\[
|E(\Gamma')|\leq 2|V(\Gamma')|-3
\]
for every subgraph \(\Gamma'\subseteq \Gamma\).
\end{defn}

\begin{rem}
When \(d=2\), the notion of a Laman graph originates from Laman's
characterization of generic rigidity for graphs embedded in a two-dimensional
real vector space; see \cite{Laman1970OnGA}. The graph-theoretic definition
used here also appears in \cite{budzik2023feynman}.
\end{rem}

Starting from a single edge, a Laman graph can be constructed by a sequence of
the following two types of Henneberg operations; see
Fig.~\ref{fig:HennebergIandII}.

\begin{itemize}
  \item \textbf{Henneberg I:} add a new vertex and connect it to two old vertices.
  \item \textbf{Henneberg II:} remove an edge and add a new vertex connected to
  the two endpoints of that edge and to one additional old vertex.
\end{itemize}

\begin{figure}[htbp]
\centering
\begin{tikzpicture}[
  x=1cm,
  y=1cm,
  vertex/.style={circle,fill=black,inner sep=1.35pt},
  newvertex/.style={circle,draw,fill=black!15,inner sep=1.75pt},
  edge/.style={line width=0.55pt},
  removed/.style={edge,densely dashed},
  arrow/.style={-{Stealth[length=2mm]},line width=0.55pt},
  every node/.style={font=\small}
]

\node[font=\small\bfseries] at (2.75,2.05) {Henneberg I};

\node[vertex,label={[label distance=4pt]225:\(i\)}] (HIi0) at (0,0) {};
\node[vertex,label={[label distance=4pt]315:\(j\)}] (HIj0) at (1.5,0) {};

\draw[arrow] (2.05,0.35) -- (2.95,0.35);

\node[vertex,label={[label distance=4pt]225:\(i\)}] (HIi1) at (3.7,0) {};
\node[vertex,label={[label distance=4pt]315:\(j\)}] (HIj1) at (5.2,0) {};
\node[newvertex,label={[label distance=5pt]90:\(\star\)}] (HIs) at (4.45,1.25) {};
\draw[edge] (HIs)--(HIi1);
\draw[edge] (HIs)--(HIj1);

\node[font=\small\bfseries] at (2.75,-1.25) {Henneberg II};

\node[vertex,label={[label distance=4pt]225:\(i\)}] (HIIi0) at (0,-3.05) {};
\node[vertex,label={[label distance=4pt]315:\(j\)}] (HIIj0) at (1.5,-3.05) {};
\node[vertex,label={[label distance=5pt]90:\(k\)}] (HIIk0) at (0.75,-1.8) {};
\draw[edge] (HIIi0)--(HIIj0);

\draw[arrow] (2.05,-2.65) -- (2.95,-2.65);

\node[vertex,label={[label distance=4pt]225:\(i\)}] (HIIi1) at (3.7,-3.05) {};
\node[vertex,label={[label distance=4pt]315:\(j\)}] (HIIj1) at (5.2,-3.05) {};
\node[vertex,label={[label distance=5pt]90:\(k\)}] (HIIk1) at (4.45,-1.8) {};
\node[newvertex,label={[label distance=5pt]0:\(\star\)}] (HIIs) at (4.45,-2.45) {};

\draw[removed] (HIIi1)--(HIIj1);
\draw[edge] (HIIs)--(HIIi1);
\draw[edge] (HIIs)--(HIIj1);
\draw[edge] (HIIs)--(HIIk1);

\end{tikzpicture}
\caption{The two Henneberg operations. In Henneberg II, the dashed edge denotes the old edge that is removed, and the gray vertex denotes the newly added vertex.}
\label{fig:HennebergIandII}
\end{figure}

In this section, we mainly consider the following variant of the Henneberg I
operation.

\begin{defn}[Henneberg I']
We define the following operation.
\begin{itemize}
  \item \textbf{Henneberg I':} add a new vertex and connect it to the two
  endpoints of an existing edge.
\end{itemize}
See Fig.~\ref{fig:HennebergIprime}.
\end{defn}

\begin{figure}[htbp]
\centering
\begin{tikzpicture}[
  x=1cm,
  y=1cm,
  vertex/.style={circle,fill=black,inner sep=1.35pt},
  newvertex/.style={circle,draw,fill=black!15,inner sep=1.75pt},
  edge/.style={line width=0.55pt},
  arrow/.style={-{Stealth[length=2mm]},line width=0.55pt},
  every node/.style={font=\small}
]

\node[vertex,label={[label distance=4pt]225:\(i\)}] (i0) at (0,0) {};
\node[vertex,label={[label distance=4pt]315:\(j\)}] (j0) at (1.6,0) {};
\draw[edge] (i0)--(j0);

\draw[arrow] (2.25,0.35) -- (3.15,0.35);

\node[vertex,label={[label distance=4pt]225:\(i\)}] (i1) at (3.8,0) {};
\node[vertex,label={[label distance=4pt]315:\(j\)}] (j1) at (5.4,0) {};
\node[newvertex,label={[label distance=5pt]90:\(\star\)}] (s) at (4.6,1.35) {};

\draw[edge] (i1)--(j1);
\draw[edge] (s)--(i1);
\draw[edge] (s)--(j1);

\end{tikzpicture}
\caption{A Henneberg I' operation: a new vertex \(\star\) is added and connected to the two endpoints of an existing edge.}
\label{fig:HennebergIprime}
\end{figure}

We will mainly focus on a special class of Laman graphs, which we call
\textit{Type I'}.

\begin{defn}[Type I' Laman graph]
A Laman graph is called \textit{Type I'} if it can be constructed by a sequence
of Henneberg I' operations.
\end{defn}

Suppose \(\Gamma\) is a Type I' Laman graph.
Fix a vertex \(o\in V(\Gamma)\) and an edge \(e=[ov]\in E(\Gamma)\).
Then \(\Gamma\) can be constructed by a sequence of Henneberg I' operations
\[
\mathbf{H}_{\Gamma}(o,e)
=
[(e_0=e,v_1),(e_1,v_2),\dots,(e_{k-1},v_k)].
\]
Here the vertex set of \(\Gamma\) is
\[
V(\Gamma)=\{o,v,v_1,\dots,v_k\}.
\]
Starting from the edge \(e_0=e=[ov]\), we add a new vertex \(v_1\) and draw two
edges connecting \(v_1\) to \(o\) and \(v\).
Inductively, at the \(i\)-th stage we add a new vertex \(v_i\), choose an edge
\(e_{i-1}\) of the current graph, allowing \(e_{i-1}=e_{i'}\) for some
\(i'<i-1\), and connect \(v_i\) to the two endpoints of \(e_{i-1}\).
This procedure is illustrated in Fig.~\ref{fig:TypeIprimeSequence}.

\begin{figure}[htbp]
\centering
\begin{tikzpicture}[
  x=1cm,
  y=1cm,
  vertex/.style={circle,fill=black,inner sep=1.35pt},
  newvertex/.style={circle,draw,fill=black!15,inner sep=1.65pt},
  edge/.style={line width=0.55pt},
  arrow/.style={-{Stealth[length=2mm]},line width=0.55pt},
  edgelabel/.style={font=\scriptsize,fill=white,inner sep=1.3pt},
  every node/.style={font=\small}
]

\begin{scope}
\node[vertex] (Lo) at (0,0) {};
\node[vertex] (Lv) at (0,2.1) {};
\node[vertex] (Lvone) at (2.55,1.5) {};
\node[vertex] (Lvtwo) at (1.25,3.55) {};
\node[vertex] (Lvthree) at (3.55,0.35) {};

\draw[edge] (Lo)--(Lv);
\draw[edge] (Lo)--(Lvone);
\draw[edge] (Lv)--(Lvone);
\draw[edge] (Lv)--(Lvtwo);
\draw[edge] (Lvone)--(Lvtwo);
\draw[edge] (Lv)--(Lvthree);
\draw[edge] (Lvone)--(Lvthree);

\node at (-0.35,-0.15) {\(o\)};
\node at (-0.35,2.25) {\(v\)};
\node at (2.9,1.55) {\(v_1\)};
\node at (1.25,3.9) {\(v_2\)};
\node at (3.9,0.35) {\(v_3\)};

\node[edgelabel] at (-0.32,1.05) {\(e\)};
\node[edgelabel] at (1.25,2.03) {\(e_1=e_2\)};
\end{scope}

\draw[arrow] (4.3,1.6) -- node[above=3pt]{\((e_3,v_4)\)} (5.75,1.6);

\begin{scope}[xshift=6.85cm]
\node[vertex] (Ro) at (0,0) {};
\node[vertex] (Rv) at (0,2.1) {};
\node[vertex] (Rvone) at (2.55,1.5) {};
\node[vertex] (Rvtwo) at (1.25,3.55) {};
\node[vertex] (Rvthree) at (3.55,0.35) {};
\node[newvertex] (Rvfour) at (4.25,3.05) {};

\draw[edge] (Ro)--(Rv);
\draw[edge] (Ro)--(Rvone);
\draw[edge] (Rv)--(Rvone);
\draw[edge] (Rv)--(Rvtwo);
\draw[edge] (Rvone)--(Rvtwo);
\draw[edge] (Rv)--(Rvthree);
\draw[edge] (Rvone)--(Rvthree);
\draw[edge] (Rvfour)--(Rvone);
\draw[edge] (Rvfour)--(Rvtwo);

\node at (-0.35,-0.15) {\(o\)};
\node at (-0.35,2.25) {\(v\)};
\node at (2.9,1.55) {\(v_1\)};
\node at (1.25,3.9) {\(v_2\)};
\node at (3.9,0.35) {\(v_3\)};
\node at (4.6,3.1) {\(v_4\)};

\node[edgelabel] at (-0.32,1.05) {\(e\)};
\node[edgelabel] at (1.25,2.03) {\(e_1=e_2\)};
\node[edgelabel] at (2.05,2.88) {\(e_3\)};
\end{scope}

\end{tikzpicture}
\caption{A sequence of Henneberg I' operations constructing a Type I' Laman graph.}
\label{fig:TypeIprimeSequence}
\end{figure}

\paragraph{Formal edge variables and graph weights.}
For each edge \(e=\{i,j\}\in E(\Gamma)\), we choose once and for all an
orientation \(\vec e=(i\to j)\). Let
\[
\vec E(\Gamma)=(\vec e_1,\dots,\vec e_N)
\]
be the resulting ordered list of oriented edges. For \(\vec e=(i\to j)\), we
write
\[
\mathfrak z_{\vec e}\define\mathfrak z_{ij}=-\mathfrak z_{ji}.
\]
For later use, we record the corresponding incidence sign. If
\(\vec e=(i\to j)\), set
\[
  t(\vec e)=i,
  \qquad
  h(\vec e)=j,
\]
and define
\[
  \varepsilon_v(\vec e)
  \define
  \begin{cases}
  +1, & v=t(\vec e),\\
  -1, & v=h(\vec e),\\
  0,  & v\notin \vec e.
  \end{cases}
\]
Thus \(\varepsilon_v(\vec e)\) records whether the formal edge variable
\(\mathfrak z_{\vec e}\) points away from or toward the vertex \(v\).
The variables \(\mathfrak z_{ij}^{\1},\mathfrak z_{ij}^{\2}\) are formal
variables recording derivatives of the propagator, as in
Section~\ref{ResidueOp}; they are not coordinate differences.

For an oriented edge \(\vec e=(i\to j)\), set
\[
P_{\vec e}(\mathfrak z_{\vec e})
\define
P_{ij}(\mathfrak z_{ij}).
\]
With the convention \(\mathfrak z_{ji}=-\mathfrak z_{ij}\), the derivative
generating propagator is symmetric in dimension two:
\[
P_{ji}(\mathfrak z_{ji})=P_{ij}(\mathfrak z_{ij}).
\]
Indeed, \(x_{ji}=-x_{ij}\), \(P_{ji}=P_{ij}\), and
\((x_{ji}\mid\mathfrak z_{ji})=(x_{ij}\mid\mathfrak z_{ij})\).

We define
\[
W_\Gamma(\mathfrak z_{\vec e})
\define
P_{\vec e_1}(\mathfrak z_{\vec e_1})
\cdots
P_{\vec e_N}(\mathfrak z_{\vec e_N})
\,
d\mathbf z_{V(\Gamma)}^{\blacklozenge}.
\]
Since the propagators are odd one-forms, the displayed order of the product is
part of the convention. Changing the edge order changes the sign by the sign of
the corresponding permutation. When the shifted holomorphic volume form is clear
from context, we sometimes omit it and write \(W_\Gamma(\mathfrak z_{\vec e})\)
for the propagator product alone.

A Type I' Laman graph can in general be constructed from different sequences of
Henneberg I' operations,
\[
\mathbf{H}_{\Gamma}(o,e),
\qquad
\mathbf{H}'_{\Gamma}(o,e),
\]
starting from the same initial data \((o,e)\).
Using the Arnold relation, we can show that the propagator part of
\(W_{\Gamma}(\mathfrak z_{\vec e})\) is exact in the Jouanolou model. Hence,
by the \(L_{\infty}\)-relations for chiral operations, the computation reduces
to chiral operations associated to smaller graphs.
The Type I' condition ensures that these smaller graphs are still Type I'
Laman graphs.
In this section, we carry out this construction in detail and obtain an
explicit recursive formula for the chiral operation associated to a Type I'
Laman graph.

\begin{thm}\label{MainTheoremRecursiveRS}
Let \(\Gamma\) be a Type I' Laman graph, and let
\(\mathbf{H}_{\Gamma}(o,e)\) be a sequence of Henneberg I' operations
constructing \(\Gamma\). Put
\[
m\define|\mathbf{H}_{\Gamma}(o,e)|=|V(\Gamma)|-2,
\qquad
n\define|V(\Gamma)|-1.
\]
Choose an auxiliary ordering of \(V(\Gamma)\setminus\{o\}\). Then
\((\lambda_1,\dots,\lambda_n)\) denotes the corresponding ordered tuple of
vertex-labelled variables \((\lambda_v)_{v\in V(\Gamma)\setminus\{o\}}\);
substitutions such as
\(\lambda_i\mapsto \lambda_i+(1-s_{\star})\lambda_{\star}\) are always made
at the vertex \(i\).

Then there are polynomials
\[
f^{\Gamma,\mathbf{H}_{\Gamma}(o,e)}_{v,\vec e}
\in
\mathbb C[r_1,s_1,\dots,r_m,s_m],
\qquad
v\in V(\Gamma)\setminus\{o\},\quad
\vec e\in \vec E(\Gamma),
\]
and
\[
\mathbf{G}^{r,s}_{\Gamma,\mathbf{H}_{\Gamma}(o,e)}
\in
\mathbb C\left[
r_1,s_1,\dots,r_m,s_m,
\lambda_v^{\sss}
\ \middle|\
v\in V(\Gamma)\setminus\{o\},\ \sss\in\{\1,\2\}
\right],
\]
such that
\[
\mu_{V(\Gamma)}
\left(W_{\Gamma}(\mathfrak z_{\vec e})\right)
=
\left(
\int_{([0,1]\times[0,1])^{m}}
e^{\mathbf{W}^{r,s}_{\Gamma,\mathbf{H}_{\Gamma}(o,e)}}
\cdot
\mathbf{G}^{r,s}_{\Gamma,\mathbf{H}_{\Gamma}(o,e)}
\cdot
\bigwedge_{k=1}^{m}(dr_k\wedge ds_k)
\right)
d\mathbf z_o^{\blacklozenge},
\]
where
\[
\mathbf{W}^{r,s}_{\Gamma,\mathbf{H}_{\Gamma}(o,e)}
=
\mathbf{W}^{r,s}_{\Gamma,\mathbf{H}_{\Gamma}(o,e)}
(\lambda;\mathfrak z_{\vec e})
\define
\sum_{v\in V(\Gamma)\setminus\{o\}}
\left(
\lambda_v
\mid
\sum_{\vec e\in \vec E(\Gamma)}
f^{\Gamma,\mathbf{H}_{\Gamma}(o,e)}_{v,\vec e}
\mathfrak z_{\vec e}
\right).
\]

Now construct a new graph \(\Gamma^{\star}\) by one Henneberg I' operation
\[
\mathbf{H}_{\Gamma^{\star}}(o,e)
=
[\mathbf{H}_{\Gamma}(o,e),(e^\triangleright,\star)].
\]
Then the following recursive formulas compute
\[
\mu_{V(\Gamma^{\star})}
\left(W_{\Gamma^{\star}}(\mathfrak z_{\vec e})\right).
\]

\begin{enumerate}
\item Suppose that \(\star\) is connected to \(i\) and \(j\), with
\(i,j\in V(\Gamma)\setminus\{o\}\). We order the two endpoints so that the chosen orientation of the distinguished old edge is
\[
  \vec{e^\triangleright}=(i\to j).
\]
Introduce new variables by
\[
\begin{cases}
\mathfrak z^\triangleright_{ij}
=
\mathfrak z_{ij}
+
(1-r_{\star})
(\mathfrak z_{\star j}-\mathfrak z_{\star i}-\mathfrak z_{ij}),
& \\[2mm]
\mathfrak z^\triangleright_{\vec e}
=
\mathfrak z_{\vec e},
& \mbox{if } \vec e\neq \vec{e^\triangleright},\\[2mm]
\lambda^\triangleright_i
=
\lambda_i+(1-s_{\star})\lambda_{\star},
& \\[2mm]
\lambda^\triangleright_j
=
\lambda_j+s_{\star}\lambda_{\star},
& \\[2mm]
\lambda^\triangleright_v
=
\lambda_v,
& \mbox{if } v\neq i,j.
\end{cases}
\]
Then
\begin{align*}
\mathbf{W}^{r,s}_{\Gamma^{\star},\mathbf{H}_{\Gamma^{\star}}(o,e)}
&=
\mathbf{W}^{r,s}_{\Gamma,\mathbf{H}_{\Gamma}(o,e)}
(\lambda^\triangleright;\mathfrak z^\triangleright_{\vec e})
-
\left(
\lambda_{\star}
\mid
(1-s_{\star})\mathfrak z_{\star i}
+
s_{\star}\mathfrak z_{\star j}
\right) \\
&=
\sum_{v\in V(\Gamma)\setminus\{o\}}
\left(
\lambda^\triangleright_v
\mid
\sum_{\vec e\in \vec E(\Gamma)}
f^{\Gamma,\mathbf{H}_{\Gamma}(o,e)}_{v,\vec e}
\mathfrak z^\triangleright_{\vec e}
\right)
-
\left(
\lambda_{\star}
\mid
(1-s_{\star})\mathfrak z_{\star i}
+
s_{\star}\mathfrak z_{\star j}
\right),
\end{align*}
and
\begin{align*}
\mathbf{G}^{r,s}_{\Gamma^{\star},\mathbf{H}_{\Gamma^{\star}}(o,e)}
&=
\left(
\partial_{\mathfrak z_{ij}}
\mathbf{W}^{r,s}_{\Gamma,\mathbf{H}_{\Gamma}(o,e)}
(\lambda;\mathfrak z_{\vec e})
\wedge
\lambda_{\star}
\right)
\cdot r_{\star}\cdot
\mathbf{G}^{r,s}_{\Gamma,\mathbf{H}_{\Gamma}(o,e)}
(\lambda^\triangleright) \\
&=
\left(
\partial_{\mathfrak z_{ij}}
\mathbf{W}^{r,s}_{\Gamma,\mathbf{H}_{\Gamma}(o,e)}
(\lambda;\mathfrak z_{\vec e})
\wedge
\lambda_{\star}
\right)
\cdot r_{\star}\cdot
\mathbf{G}^{r,s}_{\Gamma,\mathbf{H}_{\Gamma}(o,e)}
(\lambda_1,\dots,
\lambda_i+(1-s_{\star})\lambda_{\star},
\dots,
\lambda_j+s_{\star}\lambda_{\star},
\dots,\lambda_n).
\end{align*}

\item Suppose that \(\star\) is connected to \(i\) and \(o\), and write
\[
\vec{e^\triangleright}=io.
\]
Introduce new variables by
\[
\begin{cases}
\mathfrak z^\triangleright_{io}
=
\mathfrak z_{io}
+
(1-r_{\star})
(\mathfrak z_{\star o}-\mathfrak z_{\star i}-\mathfrak z_{io}),
& \\[2mm]
\mathfrak z^\triangleright_{\vec e}
=
\mathfrak z_{\vec e},
& \mbox{if } \vec e\neq \vec{e^\triangleright},\\[2mm]
\lambda^\triangleright_i
=
\lambda_i+(1-s_{\star})\lambda_{\star},
& \\[2mm]
\lambda^\triangleright_v
=
\lambda_v,
& \mbox{if } v\neq i.
\end{cases}
\]
Then
\begin{align*}
\mathbf{W}^{r,s}_{\Gamma^{\star},\mathbf{H}_{\Gamma^{\star}}(o,e)}
&=
\mathbf{W}^{r,s}_{\Gamma,\mathbf{H}_{\Gamma}(o,e)}
(\lambda^\triangleright;\mathfrak z^\triangleright_{\vec e})
-
\left(
\lambda_{\star}
\mid
(1-s_{\star})\mathfrak z_{\star i}
+
s_{\star}\mathfrak z_{\star o}
\right) \\
&=
\sum_{v\in V(\Gamma)\setminus\{o\}}
\left(
\lambda^\triangleright_v
\mid
\sum_{\vec e\in \vec E(\Gamma)}
f^{\Gamma,\mathbf{H}_{\Gamma}(o,e)}_{v,\vec e}
\mathfrak z^\triangleright_{\vec e}
\right)
-
\left(
\lambda_{\star}
\mid
(1-s_{\star})\mathfrak z_{\star i}
+
s_{\star}\mathfrak z_{\star o}
\right),
\end{align*}
and
\begin{align*}
\mathbf{G}^{r,s}_{\Gamma^{\star},\mathbf{H}_{\Gamma^{\star}}(o,e)}
&=
\left(
\partial_{\mathfrak z_{io}}
\mathbf{W}^{r,s}_{\Gamma,\mathbf{H}_{\Gamma}(o,e)}
(\lambda;\mathfrak z_{\vec e})
\wedge
\lambda_{\star}
\right)
\cdot r_{\star}\cdot
\mathbf{G}^{r,s}_{\Gamma,\mathbf{H}_{\Gamma}(o,e)}
(\lambda^\triangleright) \\
&=
\left(
\partial_{\mathfrak z_{io}}
\mathbf{W}^{r,s}_{\Gamma,\mathbf{H}_{\Gamma}(o,e)}
(\lambda;\mathfrak z_{\vec e})
\wedge
\lambda_{\star}
\right)
\cdot r_{\star}\cdot
\mathbf{G}^{r,s}_{\Gamma,\mathbf{H}_{\Gamma}(o,e)}
(\lambda_1,\dots,
\lambda_i+(1-s_{\star})\lambda_{\star},
\dots,\lambda_n).
\end{align*}
\end{enumerate}
\end{thm}
\begin{rem}
The parameters \(r\) and \(s\) here are closely related to the Schwinger
parameters in the previous section. However, the precise relation is
complicated; we leave it to future investigation.
\end{rem}

Using the above formula, we can write down formulas for some loop diagrams and
match them with results in \cite{budzik2023feynman}.
The remarkable point is that the chiral operation associated to a Type I'
Laman graph is completely constrained by the \(L_{\infty}\)-relations and can
be solved recursively without explicitly evaluating Feynman integrals.
In other words, the recursive formula gives another way to compute these
complicated Feynman integrals.
We provide some explicit examples.
\subsection{Example: one-loop diagrams}\label{s:2d1loop}

\begin{figure}[htbp]
\centering
\begin{tikzpicture}[
  x=1cm,
  y=1cm,
  vertex/.style={circle,fill=black,inner sep=1.35pt},
  newvertex/.style={circle,draw,fill=black!15,inner sep=1.65pt},
  edge/.style={line width=0.55pt},
  arrow/.style={-{Stealth[length=2mm]},line width=0.55pt},
  edgelabel/.style={font=\scriptsize,fill=white,inner sep=1.2pt},
  every node/.style={font=\small}
]

\node[vertex] (Lo) at (0,0) {};
\node[vertex] (Lone) at (0,2.0) {};
\draw[edge] (Lo)--(Lone);

\node[left=5pt of Lo] {\(o\)};
\node[left=5pt of Lone] {\(1\)};
\node[edgelabel] at (-0.28,1.0) {\(e\)};

\draw[arrow] (1.15,1.0) -- node[above=3pt]{\((e,2)\)} (2.6,1.0);

\node[vertex] (Ro) at (3.8,0) {};
\node[vertex] (Rone) at (3.8,2.0) {};
\node[newvertex] (Rtwo) at (6.0,1.0) {};

\draw[edge] (Ro)--(Rone);
\draw[edge] (Rone)--(Rtwo);
\draw[edge] (Rtwo)--(Ro);

\node[left=5pt of Ro] {\(o\)};
\node[left=5pt of Rone] {\(1\)};
\node[right=5pt of Rtwo] {\(2\)};
\node[edgelabel] at (3.52,1.0) {\(e\)};

\end{tikzpicture}
\caption{The one-loop graph obtained from the edge \(e=[1o]\) by the Henneberg I' operation \((e,2)\).}
\label{fig:OneLoopHenneberg}
\end{figure}

We start from
\[
\lineW(\lambda;\mathfrak z_{\vec e})
=
-(\lambda_1|\mathfrak z_{1o}).
\]
By the recursive formula, we have
\begin{align*}
\triangleW(\lambda;\mathfrak z_{\vec e})
&=
-(\lambda^\triangleright_1|\mathfrak z^\triangleright_{1o})
-
(1-s_2)(\lambda_2|\mathfrak z_{21})
-
s_2(\lambda_2|\mathfrak z_{2o})                                      \\
&=
-\left(
\lambda_1+(1-s_2)\lambda_2
\mid
\mathfrak z_{1o}
+
(1-r_2)(\mathfrak z_{2o}-\mathfrak z_{21}-\mathfrak z_{1o})
\right)                                                              \\
&\qquad
-
(1-s_2)(\lambda_2|\mathfrak z_{21})
-
s_2(\lambda_2|\mathfrak z_{2o})                                      \\
&=
-(\lambda_1|\mathfrak z_{1o})
-(\lambda_2|\mathfrak z_{2o})
-(1-r_2)
\left(
\lambda_1
\mid
\mathfrak z_{2o}-\mathfrak z_{21}-\mathfrak z_{1o}
\right)                                                              \\
&\qquad
+
r_2(1-s_2)
\left(
\lambda_2
\mid
\mathfrak z_{2o}-\mathfrak z_{21}-\mathfrak z_{1o}
\right).
\end{align*}
Moreover,
\[
\triangleG
=
-r_2\cdot
(\partial_{\mathfrak z_{1o}}\lineW)\wedge\lambda_2
=
r_2\cdot\lambda_1\wedge\lambda_2.
\]

Put
\[
\mathfrak z
=
\mathfrak z_{2o}-\mathfrak z_{21}-\mathfrak z_{1o}
=
\mathfrak z_{2o}+\mathfrak z_{12}-\mathfrak z_{1o}.
\]
The three-point operation is given by
\begin{align}
&\mu_{\{o,1,2\}}
\left(
P_{1o}(\mathfrak z_{1o})
P_{12}(\mathfrak z_{12})
P_{2o}(\mathfrak z_{2o})
d\mathbf z^{\blacklozenge}_{\{o,1,2\}}
\right)                                                        \notag \\
&\quad =
\int_{[0,1]\times[0,1]}
e^{\triangleW(\lambda;\mathfrak z_{\vec e})}
\cdot
\triangleG
\cdot
dr_2\,ds_2
\cdot
d\mathbf z^{\blacklozenge}_o                                  \notag \\
&\quad =
e^{-(\lambda_2|\mathfrak z_{2o})-(\lambda_1|\mathfrak z_{1o})}
\cdot
\lambda_1\wedge\lambda_2                                      \notag \\
&\qquad\qquad\cdot
\left(
-\frac{1}{(\lambda_1|\mathfrak z)(\lambda_2|\mathfrak z)}
+
\frac{e^{-(\lambda_1|\mathfrak z)}}
{(\lambda_1|\mathfrak z)((\lambda_1+\lambda_2)|\mathfrak z)}
+
\frac{e^{(\lambda_2|\mathfrak z)}}
{(\lambda_2|\mathfrak z)((\lambda_1+\lambda_2)|\mathfrak z)}
\right)
d\mathbf z^{\blacklozenge}_o                                  \notag \\
&\quad =
e^{-(\lambda_2|\mathfrak z_{2o})-(\lambda_1|\mathfrak z_{1o})}
\cdot
\lambda_1\wedge\lambda_2
\cdot
\int_{\ell_1+\ell_2\leq 1}
e^{-\ell_1(\lambda_1|\mathfrak z)+\ell_2(\lambda_2|\mathfrak z)}
\,d\ell_1\,d\ell_2
\cdot
d\mathbf z^{\blacklozenge}_o .
\label{3PointOneloop}
\end{align}
After identifying the base vertex \(o\) with vertex \(3\) in their notation, the one-loop formula above agrees exactly with the shifted triangle formula in
\cite[Section~4.3, Eqs.~(4.22)--(4.23)]{budzik2023feynman}.

The leading terms are as follows. Put
\[
\Lambda\define\lambda_1\wedge\lambda_2.
\]
Then
\begin{align}
\mu_{\{o,1,2\}}
\left(
P_{1o}P_{12}P_{2o}
d\mathbf z^{\blacklozenge}_{\{o,1,2\}}
\right)
&=
\frac{1}{2}\Lambda
d\mathbf z^{\blacklozenge}_o ,
\label{OneloopConst}
\\[2mm]
\mu_{\{o,1,2\}}
\left(
\partial_{z_1^{\scriptstyle\ii}}P_{1o}
\cdot
P_{12}P_{2o}
d\mathbf z^{\blacklozenge}_{\{o,1,2\}}
\right)
&=
-\Lambda
\left(
\frac{1}{3}\lambda_1^{\ii}
+
\frac{1}{6}\lambda_2^{\ii}
\right)
d\mathbf z^{\blacklozenge}_o ,
\label{Oneloop1order}
\\[2mm]
\mu_{\{o,1,2\}}
\left(
\partial_{z_1^{\scriptstyle\ii}}P_{1o}
\cdot
P_{12}
\cdot
\partial_{z_2^{\scriptstyle\widetilde{\ii}}}P_{2o}
d\mathbf z^{\blacklozenge}_{\{o,1,2\}}
\right)
&=
\frac{\Lambda}{24}
\Big(
2\lambda_1^{\ii}\lambda_1^{\widetilde{\ii}}
+
5\lambda_1^{\ii}\lambda_2^{\widetilde{\ii}}
+
\lambda_2^{\ii}\lambda_1^{\widetilde{\ii}}
+
2\lambda_2^{\ii}\lambda_2^{\widetilde{\ii}}
\Big)
d\mathbf z^{\blacklozenge}_o ,
\label{Oneloop2order}
\end{align}
\begin{align}
&\mu_{\{o,1,2\}}
\left(
\partial_{z_1^{\scriptstyle\ii}}P_{1o}
\cdot
\partial_{z_2^{\scriptstyle\widetilde{\widetilde{\ii}}}}P_{12}
\cdot
\partial_{z_2^{\scriptstyle\widetilde{\ii}}}P_{2o}
d\mathbf z^{\blacklozenge}_{\{o,1,2\}}
\right)                                                        \notag \\
&\qquad =
\frac{\Lambda}{120}
\Big[
\lambda_1^{\ii}
\big(
\lambda_1^{\widetilde{\widetilde{\ii}}}
(4\lambda_1^{\widetilde{\ii}}+7\lambda_2^{\widetilde{\ii}})
-
\lambda_2^{\widetilde{\widetilde{\ii}}}
(3\lambda_1^{\widetilde{\ii}}+7\lambda_2^{\widetilde{\ii}})
\big)                                                         \notag \\
&\qquad\qquad\quad
+
\lambda_2^{\ii}
\big(
\lambda_1^{\widetilde{\widetilde{\ii}}}
(2\lambda_1^{\widetilde{\ii}}+3\lambda_2^{\widetilde{\ii}})
-
\lambda_2^{\widetilde{\widetilde{\ii}}}
(2\lambda_1^{\widetilde{\ii}}+4\lambda_2^{\widetilde{\ii}})
\big)
\Big]
d\mathbf z^{\blacklozenge}_o .
\label{Oneloop3order}
\end{align}
\subsection{Example: two-loop diagrams}

\begin{figure}[htbp]
\centering
\begin{tikzpicture}[
  x=1cm,
  y=1cm,
  vertex/.style={circle,fill=black,inner sep=1.35pt},
  newvertex/.style={circle,draw,fill=black!15,inner sep=1.65pt},
  edge/.style={line width=0.55pt},
  arrow/.style={-{Stealth[length=2mm]},line width=0.55pt},
  edgelabel/.style={font=\scriptsize,fill=white,inner sep=1.2pt},
  every node/.style={font=\small}
]

\node[vertex] (Lo) at (0,0) {};
\node[vertex] (Lone) at (0,2.0) {};
\node[vertex] (Ltwo) at (2.25,1.0) {};

\draw[edge] (Lo)--(Lone);
\draw[edge] (Lone)--(Ltwo);
\draw[edge] (Ltwo)--(Lo);

\node[left=5pt of Lo] {\(o\)};
\node[left=5pt of Lone] {\(1\)};
\node[right=5pt of Ltwo] {\(2\)};
\node[edgelabel] at (1.25,0.50) {\([2o]\)};

\draw[arrow] (3.05,1.0) -- node[above=3pt]{\(([2o],3)\)} (4.55,1.0);

\node[vertex] (Ro) at (5.45,0) {};
\node[vertex] (Rone) at (5.45,2.0) {};
\node[vertex] (Rtwo) at (7.75,2.0) {};
\node[newvertex] (Rthree) at (7.75,0) {};

\draw[edge] (Ro)--(Rone);
\draw[edge] (Rone)--(Rtwo);
\draw[edge] (Rtwo)--(Ro);
\draw[edge] (Rtwo)--(Rthree);
\draw[edge] (Rthree)--(Ro);

\node[left=5pt of Ro] {\(o\)};
\node[left=5pt of Rone] {\(1\)};
\node[right=5pt of Rtwo] {\(2\)};
\node[right=5pt of Rthree] {\(3\)};
\node[edgelabel] at (6.75,0.95) {\([2o]\)};

\end{tikzpicture}
\caption{The two-loop graph obtained by applying the Henneberg I' operation
\(([2o],3)\) to the one-loop graph.}
\label{fig:TwoLoopHenneberg}
\end{figure}

We start from the one-loop graph and apply a Henneberg I' operation along the
edge \([2o]\), adding the new vertex \(3\). Define the two loop variables
\[
\mathfrak z_{[21o]}
\define
\mathfrak z_{2o}-\mathfrak z_{21}-\mathfrak z_{1o},
\qquad
\mathfrak z_{[32o]}
\define
\mathfrak z_{3o}-\mathfrak z_{32}-\mathfrak z_{2o}.
\]
The recursive formula gives
\begin{align*}
\agraphW
&=
\triangleW(\lambda^\triangleright;\mathfrak z^\triangleright_{\vec e})
-
\left(
\lambda_3
\mid
(1-s_3)\mathfrak z_{32}
+
s_3\mathfrak z_{3o}
\right)                                                        \\
&=
-(\lambda_1|\mathfrak z_{1o})
-(\lambda_2|\mathfrak z_{2o})
-(\lambda_3|\mathfrak z_{3o})                                  \\
&\quad
+
\left(
\lambda_1
\mid
-(1-r_2)\mathfrak z_{[21o]}
-
(1-r_3)(1-r_2)\mathfrak z_{[32o]}
\right)                                                        \\
&\quad
+
\left(
\lambda_2
\mid
r_2(1-s_2)\mathfrak z_{[21o]}
+
(1-r_3)\bigl(r_2(1-s_2)-1\bigr)\mathfrak z_{[32o]}
\right)                                                        \\
&\quad
+
\left(
\lambda_3
\mid
r_2(1-s_2)(1-s_3)\mathfrak z_{[21o]}
+
(1-s_3)
\bigl(1-(1-r_3)(1-r_2(1-s_2))\bigr)
\mathfrak z_{[32o]}
\right).
\end{align*}
Similarly,
\begin{align*}
\agraphG
&=
\left(
\partial_{\mathfrak z_{2o}}\triangleW
\right)
\wedge\lambda_3
\cdot r_3\cdot
\triangleG(\lambda^\triangleright)                            \\
&=
-
r_3r_2
\left(
\bigl((1-r_2)\lambda_1+(1-r_2(1-s_2))\lambda_2\bigr)
\wedge\lambda_3
\right)
\cdot
\left(
\lambda_1
\wedge
\bigl(\lambda_2+(1-s_3)\lambda_3\bigr)
\right).
\end{align*}
Therefore
\begin{align*}
&\mu_{\{o,1,2,3\}}
\left(
P_{3o}(\mathfrak z_{3o})
P_{32}(\mathfrak z_{32})
P_{1o}(\mathfrak z_{1o})
P_{12}(\mathfrak z_{12})
P_{2o}(\mathfrak z_{2o})
d\mathbf z^{\blacklozenge}_{\{o,1,2,3\}}
\right)                                                        \\
&\qquad =
\int_{([0,1]\times[0,1])^2}
e^{\agraphW}
\cdot
\agraphG
\,dr_2\,ds_2\,dr_3\,ds_3
\cdot
d\mathbf z^{\blacklozenge}_o .
\end{align*}
In particular, the constant term, obtained by setting
\(\mathfrak z_{\vec e}=0\), is
\begin{align*}
&\mu_{\{o,1,2,3\}}
\left(
P_{3o}P_{32}P_{1o}P_{12}P_{2o}
d\mathbf z^{\blacklozenge}_{\{o,1,2,3\}}
\right)                                                        \\
&\qquad =
\int_{([0,1]\times[0,1])^2}
\agraphG
\,dr_2\,ds_2\,dr_3\,ds_3
\cdot
d\mathbf z^{\blacklozenge}_o                                  \\
&\qquad =
\frac{1}{24}
\left(
\lambda_1\wedge(\lambda_3+2\lambda_2)
\right)
\cdot
\left(
\lambda_3\wedge(\lambda_1+2\lambda_2)
\right)
d\mathbf z^{\blacklozenge}_o .
\end{align*}
After the relabelling
\((\lambda_1,\lambda_2,\lambda_3)\) to \((\lambda_1,\lambda_3,\lambda_2)\),
the constant term above agrees exactly with the bitriangle formula
\(f[\lambda_1,\lambda_2,\lambda_3;0,0]\) in
\cite[Section~4.6, Eq.~(4.44)]{budzik2023feynman}. 

\subsection{Example: three-loop diagrams}

\begin{figure}[htbp]
\centering
\begin{tikzpicture}[
  x=1cm,
  y=1cm,
  vertex/.style={circle,fill=black,inner sep=1.35pt},
  newvertex/.style={circle,draw,fill=black!15,inner sep=1.65pt},
  edge/.style={line width=0.55pt},
  arrow/.style={-{Stealth[length=2mm]},line width=0.55pt},
  edgelabel/.style={font=\scriptsize,fill=white,inner sep=1.2pt},
  every node/.style={font=\small}
]

\node[vertex] (Lo) at (0,0) {};
\node[vertex] (Lone) at (0,2.15) {};
\node[vertex] (Ltwo) at (2.3,2.15) {};
\node[vertex] (Lthree) at (2.3,0) {};

\draw[edge] (Lo)--(Lone);
\draw[edge] (Lone)--(Ltwo);
\draw[edge] (Ltwo)--(Lo);
\draw[edge] (Ltwo)--(Lthree);
\draw[edge] (Lthree)--(Lo);

\node[left=5pt of Lo] {\(o\)};
\node[left=5pt of Lone] {\(1\)};
\node[right=5pt of Ltwo] {\(2\)};
\node[right=5pt of Lthree] {\(3\)};
\node[edgelabel] at (1.25,0.55) {\([3o]\)};

\draw[arrow] (3.15,1.05) -- node[above=3pt]{\(([3o],4)\)} (4.65,1.05);

\node[vertex] (Ro) at (5.55,0) {};
\node[vertex] (Rone) at (5.55,2.15) {};
\node[vertex] (Rtwo) at (7.85,2.15) {};
\node[vertex] (Rthree) at (7.85,0) {};
\node[newvertex] (Rfour) at (10.05,1.05) {};

\draw[edge] (Ro)--(Rone);
\draw[edge] (Rone)--(Rtwo);
\draw[edge] (Rtwo)--(Ro);
\draw[edge] (Rtwo)--(Rthree);
\draw[edge] (Rthree)--(Ro);
\draw[edge] (Rthree)--(Rfour);
\draw[edge] (Rfour)--(Ro);

\node[left=5pt of Ro] {\(o\)};
\node[left=5pt of Rone] {\(1\)};
\node[above=5pt of Rtwo] {\(2\)};
\node[below=5pt of Rthree] {\(3\)};
\node[right=5pt of Rfour] {\(4\)};
\node[edgelabel] at (6.75,0.55) {\([3o]\)};

\end{tikzpicture}
\caption{The three-loop graph obtained by applying the Henneberg I' operation
\(([3o],4)\) to the two-loop graph.}
\label{fig:ThreeLoopHenneberg}
\end{figure}

We now apply a Henneberg I' operation to the two-loop graph along the edge
\([3o]\), adding the new vertex \(4\). A direct recursive computation gives
\(\bgraphW\), but the expression is lengthy; for the constant term below, it is enough
to record the corresponding factor \(\bgraphG\). We have
\begin{align*}
\bgraphG
&=
r_4\cdot
\agraphG(\lambda^\triangleright)
\cdot
\left(
\partial_{\mathfrak z_{3o}}\agraphW
\right)
\wedge\lambda_4                                      \\
&=
r_4r_3r_2
\left(
\left(
(1-r_2)\lambda_1+
(1-r_2(1-s_2))\lambda_2
\right)
\wedge
\left(
\lambda_3+(1-s_4)\lambda_4
\right)
\right)                                             \\
&\quad\cdot
\left(
\lambda_1
\wedge
\left(
\lambda_2+
(1-s_3)\lambda_3+
(1-s_3)(1-s_4)\lambda_4
\right)
\right)                                             \\
&\quad\cdot
\left(
(1-r_3)(1-r_2)\lambda_1
-
(1-r_3)(r_2(1-s_2)-1)\lambda_2
\right.                                             \\
&\qquad\qquad\left.
+
\left(
s_3+
(1-s_3)(1-r_3)(1-r_2(1-s_2))
\right)\lambda_3
\right)
\wedge\lambda_4 .
\end{align*}

The constant term, obtained by setting
\(\mathfrak z_{\vec e}=0\), is
\begin{align*}
&\mu_{\{o,1,2,3,4\}}
\left(
P_{4o}P_{43}P_{3o}P_{32}P_{1o}P_{12}P_{2o}
d\mathbf z^{\blacklozenge}_{\{o,1,2,3,4\}}
\right)                                                        \\
&\qquad =
\int_{([0,1]\times[0,1])^3}
\bgraphG
\,dr_2\,ds_2\,dr_3\,ds_3\,dr_4\,ds_4
\cdot
d\mathbf z^{\blacklozenge}_o                                  \\
&\qquad =
F(\lambda_1,\lambda_2,\lambda_3,\lambda_4)
\cdot
d\mathbf z^{\blacklozenge}_o .
\end{align*}
After the relabelling
\[
(\lambda_1,\lambda_2,\lambda_3,\lambda_4)
\quad\text{to}\quad
(\lambda_1,\lambda_4,\lambda_2,\lambda_3),
\]
the polynomial \(F(\lambda_1,\lambda_2,\lambda_3,\lambda_4)\) agrees exactly with
the tritriangle formula
\(f[\lambda_1,\lambda_2,\lambda_3,\lambda_4;0,0,0]\) in
\cite[Section~4.8, Eq.~(4.67)]{budzik2023feynman}.
We have provided explicit \texttt{Mathematica} code for this calculation hosted at \cite{code}. 

\subsection{Proof of Theorem~\ref{MainTheoremRecursiveRS}}

We now turn to the proof of Theorem~\ref{MainTheoremRecursiveRS}.
We begin with three elementary integral identities that will be used in the
induction step.

\begin{lem}\label{IntegralXP}
Let \(F\) and \(G\) be polynomials in two variables, and let
\(\ii\in\{\1,\2\}\). Then
\begin{align*}
&\left(
G(\partial_{z_1^{\scriptstyle\1}},\partial_{z_1^{\scriptstyle\2}})
x^{\ii}_{12}(\mathfrak z)
\right)
\cdot
\left(
F(\partial_{z_1^{\scriptstyle\1}},\partial_{z_1^{\scriptstyle\2}})
P_{12}(\mathfrak w)
\right)                                                        \\
&\qquad =
G(\partial_{\mathfrak z^{\scriptstyle\1}},\partial_{\mathfrak z^{\scriptstyle\2}})
F(\partial_{\mathfrak w^{\scriptstyle\1}},\partial_{\mathfrak w^{\scriptstyle\2}})
(-\partial_{\mathfrak w^{\scriptstyle\ii}})
\int_0^1
P_{12}\bigl((1-s)\mathfrak z+s\mathfrak w\bigr)\,ds             \\
&\qquad =
\int_0^1
G\bigl((1-s)\partial_{z_1^{\scriptstyle\1}},(1-s)\partial_{z_1^{\scriptstyle\2}}\bigr)
F\bigl(s\partial_{z_1^{\scriptstyle\1}},s\partial_{z_1^{\scriptstyle\2}}\bigr)
(-s\partial_{z_1^{\scriptstyle\ii}})
P_{12}\bigl((1-s)\mathfrak z+s\mathfrak w\bigr)\,ds .
\end{align*}
\end{lem}

\begin{proof}
It is enough to compute directly. We have
\begin{align*}
&\left(
G(\partial_{z_1^{\scriptstyle\1}},\partial_{z_1^{\scriptstyle\2}})
x^{\ii}_{12}(\mathfrak z)
\right)
\cdot
\left(
F(\partial_{z_1^{\scriptstyle\1}},\partial_{z_1^{\scriptstyle\2}})
P_{12}(\mathfrak w)
\right)                                                        \\
&\quad =
G(\partial_{\mathfrak z^{\scriptstyle\1}},\partial_{\mathfrak z^{\scriptstyle\2}})
F(\partial_{\mathfrak w^{\scriptstyle\1}},\partial_{\mathfrak w^{\scriptstyle\2}})
\left(
x^{\ii}_{12}(\mathfrak z)P_{12}(\mathfrak w)
\right)                                                        \\
&\quad =
G(\partial_{\mathfrak z^{\scriptstyle\1}},\partial_{\mathfrak z^{\scriptstyle\2}})
F(\partial_{\mathfrak w^{\scriptstyle\1}},\partial_{\mathfrak w^{\scriptstyle\2}})
\left(
\frac{x^{\ii}_{12}(\mathfrak z)P_{12}(\mathfrak z)}
{\bigl(1+(x_{12}(\mathfrak z)|\mathfrak w-\mathfrak z)\bigr)^2}
\right)                                                        \\
&\quad =
G(\partial_{\mathfrak z^{\scriptstyle\1}},\partial_{\mathfrak z^{\scriptstyle\2}})
F(\partial_{\mathfrak w^{\scriptstyle\1}},\partial_{\mathfrak w^{\scriptstyle\2}})
\sum_{m=0}^{\infty}
\bigl(-(x_{12}(\mathfrak z)|\mathfrak w-\mathfrak z)\bigr)^m
(m+1)x^{\ii}_{12}(\mathfrak z)P_{12}(\mathfrak z)              \\
&\quad =
G(\partial_{\mathfrak z^{\scriptstyle\1}},\partial_{\mathfrak z^{\scriptstyle\2}})
F(\partial_{\mathfrak w^{\scriptstyle\1}},\partial_{\mathfrak w^{\scriptstyle\2}})
\sum_{m=0}^{\infty}
(\partial_{z_1}|\mathfrak w-\mathfrak z)^m
\frac{-\partial_{z_1^{\scriptstyle\ii}}P_{12}(\mathfrak z)}
{m!(m+2)}                                                       \\
&\quad =
G(\partial_{\mathfrak z^{\scriptstyle\1}},\partial_{\mathfrak z^{\scriptstyle\2}})
F(\partial_{\mathfrak w^{\scriptstyle\1}},\partial_{\mathfrak w^{\scriptstyle\2}})
\int_0^1
e^{s(\partial_{z_1}|\mathfrak w-\mathfrak z)}
s(-\partial_{z_1^{\scriptstyle\ii}})P_{12}(\mathfrak z)\,ds                \\
&\quad =
G(\partial_{\mathfrak z^{\scriptstyle\1}},\partial_{\mathfrak z^{\scriptstyle\2}})
F(\partial_{\mathfrak w^{\scriptstyle\1}},\partial_{\mathfrak w^{\scriptstyle\2}})
(-\partial_{\mathfrak w^{\scriptstyle\ii}})
\int_0^1
P_{12}\bigl((1-s)\mathfrak z+s\mathfrak w\bigr)\,ds             \\
&\quad =
\int_0^1
G\bigl((1-s)\partial_{z_1^{\scriptstyle\1}},(1-s)\partial_{z_1^{\scriptstyle\2}}\bigr)
F\bigl(s\partial_{z_1^{\scriptstyle\1}},s\partial_{z_1^{\scriptstyle\2}}\bigr)
(-s\partial_{z_1^{\scriptstyle\ii}})
P_{12}\bigl((1-s)\mathfrak z+s\mathfrak w\bigr)\,ds .
\end{align*}
\end{proof}

The next lemma is an immediate consequence of Lemma~\ref{IntegralXP}. It
relates the graph weight to the weight obtained by inserting a coordinate
function into one specified edge factor.

\begin{lem}\label{IntegralR}
Let \(\Gamma\) be a graph with vertices \(\{o,1,\dots,n\}\), and suppose that
\(\vec e=(i\to j)\in \vec E(\Gamma)\). Let
\[
F(\lambda_1,\dots,\lambda_n;\mathfrak z_{\vec e})
\define
\mu_{\{o,1,\dots,n\}}
\left(W_{\Gamma}(\mathfrak z_{\vec e})\right).
\]
Then, for \(\sss\in\{\1,\2\}\),
\[
\mu_{\{o,1,\dots,n\}}
\left(W_{\Gamma}(\mathfrak z_{\vec e})\cdot x_{ij}^{\ii}\right)
=
-\partial_{\mathfrak z_{ij}^{\scriptstyle\ii}}
\int_0^1
F(\lambda_1,\dots,\lambda_n;\dots,r\mathfrak z_{ij},\dots)\,dr.
\]
Here \(W_{\Gamma}(\mathfrak z_{\vec e})\cdot x_{ij}^{\ii}\) means that the
edge factor \(P_{ij}(\mathfrak z_{ij})\) is replaced by
\(x_{ij}^{\ii}P_{ij}(\mathfrak z_{ij})\), while all other edge factors are
unchanged.
\end{lem}

\begin{figure}[htp]
  \centering

\tikzset{every picture/.style={line width=0.75pt}}

\begin{tikzpicture}[x=0.75pt,y=0.75pt,yscale=-1,xscale=1]

\draw    (318.95,172.29) -- (306.1,135.94) ;
\draw [shift={(306.1,135.94)}, rotate = 250.52] [color={rgb, 255:red, 0; green, 0; blue, 0 }][fill={rgb, 255:red, 0; green, 0; blue, 0 }][line width=0.75] (0, 0) circle [x radius= 3.35, y radius= 3.35] ;
\draw [shift={(318.95,172.29)}, rotate = 250.52] [color={rgb, 255:red, 0; green, 0; blue, 0 }][fill={rgb, 255:red, 0; green, 0; blue, 0 }][line width=0.75] (0, 0) circle [x radius= 3.35, y radius= 3.35] ;

\draw [dash pattern={on 4.5pt off 4.5pt}] (328.95,57.74) -- (254.67,63.33) ;
\draw [shift={(254.67,63.33)}, rotate = 175.7] [color={rgb, 255:red, 0; green, 0; blue, 0 }][fill={rgb, 255:red, 0; green, 0; blue, 0 }][line width=0.75] (0, 0) circle [x radius= 3.35, y radius= 3.35] ;
\draw [shift={(328.95,57.74)}, rotate = 175.7] [color={rgb, 255:red, 0; green, 0; blue, 0 }][fill={rgb, 255:red, 0; green, 0; blue, 0 }][line width=0.75] (0, 0) circle [x radius= 3.35, y radius= 3.35] ;

\draw    (274.67,147.14) -- (306.1,135.94) ;
\draw [shift={(306.1,135.94)}, rotate = 340.37] [color={rgb, 255:red, 0; green, 0; blue, 0 }][fill={rgb, 255:red, 0; green, 0; blue, 0 }][line width=0.75] (0, 0) circle [x radius= 3.35, y radius= 3.35] ;
\draw [shift={(274.67,147.14)}, rotate = 340.37] [color={rgb, 255:red, 0; green, 0; blue, 0 }][fill={rgb, 255:red, 0; green, 0; blue, 0 }][line width=0.75] (0, 0) circle [x radius= 3.35, y radius= 3.35] ;

\draw    (274.67,147.14) -- (318.95,172.29) ;
\draw [shift={(318.95,172.29)}, rotate = 29.59] [color={rgb, 255:red, 0; green, 0; blue, 0 }][fill={rgb, 255:red, 0; green, 0; blue, 0 }][line width=0.75] (0, 0) circle [x radius= 3.35, y radius= 3.35] ;
\draw [shift={(274.67,147.14)}, rotate = 29.59] [color={rgb, 255:red, 0; green, 0; blue, 0 }][fill={rgb, 255:red, 0; green, 0; blue, 0 }][line width=0.75] (0, 0) circle [x radius= 3.35, y radius= 3.35] ;

\draw    (254.67,63.33) -- (246.08,97.12) ;
\draw [shift={(254.67,63.33)}, rotate = 104.26] [color={rgb, 255:red, 0; green, 0; blue, 0 }][fill={rgb, 255:red, 0; green, 0; blue, 0 }][line width=0.75] (0, 0) circle [x radius= 3.35, y radius= 3.35] ;

\draw    (328.95,57.74) -- (340.38,78.07) ;
\draw [shift={(328.95,57.74)}, rotate = 60.66] [color={rgb, 255:red, 0; green, 0; blue, 0 }][fill={rgb, 255:red, 0; green, 0; blue, 0 }][line width=0.75] (0, 0) circle [x radius= 3.35, y radius= 3.35] ;

\draw    (306.1,135.94) -- (340.38,123.4) ;
\draw [shift={(340.38,123.4)}, rotate = 339.91] [color={rgb, 255:red, 0; green, 0; blue, 0 }][fill={rgb, 255:red, 0; green, 0; blue, 0 }][line width=0.75] (0, 0) circle [x radius= 3.35, y radius= 3.35] ;
\draw [shift={(306.1,135.94)}, rotate = 339.91] [color={rgb, 255:red, 0; green, 0; blue, 0 }][fill={rgb, 255:red, 0; green, 0; blue, 0 }][line width=0.75] (0, 0) circle [x radius= 3.35, y radius= 3.35] ;

\draw    (274.67,147.14) -- (256.1,130.38) ;
\draw [shift={(274.67,147.14)}, rotate = 222.07] [color={rgb, 255:red, 0; green, 0; blue, 0 }][fill={rgb, 255:red, 0; green, 0; blue, 0 }][line width=0.75] (0, 0) circle [x radius= 3.35, y radius= 3.35] ;

\draw    (306.1,135.94) -- (303.24,120.6) ;
\draw [shift={(306.1,135.94)}, rotate = 259.45] [color={rgb, 255:red, 0; green, 0; blue, 0 }][fill={rgb, 255:red, 0; green, 0; blue, 0 }][line width=0.75] (0, 0) circle [x radius= 3.35, y radius= 3.35] ;

\draw    (340.38,123.4) -- (342.08,110.21) ;
\draw [shift={(340.38,123.4)}, rotate = 277.36] [color={rgb, 255:red, 0; green, 0; blue, 0 }][fill={rgb, 255:red, 0; green, 0; blue, 0 }][line width=0.75] (0, 0) circle [x radius= 3.35, y radius= 3.35] ;

\draw    (318.95,172.29) -- (340.38,123.4) ;
\draw [shift={(340.38,123.4)}, rotate = 293.67] [color={rgb, 255:red, 0; green, 0; blue, 0 }][fill={rgb, 255:red, 0; green, 0; blue, 0 }][line width=0.75] (0, 0) circle [x radius= 3.35, y radius= 3.35] ;
\draw [shift={(318.95,172.29)}, rotate = 293.67] [color={rgb, 255:red, 0; green, 0; blue, 0 }][fill={rgb, 255:red, 0; green, 0; blue, 0 }][line width=0.75] (0, 0) circle [x radius= 3.35, y radius= 3.35] ;

\draw (315.1,63.11) node [anchor=north west][inner sep=0.75pt]    {$z_{j}{}$};
\draw (260.53,71.49) node [anchor=north west][inner sep=0.75pt]    {$z_{i}$};
\draw (286.48,97.71) node [anchor=north west][inner sep=0.75pt]    {$\cdots $};
\draw (272.78,43.54) node [anchor=north west][inner sep=0.75pt]  [font=\tiny,rotate=-359.06]  {$x_{ij}^{\ii} P_{ij}(\mathfrak{z}_{ij})$};

\end{tikzpicture}
  \caption{The edge modification in Lemma~\ref{IntegralR}: the distinguished edge factor is replaced by \(x_{ij}^{\ii}P_{ij}(\mathfrak z_{ij})\).}
  \label{fig:IntegralR}
\end{figure}

\begin{proof}
By Lemma~\ref{IntegralXP}, applied with \(F=G=1\), we have
\[
x_{ij}^{\ii}P_{ij}(\mathfrak z_{ij})
=
-\partial_{\mathfrak z_{ij}^{\scriptstyle\ii}}
\int_0^1 P_{ij}(r\mathfrak z_{ij})\,dr .
\]
Therefore
\begin{align*}
&\mu_{\{o,1,\dots,n\}}
\left(W_{\Gamma}(\mathfrak z_{\vec e})\cdot x_{ij}^{\ii}\right)     \\
&\quad =
\mu_{\{o,1,\dots,n\}}
\left(
-\partial_{\mathfrak z_{ij}^{\scriptstyle\ii}}
\int_0^1
W_{\Gamma}(\dots,r\mathfrak z_{ij},\dots)\,dr
\right)                                                            \\
&\quad =
-\partial_{\mathfrak z_{ij}^{\scriptstyle\ii}}
\int_0^1
\mu_{\{o,1,\dots,n\}}
\left(
W_{\Gamma}(\dots,r\mathfrak z_{ij},\dots)
\right)\,dr                                                        \\
&\quad =
-\partial_{\mathfrak z_{ij}^{\scriptstyle\ii}}
\int_0^1
F(\lambda_1,\dots,\lambda_n;\dots,r\mathfrak z_{ij},\dots)\,dr .
\end{align*}
\end{proof}

\begin{lem}\label{IntegralS}
Let \(\Gamma\) be a graph with vertices \(\{o,1,\dots,n\}\), and let
\(i,j\in\{1,\dots,n\}\) be vertices such that \((i,j)\in E(\Gamma)\).
Suppose that
\[
\mu_{\{o,1,\dots,n\}}(\alpha)
=
F_{\alpha}(\lambda_1,\dots,\lambda_n;\mathfrak z_{\vec e}).
\]
Then, for \(\ii\in\{\1,\2\}\),
\begin{align*}
&\mu_{\{\star,\bullet\}}\circ\mu_{\{o,1,\dots,n\}}
\left(\alpha\cdot x_{\star i}^{\ii}P_{\star j}\right)              \\
&\qquad =
\int_0^1
F_{\alpha}
(\lambda_1,\dots,
\lambda_i+(1-s_{\star})\lambda_{\star},
\dots,
\lambda_j+s_{\star}\lambda_{\star},
\dots,\lambda_n;\mathfrak z_{\vec e})
\cdot
s_{\star}\lambda_{\star}^{\ii}
\,ds_{\star},
\end{align*}
and
\begin{align*}
&\mu_{\{\star,\bullet\}}\circ\mu_{\{o,1,\dots,n\}}
\left(\alpha\cdot x_{\star j}^{\ii}P_{\star i}\right)              \\
&\qquad =
\int_0^1
F_{\alpha}
(\lambda_1,\dots,
\lambda_i+(1-s_{\star})\lambda_{\star},
\dots,
\lambda_j+s_{\star}\lambda_{\star},
\dots,\lambda_n;\mathfrak z_{\vec e})
\cdot
(1-s_{\star})\lambda_{\star}^{\ii}
\,ds_{\star}.
\end{align*}
Similarly, if \(\star\) is joined to \(o\) and \(i\), then
\begin{align*}
&\mu_{\{\star,\bullet\}}\circ\mu_{\{o,1,\dots,n\}}
\left(\alpha\cdot x_{\star i}^{\ii}P_{\star o}\right)              \\
&\qquad =
\int_0^1
F_{\alpha}
(\lambda_1,\dots,
\lambda_i+(1-s_{\star})\lambda_{\star},
\dots,\lambda_n;\mathfrak z_{\vec e})
\cdot
s_{\star}\lambda_{\star}^{\ii}
\,ds_{\star},
\end{align*}
and
\begin{align*}
&\mu_{\{\star,\bullet\}}\circ\mu_{\{o,1,\dots,n\}}
\left(\alpha\cdot x_{\star o}^{\ii}P_{\star i}\right)              \\
&\qquad =
\int_0^1
F_{\alpha}
(\lambda_1,\dots,
\lambda_i+(1-s_{\star})\lambda_{\star},
\dots,\lambda_n;\mathfrak z_{\vec e})
\cdot
(1-s_{\star})\lambda_{\star}^{\ii}
\,ds_{\star}.
\end{align*}
Figure~\ref{TwoMultOperation} illustrates the first identity.
\end{lem}

\begin{proof}
We prove the first identity; the other three are identical, with the roles of
the two endpoint factors interchanged. Write
\[
\mu_{\{o,1,\dots,n\}}(\alpha)
=
F_{\alpha}(\lambda)
=
\sum G(\lambda_i)H(\lambda_j),
\]
where the variables \(\lambda_i\) and \(\lambda_j\) are formally separated.
Then
\begin{align*}
&\mu_{\{\star,\bullet\}}\circ\mu_{\{o,1,\dots,n\}}
\left(\alpha\cdot x_{\star i}^{\ii}P_{\star j}\right)              \\
&\quad =
\mu_{\{\star,\bullet\}}
\left(
\sum
G(\lambda_i-\partial_{z_{\star}})
x_{\star\bullet}^{\ii}
\cdot
H(\lambda_j-\partial_{z_{\star}})
P_{\star\bullet}
\right)                                                            \\
&\quad =
\mu_{\{\star,\bullet\}}
\left(
\int_0^1
\sum
G(\lambda_i-(1-s)\partial_{z_{\star}})
H(\lambda_j-s\partial_{z_{\star}})
(-s\partial_{z_{\star}^{\scriptstyle\ii}})
P_{\star\bullet}
\,ds
\right)                                                            \\
&\quad =
\int_0^1
\sum
G(\lambda_i+(1-s)\lambda_{\star})
H(\lambda_j+s\lambda_{\star})
s\lambda_{\star}^{\ii}
\,ds                                                               \\
&\quad =
\int_0^1
F_{\alpha}
(\lambda_1,\dots,
\lambda_i+(1-s_{\star})\lambda_{\star},
\dots,
\lambda_j+s_{\star}\lambda_{\star},
\dots,\lambda_n;\mathfrak z_{\vec e})
\cdot
s_{\star}\lambda_{\star}^{\ii}
\,ds_{\star}.
\end{align*}
The second equality uses Lemma~\ref{IntegralXP}.
\end{proof}

\begin{figure}[htp]
  \centering

\tikzset{every picture/.style={line width=0.75pt}}


  \caption{The composition
  \(\mu_{\{\star,\bullet\}}\circ \mu_{\{o,1,\dots,n\}}
  \left(\alpha\cdot x_{\star i}^{\ii}P_{\star j}\right)\).}
  \label{TwoMultOperation}
\end{figure}

Now we are ready to prove Theorem~\ref{MainTheoremRecursiveRS}.
\begin{proof}
\begin{figure}[htp]
  \centering

\tikzset{every picture/.style={line width=0.75pt}} 

%
  \caption{The Arnold relation expresses
  \(W_{\Gamma}(\tilde{\mathfrak z}_{\vec e})\cdot P_{\star i}P_{\star j}\)
  as an exact term. Dashed edges indicate coordinate factors.}\label{dExactGraph}
\end{figure}
\begin{figure}[htp]
  \centering

\tikzset{every picture/.style={line width=0.75pt}} 

%
  \caption{The two vanishing terms in the \(L_{\infty}\)-relation. After the binary operation collapses \(\star\) with \(i\) or \(j\), the coordinate factor contains \(x_{ij}\wedge x_{ij}\).}\label{Cancellations}
\end{figure}
We prove the case in which \(\star\) is joined to
\(i,j\in V(\Gamma)\setminus\{o\}\), with
\[
\vec{e^\triangleright}=ij.
\]
The case in which \(\star\) is joined to \(i\) and \(o\) is similar. First observe that
\begin{align*}
&\mu_{\{o,1,\dots,n,\star\}}
\left(W_{\Gamma^{\star}}(\mathfrak z_{\vec e})\right)        \\
&\qquad =
\mu_{\{o,1,\dots,n,\star\}}
\left(
W_{\Gamma}(\mathfrak z_{\vec e})
P_{\star i}(\mathfrak z_{\star i})
P_{\star j}(\mathfrak z_{\star j})
\right)                                                       \\
&\qquad =
\mu_{\{o,1,\dots,n,\star\}}
\left(
W_{\Gamma}(\tilde{\mathfrak z}_{\vec e})P_{\star i}P_{\star j}
\right)
\cdot
 e^{(\lambda_i|\mathfrak z_{\star i})+(\lambda_j|\mathfrak z_{\star j})}.
\end{align*}
Here the shifted formal variables are defined by
\[
\begin{cases}
\tilde{\mathfrak z}_{\vec e}
=
\mathfrak z_{\vec e}+\varepsilon_i(\vec e)\mathfrak z_{\star i},
& \mbox{if } i \mbox{ is an endpoint of } \vec e,
\vec e\neq\vec{e^\triangleright},\\[1mm]
\tilde{\mathfrak z}_{\vec e}
=
\mathfrak z_{\vec e}+\varepsilon_j(\vec e)\mathfrak z_{\star j},
& \mbox{if } j \mbox{ is an endpoint of } \vec e,
\vec e\neq\vec{e^\triangleright},\\[1mm]
\tilde{\mathfrak z}_{ij}
=
\mathfrak z_{ij}+\mathfrak z_{\star i}-\mathfrak z_{\star j},
& \mbox{if } \vec e=\vec{e^\triangleright},\\[1mm]
\tilde{\mathfrak z}_{\vec e}
=
\mathfrak z_{\vec e},
& \mbox{otherwise}.
\end{cases}
\]
With this convention,
\[
\mathfrak z_{ij}-(1-r_{\star})\tilde{\mathfrak z}_{ij}
=
\mathfrak z_{ij}
+
(1-r_{\star})
(\mathfrak z_{\star j}-\mathfrak z_{\star i}-\mathfrak z_{ij})
=
\mathfrak z_{ij}^{\triangleright}.
\]

Using Corollary~\ref{RecursiveLemma}, we have
\[
P_{\star i}P_{\star j}P_{ij}(\tilde{\mathfrak z}_{ij})
=
\mathbf d\left(
P_{ij}(\tilde{\mathfrak z}_{ij})P_{\star j}\cdot
x_{\star i}\wedge x_{ij}
+
P_{ij}(\tilde{\mathfrak z}_{ij})P_{\star i}\cdot
x_{\star j}\wedge x_{ij}
\right).
\]
Substituting this into
\[
\mu_{\{o,1,\dots,n,\star\}}
\left(
W_{\Gamma}(\tilde{\mathfrak z}_{\vec e})P_{\star i}P_{\star j}
\right),
\]
and using the \(L_{\infty}\)-relation, we obtain, as illustrated in
Fig.~\ref{dExactGraph} and Fig.~\ref{Cancellations},
\begin{align*}
&\mu_{\{o,1,\dots,n,\star\}}
\left(
W_{\Gamma}(\tilde{\mathfrak z}_{\vec e})P_{\star i}P_{\star j}
\right)                                                        \\
&\quad =
\mu_{\{\star,\bullet\}}\circ
\mu_{\{o,1,\dots,n\}}
\left(
W_{\Gamma}(\tilde{\mathfrak z}_{\vec e})
\cdot
(x_{ij}\wedge x_{\star j})P_{\star i}
\right)                                                        \\
&\qquad
+
\mu_{\{\star,\bullet\}}\circ
\mu_{\{o,1,\dots,n\}}
\left(
W_{\Gamma}(\tilde{\mathfrak z}_{\vec e})
\cdot
(x_{ij}\wedge x_{\star i})P_{\star j}
\right).
\end{align*}
The two remaining terms in the \(L_{\infty}\)-relation vanish: after the
binary operation collapses \(\star\) with \(i\) or \(j\), the remaining
coordinate factor contains \(x_{ij}\wedge x_{ij}\), hence is zero.

We prove the recursive formula together with the auxiliary identity
\[
  \sum_{\vec e\in\vec E(\Gamma)}
  \varepsilon_v(\vec e)\,
  \partial_{\mathfrak z_{\vec e}}
  \mathbf W^{r,s}_{\Gamma,\mathbf H_{\Gamma}(o,e)}
  =
  -\lambda_v,
  \qquad
  v\in V(\Gamma)\setminus\{o\}.
\]

Assume this identity holds for \(\Gamma\). Then
\begin{align*}
&
\mathbf W^{r,s}_{\Gamma,\mathbf H_{\Gamma}(o,e)}
(\lambda;\dots,r_{\star}\tilde{\mathfrak z}_{ij},\dots)  =
\mathbf W^{r,s}_{\Gamma,\mathbf H_{\Gamma}(o,e)}
(\lambda;\tilde{\mathfrak z}_{\vec e})
-
(1-r_{\star})
\left(
\partial_{\mathfrak z_{ij}}
\mathbf W^{r,s}_{\Gamma,\mathbf H_{\Gamma}(o,e)}
(\lambda;\mathfrak z_{\vec e})
\mid
\tilde{\mathfrak z}_{ij}
\right)                                                        \\
&\quad =
\mathbf W^{r,s}_{\Gamma,\mathbf H_{\Gamma}(o,e)}
+
\left(
\sum_{\substack{\vec e\in\vec E(\Gamma)\\ i\in\vec e}}\varepsilon_i(\vec e)
\partial_{\mathfrak z_{\vec e}}
\mathbf W^{r,s}_{\Gamma,\mathbf H_{\Gamma}(o,e)}
\mid
\mathfrak z_{\star i}
\right)                                                        \\
&\qquad
+
\left(
\sum_{\substack{\vec e\in\vec E(\Gamma)\\ j\in\vec e}}\varepsilon_j(\vec e)
\partial_{\mathfrak z_{\vec e}}
\mathbf W^{r,s}_{\Gamma,\mathbf H_{\Gamma}(o,e)}
\mid
\mathfrak z_{\star j}
\right)
-
(1-r_{\star})
\left(
\partial_{\mathfrak z_{ij}}
\mathbf W^{r,s}_{\Gamma,\mathbf H_{\Gamma}(o,e)}
\mid
\tilde{\mathfrak z}_{ij}
\right)                                                        \\
&\quad =
\mathbf W^{r,s}_{\Gamma,\mathbf H_{\Gamma}(o,e)}
-
(\lambda_i|\mathfrak z_{\star i})
-
(\lambda_j|\mathfrak z_{\star j})+
(1-r_{\star})
\left(
\partial_{\mathfrak z_{ij}}
\mathbf W^{r,s}_{\Gamma,\mathbf H_{\Gamma}(o,e)}
\mid
\mathfrak z_{\star j}-\mathfrak z_{\star i}-\mathfrak z_{ij}
\right)                                                        \\
&\quad =
\mathbf W^{r,s}_{\Gamma,\mathbf H_{\Gamma}(o,e)}
(\lambda;\mathfrak z_{\vec e}^{\triangleright})
-
(\lambda_i|\mathfrak z_{\star i})
-
(\lambda_j|\mathfrak z_{\star j}).
\end{align*}

For simplicity, write
\[
\int_{\Box_{\mathbf H_{\Gamma}(o,e)}}
\define
\int_{([0,1]\times[0,1])^{|\mathbf H_{\Gamma}(o,e)|}}
\bigwedge_{k=1}^{|\mathbf H_{\Gamma}(o,e)|}
(dr_k\wedge ds_k).
\]
By Lemma~\ref{IntegralR}, for \(\ii\in\{\1,\2\}\), we have
\begin{align*}
&
\mu_{\{o,1,\dots,n\}}
\left(
W_{\Gamma}(\tilde{\mathfrak z}_{\vec e})
\cdot x_{ij}^{\ii}
\right)                                                        \\
&\quad =
-\partial_{\tilde{\mathfrak z}_{ij}^{\ii}}
\int_0^1
\int_{\Box_{\mathbf H_{\Gamma}(o,e)}}
e^{
\mathbf W^{r,s}_{\Gamma,\mathbf H_{\Gamma}(o,e)}
(\lambda;\dots,r_{\star}\tilde{\mathfrak z}_{ij},\dots)
}
\cdot
\mathbf G^{r,s}_{\Gamma,\mathbf H_{\Gamma}(o,e)}
\,dr_{\star}                                                    \\
&\quad =
-\int_0^1
\int_{\Box_{\mathbf H_{\Gamma}(o,e)}}
e^{
\mathbf W^{r,s}_{\Gamma,\mathbf H_{\Gamma}(o,e)}
(\lambda;\mathfrak z_{\vec e}^{\triangleright})
-
(\lambda_i|\mathfrak z_{\star i})
-
(\lambda_j|\mathfrak z_{\star j})
}
\cdot
\partial_{\mathfrak z_{ij}^{\scriptstyle\ii}}
\mathbf W^{r,s}_{\Gamma,\mathbf H_{\Gamma}(o,e)}
\cdot
\mathbf G^{r,s}_{\Gamma,\mathbf H_{\Gamma}(o,e)}
\cdot r_{\star}\,dr_{\star}.
\end{align*}

Applying Lemma~\ref{IntegralS} gives
\begin{align*}
&
\mu_{\{\star,\bullet\}}\circ
\mu_{\{o,1,\dots,n\}}
\left(
W_{\Gamma}(\tilde{\mathfrak z}_{\vec e})
\cdot
(x_{ij}\wedge x_{\star i})P_{\star j}
\right)                                                        \\
&\quad =
\int_{([0,1]\times[0,1])^{|\mathbf H_{\Gamma}(o,e)|+1}}
e^{
\mathbf W^{r,s}_{\Gamma,\mathbf H_{\Gamma}(o,e)}
(\lambda^\triangleright;\mathfrak z_{\vec e}^\triangleright)
-
(\lambda_i|\mathfrak z_{\star i})
-
(\lambda_j|\mathfrak z_{\star j})
-
\left(
\lambda_{\star}
\mid
(1-s_{\star})\mathfrak z_{\star i}
+
s_{\star}\mathfrak z_{\star j}
\right)
}                                                               \\
&\qquad\cdot
\left(
\partial_{\mathfrak z_{ij}}
\mathbf W^{r,s}_{\Gamma,\mathbf H_{\Gamma}(o,e)}
\wedge\lambda_{\star}
\right)
\cdot r_{\star}s_{\star}
\cdot
\mathbf G^{r,s}_{\Gamma,\mathbf H_{\Gamma}(o,e)}
(\lambda^\triangleright) \cdot
(dr_{\star}\wedge ds_{\star})
\bigwedge_{k=1}^{|\mathbf H_{\Gamma}(o,e)|}
(dr_k\wedge ds_k).
\end{align*}
Similarly,
\begin{align*}
&
\mu_{\{\star,\bullet\}}\circ
\mu_{\{o,1,\dots,n\}}
\left(
W_{\Gamma}(\tilde{\mathfrak z}_{\vec e})
\cdot
(x_{ij}\wedge x_{\star j})P_{\star i}
\right)                                                        \\
&\quad =
\int_{([0,1]\times[0,1])^{|\mathbf H_{\Gamma}(o,e)|+1}}
e^{
\mathbf W^{r,s}_{\Gamma,\mathbf H_{\Gamma}(o,e)}
(\lambda^\triangleright;\mathfrak z_{\vec e}^\triangleright)
-
(\lambda_i|\mathfrak z_{\star i})
-
(\lambda_j|\mathfrak z_{\star j})
-
\left(
\lambda_{\star}
\mid
(1-s_{\star})\mathfrak z_{\star i}
+
s_{\star}\mathfrak z_{\star j}
\right)
}                                                               \\
&\qquad\cdot
\left(
\partial_{\mathfrak z_{ij}}
\mathbf W^{r,s}_{\Gamma,\mathbf H_{\Gamma}(o,e)}
\wedge\lambda_{\star}
\right)
\cdot r_{\star}(1-s_{\star})
\cdot
\mathbf G^{r,s}_{\Gamma,\mathbf H_{\Gamma}(o,e)}
(\lambda^\triangleright)                                       \\
&\qquad\cdot
(dr_{\star}\wedge ds_{\star})
\bigwedge_{k=1}^{|\mathbf H_{\Gamma}(o,e)|}
(dr_k\wedge ds_k).
\end{align*}
Here we use
\[
\partial_{\mathfrak z_{ij}}
\mathbf W^{r,s}_{\Gamma,\mathbf H_{\Gamma}(o,e)}
(\lambda^\triangleright)
\wedge\lambda_{\star}
=
\partial_{\mathfrak z_{ij}}
\mathbf W^{r,s}_{\Gamma,\mathbf H_{\Gamma}(o,e)}
(\lambda)
\wedge\lambda_{\star},
\]
because the difference between the two sides before wedging is proportional to
\(\lambda_{\star}\), and \(\lambda_{\star}\wedge\lambda_{\star}=0\).

Adding the two contributions, using
\(s_{\star}+(1-s_{\star})=1\), and multiplying by
\(e^{(\lambda_i|\mathfrak z_{\star i})+(\lambda_j|\mathfrak z_{\star j})}\), we obtain
\begin{align*}
&
\mu_{\{o,1,\dots,n,\star\}}
\left(W_{\Gamma^{\star}}(\mathfrak z_{\vec e})\right)            \\
&\quad =
\int_{([0,1]\times[0,1])^{|\mathbf H_{\Gamma}(o,e)|+1}}
e^{
\mathbf W^{r,s}_{\Gamma,\mathbf H_{\Gamma}(o,e)}
(\lambda^\triangleright;\mathfrak z_{\vec e}^\triangleright)
-
\left(
\lambda_{\star}
\mid
(1-s_{\star})\mathfrak z_{\star i}
+
s_{\star}\mathfrak z_{\star j}
\right)
}                                                               \\
&\qquad\cdot
\left(
\partial_{\mathfrak z_{ij}}
\mathbf W^{r,s}_{\Gamma,\mathbf H_{\Gamma}(o,e)}
\wedge\lambda_{\star}
\right)
\cdot r_{\star}
\cdot
\mathbf G^{r,s}_{\Gamma,\mathbf H_{\Gamma}(o,e)}
(\lambda^\triangleright)                                       \\
&\qquad\cdot
(dr_{\star}\wedge ds_{\star})
\bigwedge_{k=1}^{|\mathbf H_{\Gamma}(o,e)|}
(dr_k\wedge ds_k).
\end{align*}
This is precisely the recursive formula for
\(\mathbf W_{\Gamma^{\star},\mathbf H_{\Gamma^{\star}}(o,e)}^{r,s}\) and
\(\mathbf G_{\Gamma^{\star},\mathbf H_{\Gamma^{\star}}(o,e)}^{r,s}\) in the
first case of the theorem.

It remains to check that the auxiliary identity is preserved. For
\(\Gamma^{\star}\), we have
\[
\mathbf W^{r,s}_{\Gamma^{\star},\mathbf H_{\Gamma^{\star}}(o,e)}
=
\mathbf W^{r,s}_{\Gamma,\mathbf H_{\Gamma}(o,e)}
(\lambda^\triangleright;\mathfrak z_{\vec e}^\triangleright)
-
\left(
\lambda_{\star}
\mid
(1-s_{\star})\mathfrak z_{\star i}
+
s_{\star}\mathfrak z_{\star j}
\right).
\]
Equivalently,
\begin{align*}
\mathbf W^{r,s}_{\Gamma^{\star},\mathbf H_{\Gamma^{\star}}(o,e)}
&=
\left(
1+
(\lambda_{\star}|
(1-s_{\star})\partial_{\lambda_i}
+
s_{\star}\partial_{\lambda_j})
\right)                                                        \\
&\quad\cdot
\left(
1+
(1-r_{\star})
(\mathfrak z_{\star j}-\mathfrak z_{\star i}-\mathfrak z_{ij}
|
\partial_{\mathfrak z_{ij}})
\right)
\mathbf W^{r,s}_{\Gamma,\mathbf H_{\Gamma}(o,e)}                \\
&\quad
-
\left(
\lambda_{\star}
\mid
(1-s_{\star})\mathfrak z_{\star i}
+
s_{\star}\mathfrak z_{\star j}
\right).
\end{align*}
Therefore
\begin{align*}
&
\sum_{\substack{\vec e\in\vec E(\Gamma^{\star})\\ i\in\vec e}}
 \varepsilon_i(\vec e)\partial_{\mathfrak z_{\vec e}}
\mathbf W^{r,s}_{\Gamma^{\star},\mathbf H_{\Gamma^{\star}}(o,e)}
\\
&\quad =
\left(
1+
(\lambda_{\star}|
(1-s_{\star})\partial_{\lambda_i}
+
s_{\star}\partial_{\lambda_j})
\right)
\left(
1+
(1-r_{\star})
(\mathfrak z_{\star j}-\mathfrak z_{\star i}-\mathfrak z_{ij}
|
\partial_{\mathfrak z_{ij}})
\right)                                                        \\
&\qquad\cdot
\sum_{\substack{\vec e\in\vec E(\Gamma)\\ i\in\vec e}}
 \varepsilon_i(\vec e)\partial_{\mathfrak z_{\vec e}}
\mathbf W^{r,s}_{\Gamma,\mathbf H_{\Gamma}(o,e)}
+
(1-s_{\star})\lambda_{\star}                                   \\
&\quad =
\left(
1+
(\lambda_{\star}|
(1-s_{\star})\partial_{\lambda_i}
+
s_{\star}\partial_{\lambda_j})
\right)(-\lambda_i)
+
(1-s_{\star})\lambda_{\star}                                   \\
&\quad =
-\lambda_i.
\end{align*}
The same computation gives
\[
\sum_{\substack{\vec e\in\vec E(\Gamma^{\star})\\ j\in\vec e}}
 \varepsilon_j(\vec e)\partial_{\mathfrak z_{\vec e}}
\mathbf W^{r,s}_{\Gamma^{\star},\mathbf H_{\Gamma^{\star}}(o,e)}
=
-\lambda_j.
\]
For vertices \(v\in V(\Gamma)\setminus\{o,i,j\}\), the identity is unchanged
from the induction hypothesis. Finally,
\[
\sum_{\substack{\vec e\in\vec E(\Gamma^{\star})\\ \star\in\vec e}}
 \varepsilon_\star(\vec e)\partial_{\mathfrak z_{\vec e}}
\mathbf W^{r,s}_{\Gamma^{\star},\mathbf H_{\Gamma^{\star}}(o,e)}
=
(\partial_{\mathfrak z_{\star i}}+\partial_{\mathfrak z_{\star j}})
\mathbf W^{r,s}_{\Gamma^{\star},\mathbf H_{\Gamma^{\star}}(o,e)}
=
-\lambda_{\star}.
\]
Thus the induction step is complete, and the theorem follows.
\end{proof}

\newpage
\appendix
\section{Graph theory background}\label{graph theory}

In this appendix, we introduce some concepts and facts from graph theory used in section \ref{s:feynman}.
\begin{defn}
    A directed graph $\vec{\Gamma}$ consists of the following data:
    \begin{enumerate}
        \item An ordered set $\vec{\Gamma}_{0}$ of vertices. We use $|\Gamma_{0}|$ to denote the number of vertices.
        \item An ordered set $\vec{\Gamma}_{1}$ of directed edges. We use $|\Gamma_{1}|$ to denote the number of directed edges.
        \item Two maps
        $$
        t,h:\Gamma_{1}\rightarrow\Gamma_{0},
        $$
        which are assignments of tail and head to each directed edge. We require that
        $$
        t(e)\neq h(e)
        $$
        for any $e\in\vec{\Gamma}_{1}$, i.e., the graph $\vec{\Gamma}$ has no self-loops.
    \end{enumerate}
    Furthermore, we say $\vec{\Gamma}$ is decorated (Feynman) if we have a special edge $e_{l}\in \vec{\Gamma}_{1}$ and a map
    $$
    \mmm:\Gamma_{1}\rightarrow\{1,2,\dots,d\}.
    $$
\end{defn}

\begin{defn}
    Let $\vec{\Gamma}$ be a directed graph. The incidence matrix $\rho=(\rho_{ei})_{e\in\vec{\Gamma}_{1},i\in\vec{\Gamma}_{0}}$ is given by
    $$
    \rho_{ei}=
    \begin{cases}
        1, &\text{if }t(e)=i,\\
        -1, &\text{if }h(e)=i,\\
        0, &\text{otherwise.}
    \end{cases}
    $$
\end{defn}

\begin{defn}
  Given a connected directed graph $\vec{\Gamma}$, a tree $T \subset \vec{\Gamma}$ is
  said to be a spanning tree if every vertex of $\vec{\Gamma}$ lies in $T$. We
  denote the set of all spanning trees by $\mathrm{Tree} (\vec{\Gamma})$.
\end{defn}

\begin{defn}
  Given a connected directed graph $\vec{\Gamma}$ and two
  disjoint subsets of vertices $V_1, \text{ } V_2 \subset \vec{\Gamma}_0$, we
  define $\mathrm{Cut} (\vec{\Gamma} ; V_1, V_2)$ to be the set of subsets $C
  \subset \vec{\Gamma}_1$ satisfying the following properties:
  \begin{enumerate}
    \item  Removing the edges in $C$ from $\vec{\Gamma}$ divides $\vec{\Gamma}$ into
    exactly two connected trees, which we denote by $\vec{\Gamma}' (C), \text{ }
    \vec{\Gamma}'' (C) $, such that $V_1 \subset \vec{\Gamma}_0' (C), \text{ } V_2
    \subset \vec{\Gamma}_0'' (C)$.

    \item $C$ doesn't contain any proper subset satisfying property 1.
  \end{enumerate}
\end{defn}

\begin{defn}
  Given a connected directed graph $\vec{\Gamma}$, and a function
  maps each $e \in \vec{\Gamma}_1$ to $t_e \in (0, + \infty)$, we define the
  {{weighted Laplacian}} of $\vec{\Gamma}$ by the following formula:
  \[ M_{\vec{\Gamma}} (t)_{ij} = \sum_{e \in \vec{\Gamma}_1} \rho^e_i \frac{1}{t_e}
     \rho_j^e, \text{\quad} 1 \leqslant i, j \leqslant | \Gamma_0 | - 1. \]
\end{defn}

The following facts will be used to show the finiteness of the Feynman graph
integrals.

\begin{thm}[Kirchhoff]
  Given a connected graph $\vec{\Gamma}$, the determinant of
  weighted Laplacian is given by the following formula:
  \[ \det M_{\vec{\Gamma}} (t) = \frac{\underset{T \in \mathrm{Tree} (\vec{\Gamma})}{\sum}
     \underset{e \notin T}{\prod} t_e}{\underset{e \in \vec{\Gamma}_1}{\prod} t_e} \]
\end{thm}

\begin{cor}\label{det of laplacian}
  The inverse of $M_{\vec{\Gamma}} (t)$ is given by the following
  formula:\label{Minverse}
  \[ M_{\vec{\Gamma}} (t)^{- 1}_{i j} = \frac{1}{\underset{T \in \mathrm{Tree}
     (\vec{\Gamma})}{\sum} \underset{e \notin T}{\prod} t_e} \cdot \left( \sum_{C \in
     \mathrm{Cut} (\vec{\Gamma} ; \{ i, j \}, \{ | \Gamma_0 | \})} \prod_{e \in C} t_e
     \right) \]
\end{cor}
\begin{defn}
    Let $\vec{\Gamma}$ be a connected directed graph. For any $e\in\vec{\Gamma}_{1}$, $t_e\in(0,+\infty)$. The graphical Green’s function $d^{-1}_{\vec{\Gamma}}(t)=(d^{-1}_{\vec{\Gamma}}(t)_{ei})_{e\in\vec{\Gamma}_{1},i\in\vec{\Gamma}_{0}-\{| \Gamma_0 |\}}$ is given by
    $$
    d^{-1}_{\vec{\Gamma}}(t)_{ei}=\sum_{j=1}^{| \Gamma_0 |-1}\frac{1}{t_{e}}\rho_{ej}M^{-1}_{\vec{\Gamma}}(t)_{ji}.
    $$
\end{defn}

\begin{prop}[see appendix B of \cite{Li:2011mi}]
  \label{boundness}We have the following equality:
  \begin{eqnarray*}
    &  & d^{-1}_{\vec{\Gamma}}(t)_{ei}\\
    & = & \frac{1}{\underset{T \in \mathrm{Tree} (\vec{\Gamma})}{\sum} \underset{e
    \notin T}{\prod} t_e} \left( \sum_{C \in \mathrm{Cut} (\vec{\Gamma} ; \{ i, t (e)
    \}, \{ | \Gamma_0 |, h (e) \})} \frac{\prod_{e' \in C} t_{e'}}{t_e}
    \right.\\
    & - & \left. \sum_{C \in \mathrm{Cut} (\vec{\Gamma} ; \{ i, h (e) \}, \{ |
    \Gamma_0 |, t (e) \})} \frac{\prod_{e' \in C} t_{e'}}{t_e} \right)
  \end{eqnarray*}
  In particular, every term of the numerator also appears in the denominator,
  so we have
  \[ \left| d^{-1}_{\vec{\Gamma}}(t)_{ei} \right| \leqslant 2. \]
\end{prop}

\begin{rem}
  If we view graphs as discrete spaces, the incidence matrix can be viewed as
  the de Rham differential. Then $d^{-1}_{\vec{\Gamma}}(t)_{ei}$ can be viewed as
  the Green's function of the de Rham differential on a graph. This might explain
 the importance of $d^{-1}_{\vec{\Gamma}}(t)_{ei}$.
\end{rem}

%
%

\section{Compactification of Schwinger spaces}\label{Schwinger spaces}

In this appendix, we review the construction of a compactification of
Schwinger space in {\cite{wang2024feynman,Wang:2024tjf}}.

\begin{defn}
  Given a directed graph $\vec{\Gamma}$, and $L > 0$, the {{Schwinger
  space}} is defined by $(0, L]^{\vec{\Gamma}_1 }$. The orientation is given by
  the following formula:
  \[ \int_{(0, L]^{\vec{\Gamma}_1}} \prod_{e \in \vec{\Gamma}_1} d t_e = L^{| \Gamma_1
     |} . \]
\end{defn}

Assume $\vec{\Gamma}$ is a connected directed graph, $L > 0$ is a positive real
number. Let $S \subset \vec{\Gamma}_1$ be a subset of $\vec{\Gamma}_1$, we define the
following submanifold with corners of Schwinger space:
\[ \Delta_S = \left\{ (t_1, t_2, \ldots, t_{| \Gamma_1 |}) \in [0, L]^{\vec{\Gamma}_1} |  t_e = 0 \quad \text{if } e \in S \right\} . \]
The compactification of Schwinger space is obtained by iterated real blow ups
of $[0, L]^{\vec{\Gamma}_1}$ along $\Delta_S$ for all $S \subset \vec{\Gamma}_1$ in
a certain order (see {\cite{Ammann2019ACO,epub47792}}). To avoid getting into
technical details of real blow ups of manifolds with corners, we will use
another simpler definition. Instead, we present a typical example of real blow
up, which will be helpful to understand our construction:
\begin{exa}
  Let $S = \{ 1, 2, \ldots, k \} \subset \vec{\Gamma}_1$, the real blow up of $[0,
  + \infty)^{\vec{\Gamma}_1}$ along $\Delta_S$ is the following manifold (with
  corners):
  \[ [[0, + \infty)^{\vec{\Gamma}_1 } : \Delta_S] \define \left\{ (\rho, \xi_1,
     \ldots, \xi_k, t_{k + 1}, \ldots t_{| \Gamma_1 |}) \in [0, + \infty)^{|
     \Gamma_1 | + 1} \left| \sum_{i = 1}^k \xi_i^2 = 1 \right. \right\} . \]
  We have a natural map from $[[0, + \infty)^{\vec{\Gamma}_1 } : \Delta_S]$ to
  $[0, + \infty)^{\vec{\Gamma}_1 }$:
  \[ (\rho, \xi_1, \ldots, \xi_k, t_{k + 1}, \ldots t_{| \Gamma_1 |})
     \rightarrow (t_1 = \rho \xi_1, \ldots, t_k = \rho \xi_k, t_{k + 1},
     \ldots, t_{| \Gamma_1 |}) . \]
  We also have a natural inclusion map from $(0, + \infty)^{\vec{\Gamma}_1 }$ to
  $[[0, + \infty)^{\vec{\Gamma}_1 } : \Delta_S]$:
  \begin{eqnarray*}
    &  & (t_1, \ldots, t_{| \Gamma_1 |})\\
    & \rightarrow & \left( \rho = \sqrt{\sum_{i = 1}^k t_i^2}, \xi_1 =
    \frac{t_1}{\sqrt{\sum_{i = 1}^k t_i^2}}, \ldots, \xi_k =
    \frac{t_k}{\sqrt{\sum_{i = 1}^k t_i^2}}, t_{k + 1}, \ldots, t_{| \Gamma_1
    |} \right) .
  \end{eqnarray*}
  For us, the most important property is that $\frac{t_i}{\sqrt{\sum_{i = 1}^k
  t_i^2}}$ can be extended to a smooth function $\xi_i$ on $[[0, + \infty)^{\vec{\Gamma}_1 } : \Delta_S]$.
\end{exa}

Let's consider the following natural inclusion map:
\[ i : (0, + \infty)^{\vec{\Gamma}_1 } \rightarrow \prod_{S \subset \vec{\Gamma}_1}
   [[0, + \infty)^{\vec{\Gamma}_1 } : \Delta_S] \]
\begin{defn}
  We call the closure of the image of $(0, L]^{\vec{\Gamma}_1 }$ under $i$ the
  {{compactified Schwinger space}} of $\vec{\Gamma}$. We denote it by
  $\widetilde{[0, L]}^{\vec{\Gamma}_1 }$.
\end{defn}

\begin{rem}
  Sometimes, by abuse of terminology, we will also call the closure of the image
  of $(0, + \infty)^{\vec{\Gamma}_1 }$ under $i$ the compactified Schwinger
  space, although it is not compact. We will use $\widetilde{[0, + \infty)}^{\vec{\Gamma}_1}$ to denote it.
\end{rem}

\begin{prop}[\cite{Ammann2019ACO}]
  $\widetilde{[0, L]}^{\vec{\Gamma}_1}$ is a compact manifold with corners.
\end{prop}

The matrices $M_{\vec{\Gamma}} (t)^{- 1}_{i j}$ and $d^{-1}_{\vec{\Gamma}}(t)_{ei}$
defined in appendix \ref{graph theory} admit extensions to matrix-valued smooth functions on
$\widetilde{[0, + \infty)^{| \Gamma_1 |}}$:

\begin{lem}[\cite{wang2024feynman}]
  \label{extended functions}Given a connected graph $\vec{\Gamma}$, the following functions can be extended to smooth functions on
  $\widetilde{[0, + \infty)}^{ \vec{\Gamma}_1 }$:
  \begin{enumerate}
    \item $M_{\vec{\Gamma}} (t)^{- 1}_{i j}$ \text{for} $1 \leqslant i, j \leqslant |
    \Gamma_0 | - 1$.

    \item $d^{-1}_{\vec{\Gamma}}(t)_{ei}$ \text{for} $e \in \Gamma_1, \text{ } 1
    \leqslant i \leqslant | \Gamma_0 | - 1$.
  \end{enumerate}
\end{lem}

Let $\bar{S}^{+}((0,+\infty)^{\vec{\Gamma}_{1}})$ be the closure of
$$
    S^{+}((0,+\infty)^{\vec{\Gamma}_{1}})=\{(t_{e_{1}},\dots,t_{e_{|\Gamma_{1}|}})\in(0,+\infty)^{\vec{\Gamma}_{1}}|\sum_{i=1}^{|\Gamma_{1}|}t_{i}^{2}=1\}
$$in $\widetilde{[0,+\infty)}^{\vec{\Gamma}_{1}}$, the following proposition gives a description of the boundary of $\bar{S}^{+}((0,+\infty)^{\vec{\Gamma}_{1}})$:

\begin{prop}
  \label{boundary description}Given a connected graph $\vec{\Gamma}$, the boundary of
  $\bar{S}^{+}((0,+\infty)^{\vec{\Gamma}_{1}})$ has the following decomposition:
  \[ \partial \bar{S}^{+}((0,+\infty)^{\vec{\Gamma}_{1}}) = \bigcup_{\vec{\Gamma}'
       \subset \vec{\Gamma}} \mathrm{sgn}(\sigma_{\vec{\Gamma}'_{1}\subset \vec{\Gamma}_{1}})
       \bar{S}^{+}((0,+\infty)^{\vec{\Gamma}'_{1}}) \times
       \bar{S}^{+}((0,+\infty)^{\vec{\Gamma}_{1}- \vec{\Gamma}'_{1}}). \]
\end{prop}

\begin{proof}
  This follows from the construction of $\widetilde{[0,+\infty)}^{\vec{\Gamma}_{1}}$.
\end{proof}
\section{The coefficient calculation}
\label{app:ViraCoefficient}
\label{app:CentralCoefficientCalculation}

We record the finite coefficient calculation used in Proposition~\ref{PureVira}
and Theorem~\ref{FullCentral}. The Wick-theoretic reduction has already been
explained in the proofs of Propositions~\ref{PureAffine} and~\ref{PureVira};
here we only extract the scalar one-loop coefficients from
\eqref{OneloopConst}--\eqref{Oneloop3order}.

Throughout this appendix we work in dimension \(2\). For
\[
  T_i=\sum_{\sss=1}^{2}T_i^{\sss}(z)\partial_{z^{\scriptstyle\sss}},
\]
set
\[
  D_i\define\mathrm{div}(T_i)
  =
  \sum_{\sss=1}^{2}
  \partial_{z^{\scriptstyle\sss}}T_i^{\sss},
  \qquad
  \eta_i\define\partial D_i .
\]
We also use
\[
  \langle \partial T_i,\partial T_j\rangle
  \define
  \sum_{\rrr,\sss=1}^{2}
  \partial(\partial_{z^{\scriptstyle\rrr}}T_i^{\sss})
  \wedge
  \partial(\partial_{z^{\scriptstyle\sss}}T_j^{\rrr}).
\]
Thus \(\eta_i\) is a one-form and
\(\langle \partial T_i,\partial T_j\rangle\) is a two-form.

We use the following diagonal \(\mathcal D\)-module convention. A monomial in
\(\lambda_1,\lambda_2\) multiplying \(d\mathbf z_o^{\blacklozenge}\) in
\eqref{OneloopConst}--\eqref{Oneloop3order} acts on the coefficient functions by
\[
  \lambda_i^{\sss}\rightsquigarrow
  \partial_{z_i^{\scriptstyle\sss}},
  \qquad i=1,2,\quad \sss=1,2.
\]
After applying these derivatives, we restrict to the diagonal
\(z_o=z_1=z_2=z\) and apply \(\mathrm{Res}\). We suppress the common shifted
top form \(d\mathbf z_o^{\blacklozenge}\). Since \(\mathrm{Res}\) kills total
derivatives, we use
\[
  \mathrm{Res}\bigl(\partial_{z^{\scriptstyle\sss}}H\bigr)=0
\]
to move derivatives between factors whenever needed.

For each input type, we sum the two oriented cyclic contractions
\[
  (o\to 1\to 2\to o),
  \qquad
  (o\to 2\to 1\to o),
\]
and then rewrite the result as an alternating cochain on the ordered exterior
input.

For three affine inputs \(a_i=x_i\otimes f_i\), the zero-derivative coefficient
\eqref{OneloopConst} gives
\[
  C_{AAA}(a_0\wedge a_1\wedge a_2)
  =
  \mathrm{ch}^{V}_{3}(a_0\wedge a_1\wedge a_2),
\]
as in Proposition~\ref{PureAffine}.

For one vector field and two affine inputs, the one-derivative coefficient
\eqref{Oneloop1order} gives
\[
\begin{aligned}
&C_{TAA}
\left(
T\wedge (x_0\otimes f_0)\wedge (x_1\otimes f_1)
\right)                                                     \\
&\qquad =
\frac12\,\mathrm{tr}_V(x_0x_1)\,
\mathrm{Res}
\left(
\mathrm{div}(T)\,
\partial f_0\wedge\partial f_1
\right)                                                     \\
&\qquad =
\frac12\,
\mathrm{ch}^{\mathcal T}_{1}\mathrm{ch}^{V}_{2}
\left(
T\wedge (x_0\otimes f_0)\wedge (x_1\otimes f_1)
\right).
\end{aligned}
\]

For one affine input and two vector fields, the two-derivative coefficient
\eqref{Oneloop2order} gives
\[
\begin{aligned}
&C_{ATT}
\left(
(x\otimes f)\wedge T_1\wedge T_2
\right)                                                     \\
&\qquad =
\mathrm{tr}_V(x)\,
\mathrm{Res}
\left(
\frac14\,f\,\eta_1\wedge\eta_2
-
\frac1{12}\,f\,
\langle \partial T_1,\partial T_2\rangle
\right)                                                     \\
&\qquad =
\left(
\frac18\,\mathrm{ch}^{V}_{1}(\mathrm{ch}^{\mathcal T}_{1})^{2}
-
\frac1{12}\,\mathrm{ch}^{V}_{1}\mathrm{ch}^{\mathcal T}_{2}
\right)
\left(
(x\otimes f)\wedge T_1\wedge T_2
\right).
\end{aligned}
\]
Here we use the polarized convention
\[
\mathrm{ch}^{V}_{1}(\mathrm{ch}^{\mathcal T}_{1})^{2}
\left((x\otimes f)\wedge T_1\wedge T_2\right)
=
2\,\mathrm{tr}_V(x)\,
\mathrm{Res}\left(f\,\eta_1\wedge\eta_2\right).
\]

Finally, for three vector fields, the three-derivative coefficient
\eqref{Oneloop3order} gives
\[
\begin{aligned}
C_{TTT}(T_0\wedge T_1\wedge T_2)
&=
r\left[
\frac1{24}
\sum_{\mathrm{Cyc}}
\mathrm{Res}
\left(
D_i\,\eta_j\wedge\eta_k
\right)
-
\frac1{24}
\sum_{\mathrm{Cyc}}
\mathrm{Res}
\left(
D_i\,
\langle \partial T_j,\partial T_k\rangle
\right)
\right]                                                    \\
&=
r\left(
\frac1{48}(\mathrm{ch}^{\mathcal T}_{1})^{3}
-
\frac1{24}\mathrm{ch}^{\mathcal T}_{1}\mathrm{ch}^{\mathcal T}_{2}
\right)
(T_0\wedge T_1\wedge T_2).
\end{aligned}
\]
Here
\[
  \sum_{\mathrm{Cyc}}
\]
means the cyclic sum over
\[
  (i,j,k)=(0,1,2),(1,2,0),(2,0,1),
\]
and \(r=\dim V\) is the trace over the ghost index.

Combining the four input types gives the arity-three central cochain
\[
\begin{aligned}
C^{\mathrm{vir}\ltimes\mathrm{aff}}_{3}
&=
\mathrm{ch}^{V}_{3}
+
\frac12\,\mathrm{ch}^{\mathcal T}_{1}\mathrm{ch}^{V}_{2}      \\
&\quad+
\left(
\frac18\,\mathrm{ch}^{V}_{1}(\mathrm{ch}^{\mathcal T}_{1})^{2}
-
\frac1{12}\,\mathrm{ch}^{V}_{1}\mathrm{ch}^{\mathcal T}_{2}
\right)                                                     \\
&\quad+
r\left(
\frac1{48}(\mathrm{ch}^{\mathcal T}_{1})^{3}
-
\frac1{24}\mathrm{ch}^{\mathcal T}_{1}\mathrm{ch}^{\mathcal T}_{2}
\right).
\end{aligned}
\]
This is the coefficient identity used in Theorem~\ref{FullCentral}.

\printbibliography

\end{document}